\documentclass[12pt]{amsproc}

\usepackage{graphicx,psfrag,epsfig,multirow}
\usepackage{amsmath, amssymb}
\usepackage[usenames,dvipsnames]{color}

\def\mes{{\rm mes}}
\def\Vol{{\mathrm{Vol}}}
\usepackage{color}
\definecolor{orange}{rgb}{1,0.5,0}

\def\eps{{\varepsilon}}
\def\bc{{\mathbf{c}}}
\def\bG{{\mathbf{G}}}
\def\bH{{\mathbf{H}}}
\def\bGamma{{\boldsymbol{\Gamma}}}
\def\bSigma{{\boldsymbol{\Sigma}}}
\def\e{{\varepsilon}}
\def\rationals{{\mathbb{Q}}}
\def\reals{{\mathbb{R}}}
\def\integers{{\mathbb{Z}}}
\def\C{{\mathbb{C}}}

\def\bD{{\mathbf{D}}}

\def\bmu{{\boldsymbol{\mu}}}
\def\blambda{{\boldsymbol{\lambda}}}

\def\cA{{\mathcal{A}}}
\def\cN{{\mathcal{N}}}

\def\bz{\bar z}
\def\T{\mathbb T}
\def\N{\mathbb N}
\def\R{\mathbb R}
\def\Z{\mathbb Z}
\def\Q{\mathbb Q}

\def\brk{\bar k}
\def\brp{\bar p}
\def\brq{\bar q}
\def\brx{\bar x}
\def\brDelta{{\bar \Delta}}

\def\cD{{\mathcal{D}}}
\def\cC{{\mathcal{C}}}

\def\cH{{\mathcal{H}}}
\def\cL{{\mathcal{L}}}
\def\cM{{\mathcal{M}}}
\def\cP{{\mathcal{P}}}
\def\cT{{\mathcal{T}}}
\def\cW{{\mathcal{W}}}

\def\hl{{\hat l}}

\def\tA{{\tilde{A}}}
\def\tc{{\tilde{c}}}

\def\talpha{{\tilde{\alpha}}}
\def\tGamma{{\tilde{\Gamma}}}

\def\a{\alpha}
\def\Tor{{\mathbb{T}}}
\def\Card{{\mathrm{Card}}}

\def\EXP{{\mathbb{E}}}
\def\Prob{{\mathbb{P}}}

\def\length{{\text{length}}}
\def\Leb{{\mathrm{Leb}}}
\def\Pois{{\mathfrak{P}}}

\def\LL{{\mathbb{L}}}

\def\Cauchy{{\mathfrak{C}}}
\def\sN{{\mathfrak{N}}}

\newtheorem{theorem}{Theorem}

\newtheorem{proposition}{Proposition}
\newtheorem{corollary}[theorem]{Corollary}
\newtheorem{question}{Question}
\theoremstyle{definition}

\newtheorem*{remark}{Remark}

\begin{document}

\title{Limit theorems for toral translations}
\author{Dmitry Dolgopyat}
\author{Bassam Fayad}

\thanks{The authors thank Mariusz Lema\'nczyk, Jens Marklof, Jean Paul Thouvenot, and
Ilya Vinogradov for useful discussions. The research of DD was partially supported by the NSF.
The research of BF was partially supported by ANR-11-BS01-0004.}

\address{Dmitry Dolgopyat Department of Mathematics University of Maryland College Park MD 20814 USA}
\email{dmitry@math.umd.edu}
\address{Bassam Fayad IMJ-PRG, CNRS, Paris, France}
\email{bassam.fayad@imj-prg.fr}

\maketitle
\tableofcontents 

\addtocontents{toc}{\protect\setcounter{tocdepth}{2}}
\section{Introduction}
One of the surprising discoveries of dynamical systems theory is that many deterministic systems with
non-zero Lyapunov exponents satisfy the same limit theorems as the sums of independent random variables.
Much less is known for the zero exponent case where only a few examples have been analyzed.
In this survey we consider the extreme case of toral translations where each map 
not only has zero exponents but is actually an isometry. 
These systems were studied extensively due to their relations to number theory, to the theory of integrable systems 
and to geometry. Surprisingly  many natural questions are still open. We review known results as well as the methods
to obtain them and present a list of open problems. Given a vast amount of work on this subject, it is impossible to provide
a comprehensive treatment in this short survey. Therefore we treat the topics closest to our research interests in more
detail while some other subjects are mentioned only briefly. Still we hope to provide the reader with the flavor of the subject
and introduce some tools useful in the study of toral translations, most notably, various renormalization techniques.

Let $X=\Tor^d,$ $\mu$ be the Haar measure on $X$ and $T_{\alpha}(x)=x+\alpha.$

The most basic question in smooth ergodic theory is the behavior of ergodic sums. Given a map $T$ and a zero mean  observable 
$A(\cdot)$ let 
\begin{equation}
\label{ErgSum}
A_N(x)=\sum_{n=0}^{N-1} A(T^n x)    
\end{equation}
If there is no ambiguity, we may write $A_N$ for $A_N(x)$. Conversely we may use the notation $A_N(\a,x)$ to indicate that the underlying map is the translation of vector $\a$. 
The uniform distribution of the orbit of $x$ by $T$ is characterized by the convergence to $0$ of $A_N(x)/N$. 
In the case of toral translations $T_\a$ with irrational frequency vector $\alpha$ 
the uniform distribution holds for all points $x.$ The study of the ergodic sums is then useful to quantify the rate of uniform distribution of the Kronecker sequence $n \a$ mod 1 as we will see in Section \ref{sec.discrepancy} where discrepancy functions are discussed. 
The question about the distribution of ergodic sums is analogous to the the Central Limit Theorem in probability theory.
One can also consider analogues of other classical probabilistic results. In this survey we treat two such questions.
In Section \ref{sec.poisson} we consider so called Poisson regime where \eqref{ErgSum} is replaced
by $\sum_{n=0}^{N-1} \chi_{\cC_N} (T_\a^n x)$ and the sets $\cC_N$ are scaled  in such a way that only finite number of terms
are non-zero for typical $x.$  Such sums appear in several questions in mathematical physics, including 
quantum chaos \cite{M-LH} and Boltzmann-Grad limit of several mechanical systems \cite{M-Bern}. They also describe the
resonances in the study of ergodic sums for toral translations as we will see in Section \ref{ScProofs}.  In Section \ref{sec.shrinking} 
we consider Borel-Cantelli type questions where one takes a sequence of shrinking sets and studies a number of times
a typical orbit hits whose sets. These questions are intimately related to some classical problems in the theory of Diophantine
approximations.

%If instead of $A$ we use a sequence $A^n=\chi_{C_n}-\mu(C_n)$ where  $\chi_{C_n}$ is the characteristic function of some (for example shrinking) sets $C_n$ the study of the ergodic sums tells us about the statistics of the number of entrance of a translation  orbit inside these moving (shrinking) targets $C_n$.  The latter will be discussed in Section \ref{sec.shrinking}. If now we consider in $(\star)$  an observable $A=A^N=\chi_{C_N}$ with $\mu(C_N)$ of order $1/N$ we expect only a finite number of the terms in the sum to be non zero which corresponds to a statistical behavior of the sums $A_N^N$ that is coined the Poisson regime  in relation to the case of $iid$ variables. The Poisson regime and related problems will be discussed in Section 

The ergodic sums above toral translations also appear in natural dynamical systems such as skew products, cylindrical cascades and special flows. Discrete time systems related to ergodic sums over translations are treated in 
Section~\ref{sec.skew} while flows are treated in Section \ref{ScFlows}. These systems give additional motivation to study the ergodic sums \eqref{ErgSum}
for smooth functions having singularities of various types: power, fractional power, logarithmic... 
Ergodic sums for functions with singularities are discussed 
in Section \ref{ScErgSum}.  Finally in Section \ref{ScHD} we present the results related to action of several translations at the same time.

\medskip 

\noindent{Ê\bf Notations.}

\medskip

We say that a vector $\a=(\a_1,\ldots, \a_d) \in \R^d$ is irrational if  $\{1,\a_1,\ldots,\a_d\}$ are linearly independent over $\Q$.

For $x \in \R^d$, we use the notation $\{x\}:=(x_1,\ldots,x_d) \text{ mod }(1)$.
We denote by
$ \| x \|$ the closest signed distance of some $x \in \R$ to the  integers.

Assuming that  $d \in \N$ is fixed, for  $\sigma >0$ we denote by $\cD(\sigma) \subset \R^d$ the set of Diophantine vectors with exponent $\sigma$, that is

\begin{equation}
\label{DefDio}
\cD(\sigma)=\{\alpha: \exists C\;\; \forall k\in \integers^d-0, m\in \integers \quad |(k,\alpha)-m|\geq C |k|^{-d-\sigma} \}
\end{equation}

Let us recall that 
$\cD(\sigma)$ has a full measure if $\sigma>0$, while $\cD(0)$ is an uncountable set of zero measure and $\cD(\sigma)$ is empty for $\sigma < 0$.  
The set $\cD(0)$ is called the set of constant type vector or badly approximable vectors. 
An irrational vector $\a$  that is not Diophantine for any $\sigma>0$ is called Liouville. 

We denote by $\Cauchy$ the standard Cauchy random variable with density $\frac{1}{\pi(1+x^2)}.$
The normal random variable with zero mean and variance $D^2$ will be denoted by
$\sN(D^2).$ Thus $\sN(D^2)$ has density $\frac{1}{2\pi D} e^{-x^2/2D^2}.$ We will write simply
$\sN$ for $\sN(1).$ Next, $\Pois(X, \mu)$ 
will denote the Poisson process on $X$ with measure $\mu$ (we refer the
reader to Section \ref{SSPois} for the definition and basic properties of Poisson processes). 

\section{Ergodic sums of smooth functions with singularities} 
\label{ScErgSum}

\subsection{Smooth observables} 
For toral translations,
the ergodic sums  of smooth observables are well understood. 
Namely if $A$ is sufficiently smooth with zero mean then for almost all $\alpha,$ $A$ is a coboundary, that is,
there exists $B(\a,x)$ such that
\begin{equation}
\label{CoB}
A(x)=B(x+\alpha, \alpha)-B(x,\alpha).
\end{equation}
Namely if $A(x)=\sum_{k\neq 0} a_k e^{2\pi i (k, x)}$ then 
$$B(\a,x)=\sum_{k\neq 0} b_k e^{2\pi i(k,x)} \text{ where }b_k=\frac{a_k}{e^{i2\pi(k,\alpha)}-1}.$$ 
The above series converges in $L^2$ provided 
$\alpha\in \cD(\sigma)$
and 
$A\in \cH^\sigma=\{A: \sum_k |a_k |k|^{(\sigma+d)}|^2<\infty\}.$ 
Note that \eqref{CoB} implies that
$$A_N(x,\a)=B(x+N\alpha, \alpha)-B(x,\alpha)$$
giving a complete description of the behavior of ergodic sums for almost all $\alpha.$ 
In particular we have
\begin{corollary} \label{CorErgSmooth}
If 
$\alpha$ is uniformly distributed 
on $\Tor^{d}$ then $A_N(x)$ has a limiting distribution as $N\to \infty,$
namely 
$$A_N\Rightarrow B(y,\alpha)-B(x,\alpha)$$ 
where $(y, \alpha)$ is uniformly distributed on $\Tor^d\times \Tor^d.$
\end{corollary}
\begin{proof}
We need to show that as $N\to\infty$ the random vector  $(\alpha, N\alpha)$ converge to a vector with coordinates  independent random variables
uniformly distributed on $\Tor^d\times \Tor^d.$ To this end it suffices to check that if $\phi(x,y)$ is a smooth function on
$\Tor^d\times \Tor^d$ then 
$$ \lim_{N\to\infty} \int_{\Tor^d} \phi(\alpha, N\alpha) d\alpha=\int_{\Tor^d \times \Tor^d}
\phi(\alpha, \beta) d\alpha d\beta $$
but this is easily established by considering the Fourier series of $\phi.$
\end{proof}

We will see in Section \ref{ScFlows} how our understanding of ergodic sums for smooth functions can be used
to derive ergodic properties of area preserving flows on $\T^2$ without fixed points.

%A class of  examples of special flows  above circle rotations and under smooth ceiling functions is given by conservative flows  (that preserve a measure given by a smooth density)  on surfaces that are free of fixed points. The mixing properties of these flows are given by the statistical distribution of the ergodic sums of the ceiling function above the base rotation. The main facts are absence of mixing always, linearizability in case the base frequency is Diophantine and genericity of weak mixing in case $\a$ is Liouville. We will see in Section \ref{sec.special} how all these facts can be easily obtained from our understanding of the ergodic sums of smooth functions above rotations. 

On the other hand there are many open questions related to the case when the observable $A$ is not smooth enough
for \eqref{CoB}  to hold. Below we mention several classes of interesting observables.

\subsection{Observables with singularities}

Special flows above circle rotations and under ceiling functions that are smooth except for some singularities naturally appear in the study of conservative flows 
   on surfaces with fixed points.

    Another motivation for studying ergodic sums for functions with  singularities is the case of
meromorphic functions, whose sums appear in questions related to both number theory \cite{HL} and ergodic theory
\cite{Pe}.

\subsubsection{Observables with logarithmic  singularities.}

In the study of conservative flows 
   on surfaces, non degenerate saddle singularities are responsible for logarithmic singularities of the ceiling function. 

Ceiling functions with  logarithmic singularities also appear  in the study of multi-valued Hamiltonians on the two torus.   In \cite{Arnold}, Arnold investigated  such flows 
and showed that the torus decomposes into cells that are filled up by periodic orbits and
one open ergodic component. On this component, the flow can be represented
as a special flow over an interval exchange map  of the circle and under a ceiling
function that is smooth except for some logarithmic singularities.
The singularities can be asymmetric since the coefficient in front of the
logarithm is twice as big on one side of the singularity as the one on the other side,
due to the existence of homoclinic loops (see Figure  \ref{FigSaddle}). 

\begin{figure}[htb]
    \centering
    \resizebox{!}{6cm}{\includegraphics{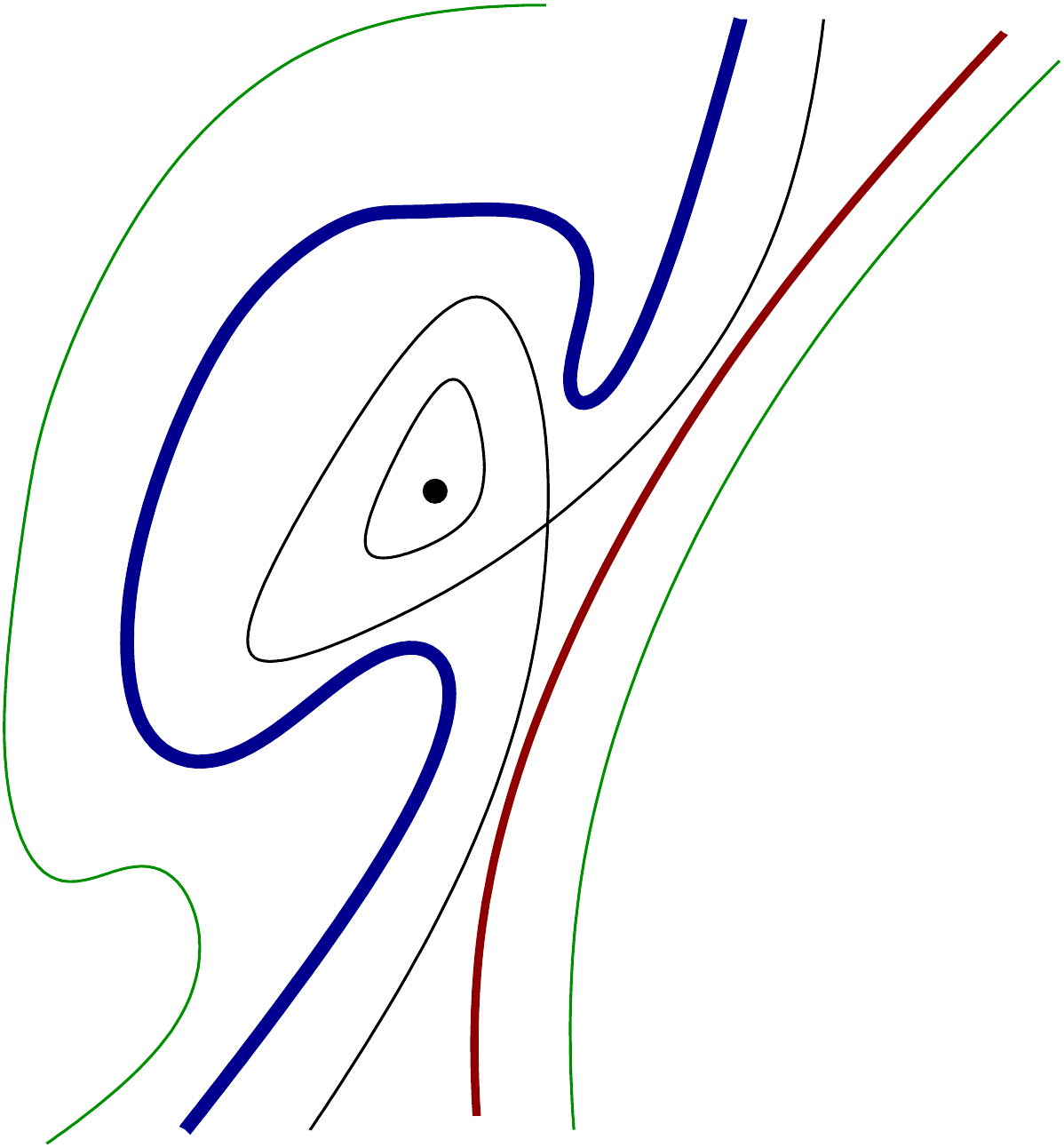}}
    \caption{Multivalued Hamiltonian flow. Note that the orbits passing to the left of the saddle 
spend approximately twice longer time comparing 
to the orbits passing to the right of the saddle and starting at the same distance from the separatrix since they pass near the saddle twice.}
    \label{FigSaddle}
\end{figure}

%{\color{Orange} More explanations. }
%   The mixing properties of these flows were studied in several papers \cite{Kocergyn76,KS,Kocergyn03,Kocergyn04,Ulicgrai1,Ulcigrai2,Lem1,Lem2}.  Ergodicity of cylindrical cascades with function with logarithmic singularities were studied in \cite{FL}. 
More motivations for studying function with logarithmic singularities as well as some numerical results for rotation
numbers of bounded type are presented in \cite{KT}.

A natural question is to understand the fluctuations of the ergodic sums for these functions as the frequency $\a$ of the underlying rotation is random as well as the base point $x$.
Since Fourier coefficients of the symmetric logarithm function have the asymptotics similar to that of the indicator function of an interval one
may expect that the results about the latter that we will discuss in Section \ref{sec.discrepancy} can be extended to the former.

\begin{question} \label{q.log}
Suppose that $A$ is smooth away from a finite set of points $x_1, x_2\dots x_k$ and near $x_j,$ 
$A(x)=a_j^\pm \ln |x-x_j|+r_j(x)$ where $-$ sign is taken if $x<x_j,$ $+$ sign is taken if $x>x_j$
and $r_j$ are smooth functions. 
What can be said about the distribution of $A_N(\a,x)/\ln N$ as $x$ and $\a$ are random? 
%What happens in the non symmetric case?
\end{question}

%The idea here is to treat the ergodic sums of a symmetric logarithm function in a similar way as indicator functions of intervals due to the resemblance between their Fourier coefficients. 
%We thus expect similar results as  Kesten's Cauchy limit law for the normalized ergodic sums of interval characteristic functions (see Section \ref{sec.discrepancy}). 

%{\color{red} In the symmetric case $a_j^+=a_j^-$ we have that the Fourier coefficients of $A$ satisfy $\widehat{A}_n = \frac{c_n}{n}+0(\frac{1}{n^2})$, with $c_n=C \sum_j \cos (2 \pi x_j)$  hence the same proof as in Kesten's Theorem (see Section \ref{sec.discrepancy}) yields that $A_N/ c(x_1,\ldots,x_k) \ln (N) $ converges to a Cauchy distribution (with $c(x_1,\ldots,x_k)$ independent of $x_1,\ldots,x_k$ if they are linearly independent over $\Q$).  What happens in the non symmetric case where we have  $\widehat{A}_n = 0(\frac{1}{n \ln n})$?}

\subsubsection{Observables with power like singularities.}

When considering conservative flows on surfaces with degenerate saddles one  is led to study the ergodic sums of observables
with integrable power like singularities (more discussion of these flows will be given in Section \ref{ScFlows}). Special flows above irrational rotations of the circle under such ceiling functions are called Kocergin flows. 

The  study of ergodic sums for smooth ergodic flows with nondegenerate hyperbolic singular points on surfaces of genus $p\geq 2$ shows that these flows are in general not mixing (see Section \ref{ScFlows}). {\it A contrario} Kocergin showed that special flows above irrational rotations and under 
ceiling functions with integrable power like singularities are always mixing. This is due to the important deceleration next to the singularity that is responsible for a {\it shear} along orbits that separates the points sufficiently to produce mixing. In other words, the mixing is due to large oscillations of the ergodic sums. In this note we will be frequently interested in the distribution properties of these sums. 

%\begin{figure}[htb]
%    \centering
%    \resizebox{!}{6cm}{\includegraphics{FlSurf.eps}}
%    \caption{Kocergin Flow. }
%    \label{FigKocergin}
%\end{figure}

One may also consider the case of non-integrable power singularities since they naturally appear in problems of ergodic theory and number theory.
The following result answers a question of \cite{HL}. 
\begin{theorem} \label{SU}
(\cite{SU}) If $A$ has one simple pole on $\Tor^1$ and 
$(\a,x)$ is uniformly distributed 
on $\Tor^2$ then $\frac{A_N}{N}$ has a limiting distribution as $N\to \infty.$
\end{theorem}

The function $A$ in Theorem \ref{SU} has a symmetric singularity of the form $1/x$ that is the source of cancellations in the ergodic sums. 

\begin{question} 
\label{Q1xasym}
What happens for an asymmetric singularity of the type $1/|x|$? 
\end{question}

\begin{question}
What happens in the quenched setting where $\alpha$ is fixed?
\end{question}

We now present several generalizations of Theorem \ref{SU}. 
%The first one (Theorem \ref{ThPole2})
%deals with non-integrable power singularities, while the second addresses the case of integrable power-singularities.  

\begin{theorem}
\label{ThPole2}
Let $A=\tilde{A}(x)+\frac{c_- \chi_{x<x_0}+c^+ \chi_{x>x_0}}{|x-x_0|^{a}}$ where $\tilde A$ is smooth and $a>1.$

(a) If $(\a,x)$ is uniformly distributed 
on $\Tor^2$ then $\dfrac{A_N}{N^a}$ converges in distribution. 

(b) For almost every $x$ fixed,  if $\a$ is uniformly distributed on $\Tor$ then $\dfrac{A_N(\a,x)}{N^{a}}$  
converge to the same limit as in part (a).
\end{theorem} 

\begin{theorem} \cite{LL}  \label{q.a}
If $A$ has zero mean and is smooth except for a singularity at $0$ of type $|x|^{-a}$, $a\in (0,1)$ then $A_N/N^a$ converges in distribution. 
\end{theorem} 

The proof of Theorem \ref{q.a} is inspired by the proof of Theorem \ref{ThConvex} of Section \ref{sec.discrepancy} which will 
be presented in Section \ref{SSProofsD}. %It starts with the determination of the Fourier modes $k$ in the expansion of $A_N/N$ that contribute significantly to the sum. The latter satisfy resonance conditions with $\a$ similar to the ones described in \eqref{ConvSD1}-\eqref{ConvSD2}.   The Dani correspondence principle (\cite{Da}, see Section \ref{ScProofs} below) then allows to translate these resonance conditions into geometric conditions on a lattice in $\R^2$ determined by $N$ and $\a$. 

Marklof proved in \cite{M2} that if $\a\in \cD(\sigma)$ 
with $\sigma<(1-a)/a$, then for $A$ as in Theorem \ref{q.a} $A_N(\a, \a)/N \to 0$. 
\begin{question}
\label{QM2-NI}
What happens for other angles $\a$ and other type of singularities, 
including the non integrable ones for which the ergodic theorem does not necessarily hold?
\end{question}

Another natural generalization of Theorem \ref{SU} is to consider meromorphic functions. Let 
$A$ be such a function with highest pole of order $m.$ Thus $A$ can be written as
$$ A(x)=\sum_{j=1}^r \frac{c_j}{(x-x_j)^m}+\tA(x) $$
where the highest pole of $\tA$ has order at most $m-1.$

\begin{theorem}
\label{ThMer} 
(a) Let $A$ be fixed and let $\alpha$ be distributed according to a smooth density on $\T.$
Then for any $x \in \T$, $\dfrac{A_N(\a, x)}{N^{m}}$ has a limiting distribution as $N\to \infty.$

(b) Let $\tA, c_1, \dots c_r$ be fixed while  $(\a,x, x_1\dots x_r)$ are distributed according
to a smooth density on 
on $\Tor^{r+2} $ then $\dfrac{A_N(\a, x)}{N^{m}}$ has a limiting distribution as $N\to \infty.$ 

(c) If $(x_1, x_2 \dots x_r)$ is a fixed irrational vector then for almost every $x \in \T$ the limit distribution in part (a) is the same as the limit distribution in part (b).
\end{theorem}

Proofs of Theorems \ref{ThPole2} 
and \ref{ThMer} are sketched in Section \ref{ScProofs}. 

It will be apparent from the proof of Theorem \ref{ThMer} that the limit distribution in part (a) is not the same
{\bf for all} $x_1, x_2\dots x_j.$ For example if $x_j=jx_1$ we get an exceptional distribution since a close 
approach to $x_1$ and $x_2$ by the orbit of $x$ should be followed by a close approach to $x_j$
for $j\geq 3.$ We will see that this phenomenon appears in many limit theorems (see e.g 
Theorem \ref{ThKesten}, Theorem \ref{CLTShrinkBad} and Question \ref{QSTAE},
Theorem \ref{ThChi2} and Question \ref{RWTypAlpha},  
as well as \cite{M-Bern}).   

\begin{question} \label{q6}
What can be said about more general meromorphic functions such as $\sin 2\pi x/(\sin 2\pi x+3\cos 2\pi y)$
on $\T^d$ with $d>1?$
\end{question}

\section{Ergodic sums of characteristic functions.  Discrepancies} 
\label{sec.discrepancy} 

 The case where $A=\chi_\Omega$ is a classical subject in number theory.
Define the discrepancy function 
$$D_N(\Omega, \a,x)=\sum_{n=0}^{N-1} \chi_\Omega(x+n\alpha)-N\frac{\text{Vol}(\Omega)}{\text{Vol} (\Tor^d)}. $$

Uniform distribution of the sequence $x+k \a$ on $\T^d$ is equivalent to the fact that, 
for regular sets $\Omega,$ 
$D_N(\Omega,\a,x)/N \to 0$ as $N \to \infty$. A step further in the description of the uniform distribution is the study of the 
rate of convergence to $0$ of $D_N(\Omega,\a,x)/N$.

In $d=1$ it is known that if $\a \in \T-\Q$ is fixed, the discrepancy $D_N(\Omega,\a,x)/N$ displays an oscillatory 
behavior according to the position of $N$ with respect to the denominators of the best rational approximations of  $\a$. 
A great deal of work in Diophantine approximation has been done on estimating the discrepancy function in relation with 
the arithmetic properties of $\a \in \T$, and more generally for $\a \in \T^d$. 

%It is of  common knowledge that  in  studying  the discrepancies in dimension $1$    the continued fraction algorithm provides crucial help, and that the absence of an analogue in higher dimensions makes the study of discrepancies much harder. 

\subsection{The maximal discrepancy} 
\label{sec.beck}

Let
\begin{equation}
\label{DefMaxD}
\overline{D}_N(\a)=\sup_{\Omega \in \mathbb{B}} D_N(\Omega, \alpha, 0)
\end{equation}
where the supremum is taken over all sets $\Omega$ in some natural class of sets $\mathbb{B}$, for example balls or boxes (product of intervals).

The case of (straight) boxes was extensively studied, and  growth properties of the sequence $\overline{D}_N(\a)$ were obtained  with a special emphasis on their relations with the Diophantine approximation properties of $\a.$ 
In particular, following earlier advances of \cite{Kok, HL, O, Kh, schmidt} and others, \cite{beck1} proves

\begin{theorem} \label{th.beck}
Let 
$$\overline{D}_N(\a)=\sup_{\Omega -\text{box}} D_N(\Omega, \alpha, 0)$$
Then for any positive increasing function $\phi$ we have 
\begin{equation} \hspace{0.4cm} \sum_{n} \phi(n)^{-1}<\infty  \iff \frac{\overline{D}_N(\a)}{(\ln N)^d \phi(\ln\ln N)}   \begin{array}{l}
\text{ is  bounded for}\\
\text{ almost every  }\a \in \T^d.
\end{array} 
   \label{beckbound}  \end{equation}

\end{theorem}

In dimension $d=1$, this result 
is the content of Khinchine theorems obtained in the early 1920's \cite{Kh},
and it follows easily from well-known  results from 
the metrical theory of continued fractions (see for example the introduction of \cite{beck1}).
The higher dimensional case is significantly more difficult and the cited bound was only obtained in the 1990s.
% as opposed to the Khinchine findings dating back to the 1920s. Moreover,  it is still impossible to establish a dichotomy between the lower and upper bounds for the $\limsup$ in the higher dimensional case.    

 The bound in (\ref{beckbound}) 
 focuses on how bad can the discrepancy become along a subsequence of $N$, for a fixed $\a$ in a full measure set.  In a sense, it  deals with the worst case scenario and does not capture the oscillations of the discrepancy.  On the other hand, the restriction on $\a$ is necessary, since given any $\eps_n \to 0$ it is easy to see that for $\a \in \T$ {Ê\it sufficiently   Liouville}, the discrepancy (relative to intervals) can be as bad as $N_n \eps_n$ along a suitable sequence $N_n$ (large multiples of Liouville denominators). %\margem{made some modifications here}

For $d=1$, it is not hard to see, using  continued  fractions, that for any $\a$ : $\limsup\frac{\overline{D}_N(\alpha)}{\ln N}>0$, 
$\liminf\overline{D}_N(\alpha)\leq C$; and for $\a \in \mathcal D(0)$  $\limsup\frac{\overline{D}_N(\alpha)}{\ln N}<+\infty$.  The study of higher dimensional counterparts to these results raises several interesting questions.

\begin{question} \label{q7}
Is it true that $\limsup\frac{\overline{D}_N(\alpha)}{\ln^d N}>0$ {\bf for all} $\alpha\in \Tor^d$?
\end{question}

%{\color{red} Should we put $=\infty$ instead of $>0$?I replaced $\liminf$ by $\limsup$ otherwise it is false for $d=1$. It is not clear in Beck on page 11 of his paper when he speaks about a  lower bound for all $\a$ on the discrepancy whether he means $\liminf$ or $\limsup$ but I think it is $\limsup$. Have to check his proof on page 11-12 of a lower bound (it must be only a proof  for a.e. $\a$).}Currently the best known result is the the general lower bound $(\ln N)^{d/2}$ that holds for every sequence on $\T^d$ (\cite{Roth}). On the other hand, due to the use of continued fractions the latter conjecture can be easily verified in dimension $1$ (cf. discussion in \cite{beck1}). 

\begin{question}
Is it true that there exists $\a$ such that $\limsup\frac{\overline{D}_N(\alpha)}{\ln^d N}<+\infty$?
\end{question}

\begin{question} \label{q9}
What can one say about $\liminf\frac{\overline{D}_N(\alpha)}{a_N}$ for a.e. $\a$, where $a_N$ is an adequately chosen normalization? for every $\a$? %in distribution? 
\end{question}

\begin{question}
 Same questions as Questions \ref{q7}--\ref{q9}  when boxes are replaced by balls.
\end{question}

\begin{question} \label{q11}  Same questions as Questions \ref{q7}--\ref{q9}  for the {\it isotropic discrepancy,} when boxes are replaced by the class of 
 all convex sets \cite{Nie}. \end{question}

\subsection{Limit laws for the discrepancy as $\a$ is random}

 In this survey, we will mostly concentrate on the distribution of the discrepancy function as $\a$ is random. The above discussion naturally raises the following question. 

\begin{question}
Let $\a$ be uniformly distributed on $\T^d.$
Is it true that $\frac{\overline{D}_N(\alpha)}{\ln^d N}$ converges in distribution as $N\to\infty$?
\end{question}

Why do we need to take $\alpha$ random? The answer is that for fixed $\alpha$ the discrepancy does not have a limit distribution, no water which normalization is chosen. 

For example for $d=1$ the Denjoy-Koksma inequality says that
$$ |A_{q_n}-q_n\int A(x) dx|\leq 2 V  $$
where $q_n$ is the $n$-th partial convergent to $\alpha$ and $V$ denotes the total variation of $A.$ 
In particular $D_{q_n}(I, \a,x)$ can take at most 3 values.

In higher dimensions one can show that if $\Omega$ is either a box or any other strictly convex set
then for almost all $\alpha$ and almost all tori, when $x$ is random the variable
$$ \frac{D_N(\Omega, \reals^d/L, \alpha, \cdot)}{a_N} $$ 
does not converge to a non-trivial
limiting distribution for any choice of $a_N=a_N(\alpha, L)$
(see discussion in the introduction of \cite{DF2}).

\begin{question}
Is this true {\bf for all} $\alpha, L$? 
\end{question}

\begin{question}
Study the distributions which can appear as weak limits of $\dfrac{D_N(\Omega, \alpha, \cdot)}{a_N}$, in particular their relation with
number theoretic properties of $\alpha.$
\end{question} \vskip2mm

Let us consider the case $d=1$ (so the sets of interest are intervals and we will write $I$ instead of $\Omega$.) It is easy to see that all limit distributions are atomic for all $I$ iff $\alpha\in \rationals.$

\begin{question}
Is it true that all limit distributions are either atomic or Gaussian for almost all $I$ iff $\alpha$ is of bounded type?
\end{question}

Evidence for the affirmative answer is contained in the following results.

\begin{theorem}
(\cite{H}) If $\alpha\not\in \rationals$ and $I=[0, 1/2]$ then there is a sequence $N_j$ such that 
$\dfrac{D_{N_j}(I, \alpha, \cdot)}{j}$ converges to $\sN.$
\end{theorem}

Instead of considering subsequences, it is possible to  randomize $N.$

\begin{theorem} 
\label{ThNRandom}
Let $\alpha$ be a quadratic surd.

(a) (\cite{beck2}) If $(x, a, l)$ is uniformly distributed on $\Tor^3$ then
$\dfrac{D_{[aN]} ([0,l], \alpha, x)}{\sqrt{\ln N}}$ converges to $\sN(\sigma^2)$
for some $\sigma^2\neq 0.$

(b) (\cite{beck3}) If $M$ is uniformly distributed on $[1, N]$ and $l$ is rational then there are constants
$C(\alpha, l), \sigma(\alpha, l)$ such that
$\dfrac{D_{M} ([0,l], \alpha, 0)-C(\alpha,l)\ln N}{\sqrt{\ln N}}$ converges to $\sN (\sigma^2(\alpha, l)).$
\end{theorem} 

Note that even though we have normalized the discrepancy by subtracting the expected value
an additional normalization is required in Theorem \ref{ThNRandom}(b). The reason for this is explained at
the end of Section~\ref{SSAppl}.

So if one wants to have a unique limit distribution for all $N$ one needs to allow random $\alpha.$

The case when $d=1$ was studied by Kesten.
Define 
$$ V(u, v, w)=\sum_{k=1}^\infty \frac{\sin(2\pi u)\sin(2\pi v)\sin(2\pi w)}{k^2}. $$
If $(r, q)$ are positive integers let
$$ \theta(r,q)=\frac{\Card(j: 0\leq j \leq q-1: gcd(j, r, q)=1)}{\Card(j, k: 0\leq j, k \leq q-1: gcd(j, k, q)=1)}. $$
Finally let
$$ c(r)=\begin{cases}  
\frac{\pi^3}{12} \left[\sum_{r=0}^{q-1} \theta(p, q) \int_0^1 \int_0^1 V(u, \frac{rp}{q}, v) dudv\right]^{-1}
& \text{if }r=\frac{p}{q}
\text{ and } gcd(p,q)=1 \\
\frac{\pi^3}{12} \left[\int_0^1 \int_0^1 \int_0^1 V(u, r, v) du dr dv\right]^{-1}
& \text{if }r \text{ is irrational.}
\end{cases} $$

\begin{theorem}
\label{ThKesten}
(\cite{K1, K2})
If $(\a,x)$ is uniformly distributed on $\Tor^2$ then
$\frac{D_N([0,l], \a,x)}{c(l) \ln N}$ converges to the Cauchy distribution $\Cauchy.$
%That is
%$$\lim_{N\to\infty} \text{Leb}((\a,x): \frac{D_N([0,l], \a,x)}{c(l) \ln N}<z)=\Cauchy(z) $$
%where $\Cauchy$ is the standard Cauchy cumulative distribution function
%$$\Cauchy(z)=\frac{1}{\pi} \int_{-\infty}^z \frac{1}{1+a^2}da=\frac{\tan^{-1} z}{\pi}+\frac{1}{2} .$$ 
\end{theorem}
Note that the normalizing factor is discontinuous as a function of the length of the interval at rational values.

A natural question is to extend Theorem \ref{ThKesten} to higher dimensions. The first issue is to decide which sets $\Omega$ to consider instead of intervals. It appears that a quite flexible assumption is that $\Omega$ is semialgebraic, that is, it is 
defined by a finite number of algebraic inequalities.

\begin{question}
\label{QSemiAlg}
Suppose that $\Omega$ is semialgebraic 
then there is a sequence $ a_N=a_N(\Omega)$ 
such that for a random translation of a random torus $\dfrac{D_N(\Omega, \reals^d/L, \a,x)}{a_N}$
converges in distribution as $N\to\infty.$
\end{question}
By random translation of a random torus, we  mean a translation of random angle $\a$ on a torus $\reals^d/L$ where  $L=A\integers^d$ and the triple $(\a, x, A)$ has a smooth density on $\Tor^d\times \Tor^d\times \text{GL}(\R,d)$. 
Notice that comparing to Kesten's result of Theorem \ref{ThKesten}, Question \ref{QSemiAlg} allows for additional randomness, namely, the torus is random.
In particular, for $d=1$, the study of the discrepancy of visits to $[0,l]$ on the torus $\reals/\integers$ is equivalent to 
the study of the discrepancy of visits to $[0,1]$ on the torus $\reals/(l^{-1} \integers).$ Thus the purpose of the extra randomness
is to avoid the irregular dependence on parameters observed in Theorem
\ref{ThKesten} (cf. also \cite{R1, R2}).

So far Question \ref{QSemiAlg} has been answered for two classes of sets which are natural counterparts to intervals in higher dimensions 
 : strictly convex sets and (tilted) boxes.

%\subsubsection{Limit laws for the discrepancy in the case of analytic convex bodies} To explain the result in the first case we need to introduce some notation.
%Let $\Omega$ be a strictly convex body with smooth boundary (for example a ball). This means that $\partial \Omega$ 
%is a smooth hypersurface of $\reals^d$ with strictly positive  gaussian curvature, or equivalently that 
%$\partial \Omega$ is  a smooth manifold isomorphic under the normal mapping to the unit sphere $\mathbb{S}_{d-1}$.  
%For each vector $\xi \in {\mathbb S}_{d-1}$ there exists a unique point $x(\xi)\in \partial \Omega$ 
%at which the unit outer normal vector is $\xi$. We denote by $K(\xi)$ the gaussian curvature of $\partial \Omega$ at this point. 

 Given a convex body $\Omega$, we consider the family $\Omega_r$ of bodies obtained from $\Omega$ by rescaling it with a ratio $r>0$ (we apply to $\Omega$ the homothety centered at the origin with scale $r$).  We suppose $r<r_0$ so that the rescaled bodies can fit inside the unit cube of $\R^d$. We define
\begin{equation}
D_N(\Omega,r,\a,x)  = \sum_{n=0}^{N-1} \chi_{\Omega_r}(x+n\a) - N {\rm Vol}({\Omega_r})
\end{equation}

%We introduce now some additional notations to describe the limit law of $D_N(\Omega,r,\a,x)$.

%Let
%\begin{equation}
%\cM_d=M\times T^\infty \text{ and }\cM_{2,d}=M\times T_2^\infty 
%\end{equation}

%By abusing sligtly the notation we let  $\bmu$ denote the Haar measures on both $\cM_d$ and $\cM_{2,d}.$  
%If $\mathcal C$ is not symmetric we have to extend the space where the function $\cL_\Omega$ 
%was defined in the case of balls  to 
%the fiber bundle with base $M$ and fiber $T_2^\infty$ 
%Let $\mu$ be the Haar measure on $\cM_2$.  
%Consider the following function on $\cM_{2,d}$
%\begin{equation}
%\label{LatTor}
%\cL'_{\Omega}(L, \theta, b,b')=\frac{1}{\pi^2} \sum_{m\in \cZ}\sum_{p=1}^\infty k(p,m,\theta)
%\frac{\sin (\pi p Z_m)}{R_m^{\frac{d+1}{2}} Z_m p^{\frac{d+3}{2}}}
%\end{equation}
%with 
%\begin{multline} 
%\label{DefK}
%k(p,m,\theta)=  K^{-\frac{1}{2}}(X_m/R_m)  
%\sin(2\pi(p b_m+p(m,\theta) -(d-1)/8)) \\ +K^{-\frac{1}{2}}(-X_m/R_m)  \sin(2\pi(p b'_m-p(m,\theta)-(d-1)/8))\end{multline}
%For the case of symmetric bodies, we define on the space $\cM_d$ the function  
%\begin{multline}
%\label{LimSumSym}
%\cL_{\Omega}(L, \theta, b)= \\  \frac{2}{\pi^2} \sum_{m\in \cZ}\sum_{p=1}^\infty
% K^{-\frac{1}{2}}(X_m/R_m) 
% \frac{\cos(2\pi p(m,\theta)) \sin(2\pi (p b_m-(d-1)/8)) \sin (\pi p Z_m)}
%{R_m^{\frac{d+1}{2}} Z_m p^{\frac{d+3}{2}}}. 
%\end{multline}

%We now give the description of the distribution $\cD_{\mathcal C}$ of Theorem~\ref{main.limit} 
%that generalizes the limit distribution obtained for balls

\begin{theorem}
\label{ThConvex}
(\cite{DF1})
If $(r,\a,x)$ is uniformly distributed on $X=  [a,b] \times \T^d\times \T^d$ then
$\frac{D_N(\Omega, r, \a,x )}{r^{(d-1)/2} N^{(d-1)/2d}}$ has a limit distribution as $N\to\infty.$
%If $\Omega$ is an analytic non
%symmetric strictly convex body in $\reals^d$, then for any $z \in \reals$ we have  
%\begin{equation}
%\label{distr.level}
%$$\lim_{N\to\infty} \Prob((r,\a,x) : D_N( \Omega,r,\a,x)   \leq z)=
%\bmu\left( (L, (\theta, b,b'))\in \cM_{2,d}  : \cL'_{\Omega}(L,\theta, b,b') \leq z \right). $$
%\end{equation}
%If $\Omega$ is symmetric then, for any $z \in \reals$ we have 
%\begin{equation}
%\label{distr.level2}
%$$\lim_{N\to\infty} \Prob((r,\a,x) : D_N( \Omega,r,\a,x)   \leq z)
%=  \bmu\left( (L, (\theta, b))\in \cM_d : \cL_{\Omega}(L,\theta, b) \leq z \right).%\end{equation}
%$$
\end{theorem}
The form of the limiting distribution is given in Theorem \ref{ThConvex2} in Section \ref{ScProofs}.

%\subsubsection{Limit laws for the discrepancy in the case of boxes} 
In the case of boxes we recover the same limit distribution as in Kesten but with a higher power of the 
logarithm in the normalization. 
\begin{theorem}
\label{ThBox} (\cite{DF2}) In the context of Question \ref{QSemiAlg}, if $\Omega$ is a box, then 
 $\dfrac{D_N}{c \ln^d N}$ converges to $\Cauchy$ as $N\to\infty.$
\end{theorem}
Alternatively, one can consider gilded boxes, namely: for $u=(u_1,\ldots,u_d)$ with $0<u_i<1/2$  for every $i$, we  define a cube on the $d$-torus by 
$C_{u}=[-u_1,u_1]\times \ldots [-u_d,u_d]$. Let $\eta>0$ and $M C_{u}$
be the image of $C_{u}$ by a matrix $M \in {\rm SL}(d,\R)$ such that 
$$M = (a_{ij}) \in G_\eta =\{ |a_{i,i}-1|<\eta, \text{ for every } i \text{ and }|a_{i,j}|<\eta \text{ for every }j\neq i\}. $$
For a point $x \in \T^d$ and a translation frequency vector $\a \in \T^d$ 
we denote $\xi=(u,M,\a,x)$
 and define the following discrepancy function 
$$D_N(\xi) = \# \{1\leq m \leq N : (x+
m\a)  {\rm  \  mod  \ } 1 \in M C_{u}  \} - 2^d \left(\Pi_i u_i\right) N.$$
Fix  $d$ segments $[v_i,w_i]$ such that $0<v_i<w_i<1/2 \forall i=1,\ldots,d$.
Let 
\begin{equation}
\label{DefX}
X=(u,\a,x, (a_{i,j})) \in [v_1,w_1] \times \dots [v_d,w_d] \times \T^{2d} \times G_\eta
\end{equation}
 We denote by $\Prob$ the normalized restriction of the Lebesgue $\times$ Haar measure on $X$. Then, the precise statement of Theorem \ref{ThBox} is  
\begin{theorem} (\cite{DF2}) \label{dimd} 
Let $\rho=\frac{1}{d!}\left(\frac{2}{\pi}\right)^{2d+2} . $ If $\xi$ is distributed according to $\lambda$ then
%For any $\eta>0$ and any $z \in \R$ we have
%\begin{equation} \label{cauchy} \lim_{N \to \infty} 
$ \frac{ D_N(\xi)}{\rho (\ln N)^d}$ converges to
$\Cauchy$ as $N\to\infty.$
 %\end{equation}
\end{theorem}

\begin{question} 
\label{QFixVar}
Are Theorems \ref{ThConvex}--\ref{dimd} valid if 

(a) The lattice $L$ is fixed; (b) $x$ is fixed?
\end{question}

\begin{question}
Describe large deviations for $D_N.$ That is, given $b_N\gg a_N$ 
where $a_N$ is the same as in Question \ref{QSemiAlg},
study the asymptotics of 
$ \Prob(D_N\geq b_N).$ One can study this question in the annealed setting when all variables are random or in the quenched setting where some of them are fixed. 
\end{question}

\begin{question}
\label{QLLT}
Does a local limit theorem hold? That is, is it true that given a finite interval $J$ we have
$$ \lim_{N\to\infty} a_N \Prob(D_N\in J)=c |J| ? $$
%where $a_N$ is the same as in Question \ref{QSemiAlg}?
\end{question}

\section{Poisson regime} \label{sec.poisson}
The results presented in the last section deal with the so called CLT regime. This is the regime when, since the target set $\Omega$ is {\it macroscopic} (having volume of order $1$), if $T$ was sufficiently mixing, one would
get the Central Limit Theorem for the ergodic sums of $\chi_\Omega$. In this section we discuss Poisson ({\it microscopic}) regime, that is,  we  let $\Omega=\Omega_N$ shrink 
so that $\EXP(D_N(\Omega_N, \a,x))$ is constant. 
In this case, the sum in the discrepancy consists of a large number of terms each one of which  vanishes with probability close to 1 so that typically only finitely many terms
are non-zero. 

\begin{theorem}
\label{ThFinExp}
(\cite{M-ETDS})
Suppose that $\Omega$ is a bounded set whose boundary has zero measure.
%\margem{maybe we say that if $x$ is irrational vector then the distribution for fixed $x$ is the same as for random $x$?} 

If $(\a,x)$ is uniformly distributed on $\Tor^d\times \Tor^d$ then both
$D_N(N^{-1/d} \Omega, \alpha, x)$ and $D_N(N^{-1/d} \Omega, \alpha, 0)$ 
converge in distribution.
\end{theorem}
Note that in this case the result is less sensitive to the shape of $\Omega$ than in the case of sets of unit size. 

We will see later (Theorem \ref{ThDO-Pois} in Section \ref{ScProofs}) that
one can also handle several sets at the same time.

\begin{corollary}
If $(\a,x)$ is uniformly distributed on $\Tor^d\times \Tor^d$ then the following random variables have limit distributions

(a) $\displaystyle N^{1/d} \min_{0\leq n<N} d(x+n\alpha, \brx)$ where $\brx$ is a given point in $\T^d;$

(b) $\displaystyle N^{2/d} \min_{0\leq n<N} [A(x+n\a)-A(\brx)]$ where $A$ is a Morse function with minimum at $\brx.$
\end{corollary}

\begin{proof}
To prove (a) note that $\displaystyle N^{1/d} \min_{0\leq n<N} d(x+n\alpha, \brx)\leq s$ iff the number of points of the
orbit of $x$ of length $N$ inside $B(\brx, s N^{-1/d})$ is zero.

To prove (b) note that if $A$ is a Morse function and x is close to $\brx$ then $A(x)\approx A(\brx)+(D^2A)(\brx)(x-\brx, x-\brx).$
\end{proof}

There are two natural ways to extend this result.

\begin{question}
If $S\subset \T^d$ is an analytic submanifold of codimension $q$ find the limit distribution of
$\displaystyle N^{1/q} \min_{0\leq n<N} d(x+n\alpha, S).$
\end{question}

\begin{question}
Given a typical analytic function $A$ find a limit distribution of $\displaystyle N^{\frac 1 d} \min_{0\leq n<N} |A(x+n\a)|.$ 
\end{question}

As we shall see in Section \ref{SSProofSing} this question is closely related to Question \ref{q6}.

%By a {\bf random $d$-dimensional lattice} (centered at 0) we mean a lattice $L=Q\Z^d$ where $Q$ is distributed according to a Haar measure on
%$SL_d(\R)/SL_d(\Z).$

%By a {\bf random $d$-dimensional affine lattice} we mean an affine  lattice $L=Q\Z^d+b$ where $(Q,b)$ 
%is distributed according to a Haar measure on
%$(SL_d(\R)\ltimes \R^d)/(SL_d(\Z)\ltimes \Z^d).$

%{\colll We close this section with a probabilistic reminder on Poisson processes that will indeed be useful in the study of the micrscopic regime distribution of the Kronecker sequence $n \a [1]$.}

\section{Outlines of proofs}
\label{ScProofs}
\subsection{Poisson processes} 

\label{SSPois}

In this section  we recall some facts about the Poisson processes 
referring the reader to \cite[Section 11]{M-Bern} or \cite{King} for more details. The next section contains
preliminaries from homogenous dynamics.

Recall that a random variable $N$ has Poisson distribution with parameter $\lambda$ if
$\Prob(N=k)=e^{-\lambda} \frac{\lambda_k}{k!}.$ Now an easy combinatorics shows the following facts 

(I) If $N_1, N_2\dots N_m$ are independent random variables and $N_j$ have Poisson distribution with parameters
$\lambda_j$, then $N=\sum_{j=1}^m N_j$ has Poisson distribution with parameter $\sum_{j=1}^m \lambda_j.$

(II) Conversely, take $N$ points distributed according to a Poisson distribution with parameter $\lambda$ and color each
point independently with one of $m$ colors where color $j$ is chosen with probability $p_j.$ Let $N_j$ be the
number of points of color $j.$ Then $N_j$ are independent and $N_j$ has Poisson distribution with parameter
$\lambda_j=p_j\lambda.$

Now let $(\Omega, \mu)$ be a measure space. By a Poisson process on this space we mean a random point process on $X$
such that if $\Omega_1, \Omega_2\dots \Omega_m$ are disjoint sets and $N_j$ is the number of points in $\Omega_j$
then $N_j$ are independent Poisson random variables with parameters $\mu(\Omega_j)$ (note that this definition is
consistent due to (I)). We will write $\{x_j\}\sim \Pois((X, \mu))$ to indicate that $\{x_j\}$ is a Poisson process with parameters
$(X, \mu).$ If $(X, \mu)=(\R, c \Leb)$ we shall say that $X$ is a Poisson process with intensity $c.$
The following properties of the Poisson process are straightforward consequences of (I) and (II) above.

\begin{proposition}
\label{PrPropPois}
(a) If $\{x_j'\}\sim\Pois(X, \mu')$ and $\{x_j''\}\sim\Pois(X, \mu'')$ are independent then
$\{x_j'\}\cup \{x_j''\}\sim \Pois(X, \mu'+\mu'').$

(b) If $\{x_j\}\sim \Pois(X, \mu)$ and $f:X\to Y$ is a measurable map then
$\{f(x_j)\}\sim\Pois(Y, f^{-1}\mu).$

(c) Let $X=Y\times Z,$ $\mu=\nu\times \lambda$ where $\lambda$ is a probability measure on $Z.$
Then $\{(y_j, z_j)\}\sim \Pois(X, \mu)$ iff $\{y_j\}\sim \Pois(Y, \nu)$ and 
$z_j$ are random variables independent from $\{y_j\}$ and each other and distributed according to $\lambda.$
\end{proposition}

Next recall \cite[Chapter XVII]{Fel} that the Cauchy distribution is unique (up to scaling) symmetric
distribution such that if $Z, Z'$ and $Z''$ are independent random variables with that distribution then $Z'+Z''$ has the same distribution as $2Z.$
We have the following representation of the Cauchy distribution.

\begin{proposition}
(a) If $\{x_j\}$ is a Poisson process with constant intensity then $\sum_j \frac{1}{x_j}$ has Cauchy distribution.
(the sum is understood in the sense of principle value). 

(b) If $\{x_j\}$ is a Poisson distribution with constant intensity and $\xi_j$ are random variables with finite expectation 
independent from $\{x_j\}$ and from each other
then 
\begin{equation}
\label{CauchyPois}
\sum_j \frac{\xi_j}{x_j}
\end{equation}
 has Cauchy distribution.
\end{proposition}

To see part (a) let $\{x_j'\}, \{x_j''\}$ and  $\{x_j\}$ are independent Poisson processes with intensity $c.$ Then 
$$ \sum \frac{1}{x_j'}+\sum \frac{1}{x_j''}=\sum_{y\in \{x_j'\}\cup \{x_j''\}} \frac{1}{y} $$
and by Proposition \ref{PrPropPois} (a) and (b)  both $\{x_j'\}\cup \{x_j''\}$ and $\{\frac{x_j}{2}\}$ are 
Poisson processes with intensity $2c.$

To see part (b) note that by Proposition \ref{PrPropPois}(b) and (c), $\{\frac{x_j}{\xi_j}\}$ is a Poisson process
with constant intensity.

%\section{Remarks on the proofs: Preliminaries.}
\subsection{Uniform distribution on the space of lattices}

\label{notations.lattices} 

In order to describe ideas of the proofs from Sections \ref{ScErgSum}, \ref{sec.discrepancy}, 
and \ref{sec.poisson}
we will first go over some preliminaries.   By a {\bf random $d$-dimensional lattice} (centered at 0) we mean a lattice $L=Q\Z^d$ where $Q$ is distributed according to  Haar measure on
$G=SL_d(\R)/SL_d(\Z).$

By a {\bf random $d$-dimensional affine lattice} we mean an affine  lattice $L=Q\Z^d+b$ where $(Q,b)$ 
is distributed according to a Haar measure on
$\bar G=(SL_d(\R)\ltimes \R^d)/(SL_d(\Z)\ltimes \Z^d).$ Here $\bar G$ is equipped with the  multiplication rule
$(A, a)(B, b)=(AB, a+Ab)$.

%If there is no ambiguity we will use the same notation $\P$ to denote the Haar measure on the homogeneous spaces that appear all through the text. 

We denote by $g_t$ the diagonal action on $G$ given by
$$ g^t=%\left(
\left(\begin{array}{llll}
e^{t/d} &  \dots & 0 & 0 \\
&&&\\
0 & \dots & e^{t/d} & 0\\
0 & \dots & 0 & e^{-t} \end{array}\right) $$
%, \left(\begin{array}{l} 0 \\ {} \\ 0 \\ 0 \end{array}\right) \right), $$
and for $\a \in \R^{d-1}$ we denote by $\Lambda_\a$ the horocyclic action 
\begin{equation}
\label{InCond}
\Lambda_\a=\left(\begin{array}{llll}
1 &  \dots & 0 & \alpha_1 \\
&&&\\
0 & \dots & 1 & \alpha_{d-1} \\
0 & \dots & 0 & 1 \end{array}\right).
\end{equation}

The action of $g^t$ on the space of affine lattices $G$ is partially hyperbolic and 
unstable manifolds are orbits of $\Lambda_\alpha$ where $\alpha \in\R^{d-1}$.

Similarly the action by ${g}_t:=(g_t,0)$, defined on the space of affine lattices is partially hyperbolic 
and 
unstable manifolds are orbits of $(\Lambda_\alpha,\bar{x})$ where $(\alpha,x) \in(\R^{d-1})^2$ and $$\bar x=\left(\begin{array}{l} x_1 \\ {} \\ x_{d-1} \\ 0 \end{array}\right).$$
For convenience, here and below we will use the notation $\bar x= (x_1,\ldots,x_{d-1},0)$ 
for the column vector $ x\in\reals^{d-1}$.

We can also equip $ SL_d(\R)\ltimes (\R^d)^r $ with the multiplication rule
\begin{equation} (A, a_1,\dots, a_r)(B, b_1,\dots, b_r)=(AB, a_1+Ab_1,\dots, a_r+A b_r), \label{eq.mult} \end{equation}
and consider the space of periodic configurations in $d$-dimensional space
 $\hat{G}= SL_d(\R)\ltimes (\R^d)^r/SL_d(\Z)\ltimes (\Z^d)^r$.

The action of $g_t:=(g^t,0,\ldots,0)$ on $\hat{G}$ is partially hyperbolic and 
unstable manifolds are orbits of $(\Lambda_\alpha, \bar x_1,\ldots,\bar x_r)$ where $\alpha, \bar x_j\in\R^{d-1}$.

We will denote these unstable manifolds by $n_+(\a)$ or $n_+(\a,\bar x)$ or $n_+(\a, \bar x_1,\ldots,\bar x_r)$. Note also that   $n_+(\a,\bar x)$ or $n_+(\a, \bar x_1,\ldots,\bar x_r)$ for fixed $\bar x, \bar x_1,\ldots,\bar x_r$ form positive codimension manifolds inside the full unstable leaves of the action of $g_t$.

We will often use the uniform distribution of 
the images of unstable manifolds for partially hyperbolic flows (see e.g. \cite{EM}) to assert that $g_t (\Lambda_\alpha)$ or $g_t (\Lambda_\alpha, \bar x)$ or $g_t (\Lambda_\alpha, \bar x_1,\ldots,\bar x_r)$
becomes uniformly distributed in the corresponding lattice spaces according to their Haar measures as $\alpha, \bar x, \bar x_1,\ldots,\bar x_r$ are independent and distributed according to any absolutely continuous measure on $\R^{d-1}$. 
In fact, if the original measure has smooth density, then one has exponential estimate for the rate of equidistribution
(cf. \cite{KM1}). The explicit decay estimates play an important role in proving limit theorems by martingale methods.
For example, such estimates are helpful in proving Theorem \ref{ThKesten} in Section \ref{sec.discrepancy}
and Theorems \ref{dfv1}, \ref{dfv2} in Section \ref {sec.shrinking}. 

Below we shall also encounter a more delicate situation when all or some of $\brx, \brx_j$ are fixed so we have to
deal with positive codimension manifolds inside the full unstable horocycles. In this case one has to use Ratner 
classification theory for unipotent actions. Examples of unipotent actions are 
$n_+(\cdot)$ or $n_+(\cdot,\bar x)$ or $n_+(\a, \cdot)$. The computations of the limiting distribution of
the translates of unipotent orbits proceeds in two steps (cf. \cite{M-Bern}). 
For several results described in 
the previous sections we need the limit distribution of 
$g_t \Lambda_\alpha w$ inside $X=\bG/\bGamma$ where $X$ can be any of the sets $G, \bar{G}, \hat{G}$
described above and $w\in\bG.$ In fact, the identity
$g_t \Lambda_\alpha w=w(w^{-1} g_t \Lambda_\alpha w)$ allows us to assume that $w=id$ at the cost of replacing the
action of $SL_d(\R)$ by right multiplication by a twisted action 
$\phi_w(M) u=w^{-1} M w u .$ So we are interested in $\phi_w(g_t \Lambda_\alpha)\text{Id}$ for some fixed $w \in X$.  The first step in the analysis is to use Ratner Orbit Closure Theorem 
\cite{Rat2} to find a closed connected subgroup, that depends on $w$, $\bH\subset \bG$ such that 
$\overline{\phi_w(SL_d(\R))\Gamma}=\bH\Gamma$ and $\bH\cap \Gamma$ is a lattice in $\bH.$
The second step is to use Ratner Measure Classification  Theorem \cite{Rat1} to conclude that the sets in question are uniformly
distributed in $\bH\Gamma/\Gamma.$ Namely, we have the following result (see \cite[Theorem 1.4]{Sh} or
\cite[Theorem 18.1]{M-Bern}).

%For example, the approach based on Ratner's theory is most helpful in the proof of Theorem \ref{ThDO-Pois}(b) where $x$ is a fixed irrational vector.  We will deal with this case using a useful corollary of Ratner's theorem that allows to show in some cases that the invariant measures are actually equal to Haar is Shah's theorem that we state now.
%equidistribution with respect to Haar in some cases, i.e. $H=G$ in Ratner's theorem and therefore $\nu$ is the Haar measure is given by Shah's theorem.
\begin{theorem} 
\label{th.shah} 
%(Shah's Theorem, \cite[Theorem 1.4]{Sh})  
%Suppose $X$ contains an embedding of an $SL(d,\R)$ action $\varphi: SL(d,\R) \to X$, $\varphi(SL(d,\R))=H$, and that   $H \Gamma/\Gamma$ is dense in 
%$X/\Gamma$, then 
(a) For any bounded piecewise continuous functions $f:X\to \R$ and $h: \R^{d-1} \to \R$ the following holds 

\begin{equation}Ê\label{eq.shah} 
\lim_{t \to \infty} \int_{\R^{d-1}} f\left(\varphi_w(g_t \Lambda_\a)\right) h(\a) d\a = \int_X f d\mu_{\bH} \mathbb \int_{\R^{d-1}} h(\a) d\a \end{equation}
where $\mu_\bH$ denotes the Haar measure on $\bH\Gamma/\Gamma.$

(b) In particular, if $\phi(SL_d(\R))\Gamma$ is dense in $X$ then 
$$ \lim_{t \to \infty} \int_{\R^{d-1}} f\left(\varphi_w(g_t \Lambda_\a)\right) h(\a) d\a = \int_X f d\mu_{\bG} 
\mathbb \int_{\R^{d-1}} h(\a) d\a 
$$
where $\mu_\bG$ denotes the Haar measure on $X.$
\end{theorem}

To apply this Theorem one needs to compute $\bH=\bH(w).$ Here we provide an example of such computation based on
\cite[Sections 17-19]{M-Bern}, \cite[Theorem 5.7]{M-Inh},
\cite[Sections 2 and 4]{MS-Cr} and \cite[Section 3]{E}.

\begin{proposition}
\label{PrRatnerGroup}
Suppose that $d=2$ and $w=\Lambda(I , \brx_1\dots \brx_r).$ 

(a) If  the vector
$(x_1,  \dots, x_r)$ is irrational then $\bH(w)=SL_2(\R)\ltimes (\R^2)^r.$

{ (b) If the vector $(x_1, \dots, x_k)$ is irrational and for $j>k$
$$x_j=q_j +\sum_{i=1}^k q_{ij} x_j$$ 
where $q_j$ and $q_{ij}$ are rational numbers then 
$\bH(w)$ is isomorphic to $SL_2(\R)\ltimes (\R^2)^k.$}
\end{proposition} 

\begin{proof}
(a) Denote 
$$(M,0)=\left(M,\left(\begin{array}{l} 0 \\ 0 \end{array}\right), \ldots, \left(\begin{array}{l} 0 \\ 0 \end{array}\right)\right), \quad
 U_x= (I,\bar{x}) = \left(\left(\begin{array}{ll} 1 & 0  \\ 0 & 1 \end{array}\right), \bar x_1, \ldots, \bar x_r \right) .$$
We need to show that 
 $U_x^{-1} (M,0) U_x (\gamma, n)$ is dense in 
 $SL_d(\R)\ltimes (\R^2)^r$ as $M$, $\gamma,$ $n=(n_1,\ldots,n_r)$ vary in ${\rm SL}(2,\R)$, ${\rm SL}(2,\Z)$, and 
 $(\Z^2)^{r}$ respectively. It is of course sufficient to prove the density of 
 $$(M,0) U_x (\gamma, n)= (M \gamma, M \bar x_1 + M n_1 , \ldots,M \bar x_r + M n_r  )$$ 
which in turn follows from the density in $(\R^2)^r$ of 
$$\left( \gamma^{-1} (\bar x_1+n_1),\ldots, \gamma^{-1} (\bar x_r+n_r)\right).$$
To prove this last claim fix $\epsilon>0$ and any $v \in (\R^2)^{r}.$ Let $z=(x_1,\ldots,x_r)$.
Since  $\{1, x_1\dots x_r\}$ are linearly independent over $\Q$, the $T_z$ orbit of $0$ is dense in $\T^r$ and hence 
there exists $a\in \N$ and a vector $m_1\in \Z^r$ such that
$$ |a x_i - v_{i,1}-m_{i,1} |<\epsilon. $$
Since the $T_{az}$ orbit of $z$ is dense in $\T^r$ there exists $c\in\mathbb{N}$ such that
$$c\equiv 1\text{ mod }a\text{ and } |c x_i - v_{i,2}-m_{i,2} |<\epsilon. $$  
Since $a \wedge c=1$
we can find $b,d \in \Z$ such that $ad-bc=1.$ Let 
$\gamma^{-1} = \left(\begin{array}{ll} a & b  \\ c & d \end{array}\right)$ and $n_i=-\gamma m_i$ so that 
$\left| \left(\gamma^{-1} (\bar x_i+n_i)\right)_j-v_{i,j}\right| <\eps$ for every $i= 1,\ldots,r$ and $j=1,2$. This finishes the proof of density %and implies that $(g_{\ln N},0)$ 
completing the proof of part (a). 

{
(b) Suppose first that $q_j$ and $q_{ij}$ are integers. In this case a direct inspection shows that
$\phi(SL_2(\R))$ is contained in the orbit of $H=SL_2(\R)\times V$ where 
$$V=\{(z_1, z_2, \dots, z_r) : z_j=q_j+\sum_i q_{ij} z_i\} $$
and that the orbit of $H$ is closed. Hence $\bH(w)\subset H.$ To prove the opposite inclusion it suffices to show that
$\dim(H)=\dim(\bH).$ To this end we note that since the action of $SL_2(\R)$ is a skew product, $\bH$ projects 
to $SL_2(\R).$ On the other hand the argument of part (a) shows that the closure of $\phi(SL_2(\R)) $ contains the
elements of the form $(Id, v)$ with $v\in V.$ This proves the result in case $q_i$ and $q_{ij}$ are integers.

In the general case where $q_j=\frac{p_j}{Q}$ and $q_{ij}=\frac{p_{ij}}{Q}$ where $Q, p_j$ and $p_{ij}$ are integers, 
the foregoing argument shows that 
$$\overline{\phi_w(SL_2(\R)) (SL_2(\Z) \ltimes (\Z/Q)^{2r})}=H(SL_2(\Z) \ltimes (\Z/Q)^{2r}).$$
Accordingly, the orbit of $Id$ is contained in a finite union of $H$-orbits and intersects one of these orbits by
the set of measure at least $(1/Q)^r.$ Again the dimensional considerations imply that $\bH(w)=H.$
}
\end{proof}

%We refer the reader to \cite[Sections 17-19]{M-Bern} or
%\cite[Sections 2 and 4]{MS-Cr} for more details regarding the convergence in Ratner's theorem as well as for computation of the
%Ratner subgroup $H$ in various cases, as for example in the result of  Theorem \ref{ThDO-Pois}(c).

We are now ready to explain some of the ideas behind the proofs of the Theorems of Sections
\ref{ScErgSum}, \ref{sec.discrepancy}  and \ref{sec.poisson} following \cite{DF1, DF2}.
Applications to similar techniques to the related problems could be found
for example in \cite{BM, E, EM, KM1, KM2, M-LH, M2}. %\margem{Why these references only?}.
We shall see later that the same approach can be used to prove several  other limit theorems.

\subsection{The Poisson regime} 
Theorem \ref{ThFinExp} is a consequence of the following more general result.

\begin{theorem}
\label{ThDO-Pois}
(\cite{M-ETDS, MS-LG})
(a) If $(\a,x)$ is uniformly distributed on $\Tor^d\times \Tor^d$ then
$N^{1/d} \{x+n\alpha\}$ converges in distribution to
$$ \{X\in\R^d \text{ such that for some }Y\in [0,1] \text{ the point }
(X,Y)\in L \} $$%\margem{isn't it  $\{X: (X,Y)\in L, Y\in [0,1] \}$}
where $L\in \reals^{d+1}$ is a random affine lattice.
%\margem{isn't there a restriction $N^{1/d} ||x+n\alpha||<1/\eps$? }

(b) The same result holds if $x$ is a fixed irrational vector and $\alpha$ is uniformly distributed on $\Tor^d.$

(c) If $\alpha$ is uniformly distributed on $\Tor^d$ then 
$N^{1/d} \{n\alpha\}$ converges in distribution to
$$ \{X\in \R^d \text{ such that for some }Y\in [0,1] \text{ the point }
(X,Y)\in L \} $$%\margem{isn't it  $\{X: (X,Y)\in L, Y\in [0,1] \}$}
where $L\in \reals^{d+1}$ is a random lattice centered at 0. 
\end{theorem}

Here the convergence in, say part (a), means the following. Take a collection of sets 
$\Omega_1, \Omega_2\dots \Omega_r\subset \R^d$ whose boundary has zero measure and
let $\cN_j(\a,x , N)=\Card(0\leq n\leq N-1: N^{1/d} \{x+n\alpha\}\in \Omega_j. $ Then for each $l_1\dots l_m$
\begin{equation}
\label{EqDO-Pois}
 \lim_{N\to\infty} \Prob(\cN_1(\a,x , N)=l_1,\dots, \cN_r(\a,x , N)=l_r)=
\end{equation}
$$\mu_G(L : \Card(L\cap (\Omega_1\times [0,1]))=l_1,
\dots, \Card(L\cap (\Omega_r\times [0,1]))=l_r)$$
where $\mu_G$ is the Haar measure on $G=SL_d(\R)$, or $G=SL_d(\R)\ltimes \R^d$ respectively.
%in the case of random lattices. 

The sets appearing in Theorem \ref{ThDO-Pois} are called {\bf cut-and-project sets.} We refer the reader to
Section \ref{SSCutProject} of the present paper as well as to \cite[Section 16]{M-Bern} for more discussion of these objects.

\begin{remark} Random (quasi)-lattices provide important examples of  random point processes in the Euclidean space having a large symmetry group. This high symmetry explains why they appear as limit processes in several limit theorems (see discussion in \cite[Section 20]{M-Bern}). Another point process with large symmetry group is a Poisson process discussed in Section \ref{SSPois}.
Poisson processes will appear in
Theorem \ref{pois} below. \end{remark}

A variation on Theorem \ref{ThDO-Pois} is the following. 

Let 
 $G=SL_2(\R)\ltimes (\R^2)^r$ equipped with the  multiplication rule defined in Section \ref{notations.lattices} and consider $X=G/\Gamma$,  $\Gamma=SL_2(\Z)\ltimes (\Z^2)^r$.
\begin{theorem} \label{th.multipole}  Assume $(\a, z_1\dots z_r) \in [0,1]^{r+1}$.   For any collection of sets 
$\Omega_{i,j} \subset \R$, $i=1,\ldots,r$ and $j=1,\ldots,M$ whose boundary has zero measure 
let 
$$\cN_{i,j}(\a, N)=\Card\left(0\leq n\leq N-1: N \{z_i+n\alpha\}\in \Omega_{i,j}\right).$$

(a) If $(\a, z_1\dots z_r)$ are uniformly distributed on $[0,1]^{r+1}$ or if $\a$ is uniformly distributed on $[0,1]$ and $(z_1,\ldots,z_r)$ is  a fixed irrational vector then 
for each $l_{i,j}$
\begin{multline} \label{EqDO-Pois2}
 \lim_{N\to\infty} \Prob \left( \cN_{i,j} (\a, N) = l_{i,j}, \forall i,j\right)  \\ =  \mu_G\left((L,a_1,\ldots,a_r ) \in G  : \Card  (L+a_i \cap (\Omega_{i,j}\times [0,1]))=l_{i,j},  \forall i,j \right)
\end{multline}
where we use the notation $\Prob$ for Haar measure on $\T$ (in the case of fixed vector $(z_1,\ldots,z_r)$) as well as on $\T^{r+1}$ (in the case of random vector $(z_1,\ldots,z_r)$), and $\mu_G$ is the Haar measure on $G$.  

(b) For arbitrary $(z_1, \dots, z_r)$ there is a subgroup $\bH\subset G$ such that 
for each $l_{i,j}$
\begin{multline} \label{EqDO-Pois2-B}
 \lim_{N\to\infty} \Prob \left( \cN_{i,j} (\a, N) = l_{i,j}, \forall i,j\right)  \\ 
=  \mu_\bH \left((L,a_1,\ldots,a_r ) \in G  : \Card  (L+a_i \cap (\Omega_{i,j}\times [0,1]))=l_{i,j},  \forall i,j \right)
\end{multline}
where $\mu_\bH$ denotes the Haar measure on the orbit of $\bH.$ 
\end{theorem}

\begin{proof}[Proof of Theorem \ref{ThDO-Pois}]
(a) We provide a sketch of proof referring the reader to \cite[Section 13]{M-Bern} for
more details.

Fix a collection of sets 
$\Omega_1, \Omega_2\dots \Omega_r\subset \R^d$ and $l_1,\ldots,l_r \in \N$. We want to prove \eqref{EqDO-Pois}.

Consider the following functions on the space of affine lattices 
$\bar G=(SL_{d+1}(\R)\ltimes \R^{d+1})/(SL_{d+1}(\Z)\ltimes \Z^{d+1})$
$$f_j(L)=\sum_{e\in L} \chi_{\Omega_j\times [0,1]}(e)$$
and let
$ \cA=\{L: f_j(L)=l_j\}.$ By definition, the right hand side of  \eqref{EqDO-Pois} is $\mu_{\bar G}( L: L \in \cA)$.

On the other hand, the Dani  
 correspondence principle states that 
\begin{equation}Ê\label{eq.dani}  \cN_1(\a,x , N)=l_1,\dots, \cN_r(\a,x , N)=l_r  \text{ iff }  Ê  g^{\ln N} (\Lambda_\a\Z^{d+1}+\bar x)\in \cA \end{equation}
where $g_t$ is the diagonal action $(e^{t/d},\ldots,e^{t/d},e^{-t})$ and $\bar x= (x_1,\ldots,x_d,0)$ are defined as in Section \ref{notations.lattices}. 
To see this, fix $j$ and suppose that  $\{x+n\alpha\} \in  N^{-1/d}\Omega_j$ for some $n \in [0,N]$. Then 
$$\{x+n\alpha\}=(x_1+n\alpha_1+m_1(n), \dots, x_d+n\alpha_d+m_d(n))$$ 
with $m_i(n)$ uniquely defined so that  $x_i+n\alpha_i+m_i\in (-1/2, 1/2]$,  
and the vector $v=(m_1(n),\ldots,m_d(n),n)$ is such that 
$$  \chi_{\Omega_j\times [0,1]} \left(g^{\ln N} (\Lambda_\a\Z^{d+1}+\bar x) v \right) =1.$$ 
The converse is similarly true, namely that any vector that counts in the right hand side of \eqref{eq.dani} corresponds uniquely to an $n$ that counts in the left hand side visits.

%; and if  $x+n\alpha+m\in N^{-1/d}\Omega_j$ for some $j$ then $x+n\alpha+m$should be close to 0 which implies that $m=m(n).$ Indeed, if we choose $m_i(n)$ so that $x_i+n\alpha_i+m_i\in (-1/2, 1/2]$ then  

%$$ g^t=\left(\left(\begin{array}{llll}e^{t/d} &  \dots & 0 & 0 \\&&&\\0 & \dots & e^{t/d} & 0\\0 & \dots & 0 & e^{-t} \end{array}\right),\left(\begin{array}{l} 0 \\ {} \\ 0 \\ 0 \end{array}\right) \right), $$\begin{equation}
%\label{InCond}Q=\left(\begin{array}{llll}1 &  \dots & 0 & \alpha_1 \\&&&\\0 & \dots & 1 & \alpha_d \\0 & \dots & 0 & 1 \end{array}\right), \quadb=\left(\begin{array}{l} x_1 \\ {} \\ x_d \\ 0 \end{array}\right).\end{equation}
%Namely let$ G=SL_2(\R)\ltimes (\R^2)^r $ with multimplication rule
%$$(A, a_1,\dots, a_r)(B, b_1,\dots, b_r)=(AB, a_1+Ab_1,\dots, a_r+A b_r), $$
%It turns out that the action of $g^t$ on the sapce of affine lattices is partially hyperbolic and 
%unstable manifolds are orbits of $(Q(\alpha), b(x))$ where $\alpha, x\in\R^d$ and 
%$Q$ and $b$ are given by \eqref{InCond}. 

Now \eqref{eq.dani}, and thus \eqref{EqDO-Pois}, follow if we prove that 
\begin{multline} \label{eq.equidist} \lim_{N \to \infty} { \mathbb P} \left( (\a,x) \in \T^d \times \T^d : g^{\ln N} (\Lambda_\a\Z^{d+1}+\bar x)\in \cA\right) =\\ \mu_G( L \in \bar G: L \in \cA) \end{multline} 

Finally, \eqref{eq.equidist} holds due to
 the uniform distribution of 
the images of unstable manifolds $n_+(\a,x)$ for partially hyperbolic flows. 

(b) Following the same arguments as above, we see that in order to prove  Theorem \ref{ThDO-Pois}(b) we need to show 
that $(g^{\ln N},0) . (\Lambda_\a,\bar x) $ becomes equidisitributed with respect to Haar measure on 
$$G=(SL_{d+1}(\R)\ltimes \R^{d+1})/(SL_{d+1}(\Z)\ltimes \Z^{d+1})$$ 
if $\a$ is random and $x$ is a fixed irrational vector.
This can be derived from Theorem \ref{th.shah}(b) using a generalization of Proposition \ref{PrRatnerGroup}(a).

The argument for part (c) is the same as in part (a) but we use the space of lattices rather 
than the space of affine lattices.
\end{proof}

\begin{proof}[Proof of Theorem \ref{th.multipole}]
As in  the proof of Theorem \ref{ThDO-Pois} we use Dani's correspondence principle to identify the left hand side in \eqref{EqDO-Pois2} with 
$$ \Prob\left( \Card  ((     g_{\ln N} \Lambda_\a + g_{\ln N} \bz_i  ) 
 \cap (\Omega_{i,j}\times [0,1]))=l_{i,j},  \forall i,j \right)$$
where $\bz_i = \left(\begin{array}{l} z_i \\ 0 \end{array}\right)$.

Now \eqref{EqDO-Pois2} follows 
if we have that $(g_{\ln N},0) . (\Lambda_\a, \bz_1,\ldots,\bz_r))  $ is distributed 
according to the Haar measure in $X=G/\Gamma$
with 
 $G=SL_2(\R)\ltimes (\R^2)^r$ and   $\Gamma=SL_2(\Z)\ltimes (\Z^2)^r$. But this last statement follows
 from Theorem \ref{th.shah}(b) and Proposition \ref{PrRatnerGroup}(a).

Likewise \eqref{EqDO-Pois2-B} follows from Theorem \ref{th.shah}(a). 
\end{proof}

%If $(\a, z_1\dots z_r)$ are uniformly distributed on $[0,1]^{r+1}$ or if $\a$ is uniformly distributed on $[0,1]$ and $(z_1,\ldots,z_r)$ is  a fixed irrational vector then for each $l_{i,j}$
%\begin{multline} 
 %\lim_{N\to\infty} \Prob \left( \cN_{i,j} (\a, N) = l_{i,j}, \forall i,j\right)  \\ =  \Prob\left((L,a_1,\ldots,a_r ) \in G  : \Card  (L+a_i \cap (\Omega_{i,j}\times [0,1]))=l_{i,j},  \forall i,j \right)
%\end{multline}

% the RHS in \eqref{MainTermMer}  with $ \Phi_R((g_{\ln N},0) . (\Lambda_\a, \bz_1, \dots \bz_r))$ where $\bz_i = \left(\begin{array}{l} z_i \\ 0 \end{array}\right)$

%$ \Phi(g_{\ln N} h(\alpha, z_1,\dots, z_r)) $
%where
%$$g_t=\left(\left(\begin{array}{ll} e^t & 0 \\ 0 & e^{-t}\end{array}\right), 0,\dots 0\right), $$
%$$h(\alpha, z_1, \dots ,z_r)=\left(\left(\begin{array}{ll} 1 & \a \\ 0 & 1 \end{array}\right), 
%\left(\begin{array}{l} z_1 \\ 0 \end{array}\right)
%M\dots \left(\begin{array}{l} z_r \\ 0 \end{array}\right)\right). $$

%For almost every $x \in \T$ we have that $(z_1, \dots ,z_r)$ is an irrational vector so Theorem \ref{th.multipole} applies and we get that 

\subsection{Application of the Poisson regime theorems to the ergodic sums of smooth functions with singularities} 
\label{SSProofSing}

The microscopic or Poisson regime theorems are useful to treat the ergodic sums of smooth functions with singularities since the main contribution to these sums come from the visits to small neighborhoods of the singularity. 

\begin{proof}[Proof of Theorem \ref{ThPole2}]
Due to Corollary \ref{CorErgSmooth} we may assume that $\tilde{A}=0.$

Let $R$ be a large number 
and denote by $S_N'$ the sum of terms with $|x+n \alpha-x_0|>R/N$ and by
$S_N''$ the sum of terms with $|x+n \alpha-x_0|<R/N.$
Then
$$\EXP(|S_N'/N^a|)\leq \frac{C}{N^a} \EXP(|\xi|^{-a} \chi_{|\xi|>R/N})=O(R^{1-a}).$$
On the other hand, by Theorem \ref{ThDO-Pois}(a)
if $(\a,x)$ is random $\dfrac{S_N''}{N^a}$ converges as $N\to\infty$ to 
$$\sum_{(X,Y)\in L: Y\in [0,1], |X|<R} \frac{c_-\chi_{X<0}+c_+ \chi_{X>0}}
{|X|^a},$$
where $L$ is  a random affine lattice in $\R^2$.
Letting $R\to\infty$ we get
$$ \frac{S_N}{N^a} \Rightarrow \sum_{(X,Y)\in L: Y\in [0,1]} \frac{c_-\chi_{X<0}+c_+ \chi_{X>0}}{|X|^a}.$$

The case of fixed irrational $x$ is dealt with similarly using Theorem \ref{ThDO-Pois}(b).
\end{proof}

\begin{proof}[Sketch of proof of Theorem \ref{ThMer}] Consider first the case when the highest pole has order $m>1.$
Then the argument given in the proof of Theorem  \ref{ThPole2} shows that
for large $R,$ $A_N/N^m$ can be well approximated by
\begin{equation}
\label{MainTermMer}
\frac{A_N}{N^m} \sim  \sum_{j=1}^r \sum_{|x+n\alpha-x_j|<R/N} \frac{c_j}{(N(x+n\alpha-x_j))^m}  
 \end{equation}
where $x_1\dots x_r$ are all poles of order $m$ and $c_1\dots c_r$ are the corresponding Laurent coefficients.
%This sum can be analyzed similarly to the proof of Theorem \ref{ThDO-Pois}. 

We use Theorem \ref{th.multipole} to analyze this sum.
Namely, let
 $G=SL_2(\R)\ltimes (\R^2)^r$ with the  multiplication rule defined in Section \ref{notations.lattices} and consider $X=G/\Gamma$,  $\Gamma=SL_2(\Z)\ltimes (\Z^2)^r$.
%$$(A, a_1,\dots, a_r)(B, b_1,\dots, b_r)=(AB, a_1+Ab_1,\dots, a_r+A b_r), $$ $\Gamma=SL_2(\Z)\ltimes (\Z^2)^r$ and 
%let $\tG=SL_2(\R)$ imbedded in $G$ as  the subgroup consisting of elements of the form $(A, 0,\dots, 0).$ 
Consider the functions $\Phi:G/\Gamma\to\R$ given by
\begin{equation}  \label{eq.MainTerMer2} \Phi_R(A, a_1, \dots a_r)=
\sum_{j=1}^r \sum_{e\in A\Z^2+a_j} \frac{c_j}{x(e)^m} 
\chi_{[-R, R]\times [0,1]} (e) \end{equation} 

which as $R \to \infty$ will be distributed as %and introduce $\lim_{R\to\infty} \Phi_R(A,  a_1, \dots, a_r)$ given by 
\begin{equation} \Phi(A, a_1\dots a_r)=\sum_{j=1}^r \sum_{e\in A\Z^2+a_j} \frac{c_j}{x(e)^m}  
\chi_{[0,1]} (y(e)). \label{eq.MainTerMer3}
\end{equation} 
Now Theorem \ref{ThMer} follows from Theorem \ref{th.multipole}. Namely part (a) of Theorem \ref{ThMer} follows from 
Theorem \ref{th.multipole}(a). To get part (c) we let 
$z_j=x-x_j$, we observe that for almost every $x$ the vector $(z_1,\ldots,z_r)$ is irrational. Hence  
Theorem \ref{th.multipole} 
applies and gives us that \eqref{MainTermMer} when $\a \in \T$ is random has the same distribution as \eqref{eq.MainTerMer2} as $(A, a_1\dots a_r)$ is random in $G$. Finally Theorem \ref{ThMer}(b) follows from Theorem \ref{th.multipole} (b).

The proof in case when all poles are simple is the same except that the proof that $A_N/N$ is well approximated by
\eqref{MainTermMer} is more involved since one cannot use $L^1$ bounds. We refer to \cite{SU} for
more details.
\end{proof}

\subsection{Limit laws for discrepancies.} 
\label{SSProofsD}

The proofs of Theorems \ref{ThConvex}--\ref{dimd} use a similar strategy as Theorems \ref{ThPole2} and  \ref{ThMer} of first localizing the important terms and then 
reducing their contribution to lattice counting problems. However the analysis in that case is more complicated,
in particular, because the argument is carried over in the set of frequencies of the Fourier series of the discrepancy 
rather than in the phase space.

Let us describe the main steps in the proof of Theorem \ref{ThConvex}. Consider first the case where $\Omega$ is
centrally symmetric. We start with Fourier series of the discrepancy
\begin{multline*} D_N(\Omega, \a,x )= \\ r^{(d-1)/2} \sum_{k\in \Z^d-0} c_k(r) 
\frac{\cos (2\pi(k,x)+\pi(N-1) (k, \a)) \sin(\pi N(k,\a))}{\sin(\pi(k, \a))}\end{multline*}
where $r^{(d-1)/2} c_k(r)$ are Fourier coefficients of $\chi_{r\Omega}$ which have the following asymptotics
for large $k$ (see \cite{herz})
$$ c_k(r)\approx \frac{1}{\pi |k|^{(d+1)/2}} K^{-1/2} (k/|k|) \sin\left(2\pi\left(r P(k)-\frac{d-1}{8}\right)\right).$$
Here $K(\xi)$ is the Gaussian curvature of $\partial\Omega$ at the point where the normal to $\partial\Omega$ is
equal to $\xi$ and $P(t)=\sup_{x\in \Omega} (x,t).$

The proof consists of three steps. First, using elementary manipulations with Fourier series one shows that the
main contribution to the discrepancy comes from $k$ satisfying
\begin{equation}
\label{ConvSD1}
\eps N^{1/d}<|k|< \eps^{-1} N^{1/d} 
\end{equation}
\begin{equation}
\label{ConvSD2}
k^{(d+1)/2}|\{(k, \a)\}|<\frac{1}{\eps N^{(d-1)/2d}}.
\end{equation}
To understand the above conditions note that \eqref{ConvSD1} and \eqref{ConvSD2} imply that $\{(k, \a)\}$ is of order
$1/N$ so the sum 
$\sum_{n=0}^{N-1} e^{i2\pi (k, (x+n\a))}$ is of order $N,$ that is, it is as large as possible. 
Next the number of terms with $|k|\ll N^{1/d}$ is too small (much smaller than $N$) so for typical $\a$, we have that 
$|\{k, \a)\}|\gg 1/N$ for such $k$ ensuring the cancelations in ergodic sums. For the higher modes $k\gg N^{1/d}$, using $L^2$ norms and the decay of $c_k$ one sees that their  contribution  is negligible 

The second step consists in showing using the same argument as in the proof of Theorem \ref{ThDO-Pois}
that if $(\a,x )$ is uniformly distributed on $\T^d\times \T^d$ then the distribution of
$$ \left(\frac{k}{N^{1/d}}, (k,\a) |k|^{(d+1)/2} N^{(d-1)/2d}\right) $$
converges as $N\to \infty$ to the distribution of
$$ (X(e), Z(e) |X(e)|^{(d+1)/2})_{e\in L} $$
where $L$ is a random lattice in $\R^{d+1}$ centered at $0.$ 

Finally the last step in the proof of Theorem \ref{ThConvex} is to show that if we take prime $k$ 
satisfying \eqref{ConvSD1} and \eqref{ConvSD2} 
then the phases $(k,x),$ $N(k, \a)$ and $r P(k)$ are asymptotically independent
of each other and of the numerators. For $(k,x)$ and $r P(k)$ the independence comes from the fact that 
\eqref{ConvSD1} and \eqref{ConvSD2} do not involve $x$ or $r$,  while $N(k, \a)$ 
has wide oscillation due to the large prefactor $N.$

The argument for non-symmetric bodies is similar except that the asymptotics of their Fourier coefficients is 
slightly more complicated. 

The foregoing discussion explains the form of the limit distribution which we now present.
Let $\cM_{2,d}$ be the space of quadruples $(L, \theta, b, b')$ where $L\in X$, the space of lattices in $\R^{d+1}$, and
$(\theta, b, b')$ are elements of $\T^X$ satisfying the conditions
$$ \theta_{e_1+e_2}=\theta_{e_1}+\theta_{e_2},
\quad b_{me}=mb_e \text{ and } b'_{me}=m b'_e. $$
Let $\cM_d$ be the subset of $M_{2,d}$ defined by the condition $b=b'.$
Consider the following function on $\cM_{2,d}$
\begin{equation}
\label{LatTor}
 \cL_{\Omega}(L, \theta, b,b')=\frac{1}{\pi^2} \sum_{e\in L} \kappa(e,\theta, b, b')
\frac{\sin (\pi Z(e))}{|X(e)|^{\frac{d+1}{2}} Z(e)}
\end{equation}
with 
\begin{multline} 
\label{DefK}
\kappa(e,\theta, b, b')=  K^{-\frac{1}{2}}(X(e)/|X(e)|)  \sin(2\pi(b_e+\theta_e -(d-1)/8)) \\ 
+K^{-\frac{1}{2}}(-X(e)/|X(e)|)  \sin(2\pi (b'_e-\theta_e-(d-1)/8)).
\end{multline}
It is shown in \cite{DF1} that this sum converges almost everywhere on $\cM_{2,d}$ and $\cM_d.$
Now the limit distribution in Theorem \ref{ThConvex} can be described as follows
\begin{theorem}
\label{ThConvex2}
(a) If $\Omega$ is symmetric then
$\dfrac{D_N(\Omega, r, \a,x )}{N^{(d-1)/2d} r^{(d-1)/2}} $ converges to $\cL_\Omega(L, \theta, b, b')$ 
where $(L, \theta, b, b')$ is uniformly distributed on $\cM_d.$

(b) If $\Omega$ is non-symmetric then
$\dfrac{D_N(\Omega, r, \a,x )}{N^{(d-1)/2d} r^{(d-1)/2}} $ converges to $\cL_\Omega(L, \theta, b, b')$ 
where $(L, \theta, b, b')$ is uniformly distributed on $\cM_{2,d}.$
\end{theorem}

\begin{question}
Study the properties of the limiting distribution in Theorem \ref{ThConvex2}, in particular its tail behavior.
\end{question}

The next question is inspired by Theorem \ref{ThLargeQ} from Section \ref{ScHD}.

\begin{question}
Consider the case where $\Omega$ is the standard ball. Thus in \eqref{DefK} $K\equiv 1.$
Study the limit distribution of $\cL$ when the dimension of the torus $d\to\infty.$
\end{question}

%At this point the reader may wonder why the limit distribution in Theorems \ref{ThConvex} and \ref{dimd} look so different.
%So our next step is to recall that Cauchy distribution has a representation similar to \eqref{LatTor}. 

Next we describe the idea of the proof of Theorem \ref{dimd}.
Note that \eqref{CauchyPois} looks similar to \eqref{LatTor}.
The main ingredient in the proof of  Theorem \ref{dimd}
involves a result on the distribution of small divisors of multiplicative form $\Pi | k_i |   \| (k,\a)\|$. Namely, a 
 harmonic analysis of the discrepancy's Fourier series related to boxes allows to bound the frequencies that have essential contributions to the discrepancy and show that they must be resonant with $\a.$
The main step is then to establish a Poisson limit theorem for the distribution of small denominators and
the corresponding numerators. With the notation introduced before the statement of Theorem \ref{dimd} 
let $\brk_i=a_{i,1} k_1+\dots+a_{i,d}k_d .$ Then we have 

\begin{theorem}  \label{pois} 
(\cite{DF2}) 
Let $\xi \in X$ 
be distributed according to the normalized Lebesgue measure $\lambda.$
Then as $N\to\infty$ the point process 
$$\left\{\left( (\ln N)^d  \Pi_i \bar{k}_i  \|(k, \a)\|, N(k,\a) {\rm mod} (2), \{\brk_1 u_1\},
\dots \{\brk_d u_d\}, \right)\right\}_{k \in Z(\xi, N)} $$
where

\begin{multline*} Z(\xi,N) = \left\{ k \in \Z^d : |\bar{k}_i| \geq 1, |\Pi_i \bar{k}_i|<N,  \bar k_1>0, \right. \\ \left.  |\Pi_i \bar{k}_i|   \|(k,\a)\| \leq \frac{1}{\e (\ln N)^d},  
   \right. \\ \left. \exists m \in \Z  {\rm \ such \ that \ }  k_1 \wedge \ldots \wedge  k_d \wedge m =1  {\ and \ }  \|(k,\a)\|= (k,\a)+m \right\} \end{multline*}
converges to a Poisson process on $\R \times \R / (2 \Z) \times (\R/\Z)^d $ with intensity 
$2^{d-1} \bc_1/d!.$ 
\end{theorem}

Comparing this result with the proof of Theorem \ref{ThConvex} discussed above we see that
Theorem \ref{pois} comprises analogies of both step 2 and 3 in the former proof. Namely, it shows 
both that the small denominators contributing most to the discrepancy have asymptotically
Poisson distribution and that the numerators are asymptotically independent of the denominators
(cf. Proposition \ref{PrPropPois}(c)). 

We note that Theorem \ref{pois} is interesting in its own right since it describes the number of solutions
to Diophantine inequalities 
$$\Pi_i |\bar{k}_i| \|(k, \a)\|<\frac{c}{\ln^d N}, \quad |\brk_i|>1, \quad
\prod_i |\brk_i|<N. $$

\begin{question} \label{dista} 
What happens if  in Theorem \ref{pois} $\ln^d N$ is replaced by $\ln^a N$ with $a\in (0, d)$?
\end{question}

\begin{question}
\label{QSubMani}
Is Theorem \ref{pois} still valid if the distribution of $\xi$ is concentrated on a submanifold of $X?$ For example, one can take
$\alpha=(s, s^2).$ 
\end{question}

A special case of Question \ref{QSubMani} is when the matrix $(a_{i,j})$ is fixed equal to Identity. 
This case is directly related to Question \ref{QFixVar}(a).

The proof of Theorem \ref{pois} proceeds by martingale approach (see \cite{dSL, D-Lim}) which requires
good mixing properties in the future conditioned to the past. In the present setting, to apply this method 
it suffices to prove that most orbits of certain unipotent subgroups are equidisitributed at a polynomial rate.
Under the conditions of Theorem \ref{pois} one can assume (after an easy reduction) that the initial point has smooth
density with respect to Haar measure. Then the required equidistribution follows easily form polynomial mixing 
of the unipotent flows. In the setting of Question \ref{QSubMani} (as well as Question \ref{QKestenLF} 
in Section \ref{ScHD})
the initial point is chosen from a positive codimension submanifold so one cannot use the mixing argument. 
The problem of estimating the rate of equidistribution for unipotent orbits starting from submanifolds interpolates between
the problem of taking a random initial condition with smooth density which is solved and the problem of taking fixed initial
condition which seems very hard.

\section{Shrinking targets} \label{sec.shrinking} 
Another classical result in probability theory is the Borel-Cantelli Lemma which says that if $A_j$ are independent sets and
$\sum_j \Prob(A_j)=\infty$ then $\Prob$-almost every point belongs to infinitely many sets. 
A yet stronger  conclusion is given by the {\it strong Borel-Cantell Lemma} claiming that the number of $A_j$ which happen
up to time $N$ is asymptotic to $\sum_{j=1}^N \Prob(A_j).$
In the context of ergodic dynamical systems $(T,X,\mu)$, the law of large numbers is reflected in the Birkhoff theorem of almost sure converge in average of the ergodic means associated to a measurable observable, for example the characteristic function of a measurable set $A \subset X$. In a similar fashion one can study the so called dynamical Borel-Cantelli properties of the system $(X,T,\mu)$ by considering instead of a fixed stet $A$  a sequence of "target" sets $A_j \in X$ such that $\sum \mu(A_j)=\infty$. We then say that the dynamical Borel-Cantelli property is satisfied by $\{A_j\}$ if for almost every $x$, $T^j(x)$ belongs to $A_j$ for infinitely many $j$. 

In the context of a dynamical system $(T,X,\mu)$  on  a metric space $X$ it is natural to assume that the sets in question have nice geometric structure, since it is always possible for any dynamical system (with a non atomic invariant measure) to construct sets with divergent sum of measures that are missed after a certain iterate by the orbits of almost every point 
\cite[Proposition 1.6]{ChKl}. The simplest assumption is that the sets be balls. The dynamical Borel-Cantelli property for balls is a common feature  for deterministic systems displaying hyperbolicity features (see \cite{HV, Ph, D-Lim} and references therein). 

Due to strong correlations among iterates of a toral translation the dynamical Borel-Cantelli properties are more delicate in the quasi-periodic context. 

\subsection{Dynamical Borel-Cantelli lemmas for translations.} For toral translations one needs also to assume that the sets are nested since otherwise one can take
$A_j \subset (A_0+j\alpha)$ for some fixed set $A_0$ ensuring that the points from the compliment of $A_0$ do not visit any $A_j$ at time $j$.  This motivates 
the following definition (see \cite{HV,ChKl,F-STP}).

Given $T:(X, \mu)\to (X, \mu)$ let
$V_N(x,y)=\sum_{n=1}^N \chi_{B(y, r_n)} (T^n x).$ 
We say that $T$ has the {\bf shrinking target property} (STP)  if for any $y,$ $\{ r_n\}$
such that  $\sum_n \mu(B(y,r_n))=\infty$, it holds that $V_N(x,y)\to \infty$ for almost all $x,$ i.e.  the targets sequence $(B(y,r_n))$ satisfies the Borel-Cantelli property for $T$. 
We say that
$T$ has the {\bf monotone shrinking target property} (MSTP)  if for any $y, \{r_n\}$ 
such that $\sum_n \mu(B(r_n))=\infty$ and $r_n$ 
is non-increasing $V_N(x,y)\to \infty$ for almost all $x.$

In the case of translations, we can always assume without loss of generality that $y=0$ (replace $x$ by $x-y$). We then use the notation $V_N(x)$ for $V_N(x,y)$. We also use the notation $B(r)$ for the ball $B(0,r)$. Another interesting choice is to take $y=x$
in which case we study the rate of return rather than the rate of approach to $0.$ Note that if $V_N(x,x)$ does not depend on $x$
and so the number of close returns depends only on $\alpha.$ 
We shall write $U_N(\a)=\sum_{n=0}^{N-1} \chi_{B(r_n)} (T^n 0).$

%Shrinking target properties are common for deterministic systems displaying hyperbolicity features (see \cite{ChKl,Ath,ChCo} and references therein). As in the rest of these notes, we are interested in studying these properties for toral translations, individually or in average when the frequency is random. 

The following is a straightforward consequence of the fact that toral translations are isometries. 
\begin{theorem}
(\cite{F-STP}) Toral translations do not have STP.
\end{theorem}

%Recall the definition of the Diophantine set $\cD(\sigma)$ (see \eqref{DefDio}). The following result shows that the MSTP for translations characterizes the vectors of constant type, i.e. the set  $ \cD({d})$.
It turns out that the following Diophantine condition is relevant to this problem. Let 
\begin{equation}
\label{DInd}
\cD^*(\sigma)=\{\alpha: \forall k\in \integers-0,  \max_{i\in [1,d]}  \|k\alpha_i\| \geq C |k|^{-(1+\sigma)/d}\}.
\end{equation}

\begin{theorem}
(\cite{Ku}) \label{Ku}
A toral translation $T_\a$ has the MSTP iff $\alpha\in \cD^*(0).$
\end{theorem}
A simple proof of Theorem \ref{Ku} can be found in \cite{F-STP}.  Recall that $\cD^*(0)$ 
has zero Lebesgue measure. Hence, the latter result shows that one has to further restrict the targets if one wants that typical translations display the dynamical Borel-Cantelli property relative to these targets. 

One possible restriction on the targets  is  to impose a certain growth rate on the sum of their volumes.  This actually allows to further distinguish among distinct Diophantine classes as it is shown in the following result. We say that $T$ has  {\bf s-(M)STP}  if for any $\{ r_n\}$ such that $\sum_n r_n^{ds}=\infty$ (and $r_n$ is non-increasing)
$V_N(x)\to \infty$ for almost all $x.$ We then have the following. 
\begin{theorem} 
(\cite{Ts})

a) If $\alpha\not\in \cD^*(sd-d)$, then the  toral translation $T_\a$ does not have the s-MSTP.

b) A circle rotation $T_\a$  has the s-MSTP iff $\alpha\in \cD^*(s-1).$
\end{theorem}

\begin{question}
Is this true that  the  toral translation $T_\a$ has the s-MSTP iff $\alpha\in \cD^*(sd-d)?$
\end{question}

Another possible direction is to study specific sequences, asking for example that $r_n=c n^{-\gamma},$ 
or that $n r_n^d$ be decreasing, in which case the sequence $r_n$ is coined a {\bf Khinchin sequence}. 
The case $r_n=c n^{-1/d}$ in dimension $d$ is very particular, but important.  Indeed a vector $\a \in \T^d$ is said to be {\bf badly approximable} if for some $c>0$, the sequence $\lim_{N\to\infty} U_N(\a, \{c n^{-1/d}\})<\infty.$ It is known that the set of badly approximable vectors has zero measure. By contrast, vectors $\a$ such that 
$\lim_{N\to\infty} U_N(\a, \{c n^{-(1/d+\eps)}\})=\infty$
for some $\eps>0$ are called {\bf very well approximated, or VWA}. 
The obvious direction of the Borel-Cantelli lemma 
implies that almost every $\a \in \T^d$ is not very well approximated (cf. \cite[Chap. VII]{Cassel}).
The latter  facts are particular cases of a more general result, the Khintchine--Groshev theorem on Diophantine approximation which gives a very detailed description of the sequences such that $U_N(x, \{r_n\})$ diverges 
for almost all $\a.$
We refer the reader to \cite{BBKM} for a nice discussion of that theorem and its extensions, and to Section \ref{sec.linearforms} below. 

Khinchin sequences also display BC property much more likely than general sequences. 
For example, compare Theorem \ref{ChCo}(b) below with Theorem \ref{Ku}  which shows that the set of vectors having mSTP has
zero measure.

If a shrinking target property holds it is natural to investigate the asymptotics of the number of target hits.  This makes  the following definition natural. 
We say that   a given sequence of targets $\{A_n\}$ is  {\bf sBC or strong Borel-Cantelli}  for $(T,X,\mu)$  
if for almost every $x$  $$\lim_{N \to \infty} \frac{\sum_{n=1}^N \chi_{A_n} (T^n x)}{\sum_{n=1}^N \mu(A_n)}= 1.$$

%A particular, but important, case in dimension $d$ is given by the sequence of targets  $B_{r_n}$ $r_n=c/n^{1/d}$.
% In \cite{ChCo} the following was proven, improving on some earlier results, for example \cite{schmidt0,schmidt} \margem{maybe we have to add things here?}
\begin{theorem} \cite{ChCo} \label{ChCo}
(a) For every $\a \in \T$ such that its  convergents  satisfy $a_n \leq Cn^{\frac76}$ the sequence $\{B({\frac{c}{n}})\}$ is sBC for $T_\a.$

(b) For almost every $\a \in \T,$ any Khinchin sequence  is sBC for $T_\a.$

(c) For any  $\a \in   \cD(1)$, and any decreasing sequence $\{ r_n\}$ such that $\sum r_n = \infty,$ 
$\{B(r_n)\}$ is sBC for $T_\a$. 
\end{theorem}

Observe that the condition in (a) has full measure.  On the other hand, 
it is not hard to see that if $a_n(\a) \sim n^{2+\eps}$ for every $n$ then the sequence $(B({\frac{1}{n}}))$ does not have the sBC for $T_\a$. Indeed, if $$ x \in \left[\frac{k}{q_n}- \frac{1}{2nq_n},\frac{k}{q_n}+ \frac{1}{2nq_n}\right]$$ then since $\|q_n \a\| \leq \frac{1}{q_{n+1}} \leq   \frac{2}{n^{2+\eps} q_{n}}$ and $\ln q_n \leq Cn \ln n$
$$\sum_{l=q_n}^{n^{1+\eps/2}q_n} \chi_{B(\frac 1 n)} (x+l\a) \geq n^{1+\eps/2} \gg  \sum_{l=1}^{n^{1+\eps/2}q_n} \frac 1 n.$$ 
But it is easy to see that a.e. $x$ belongs to infinitely many intervals of the form $ [\frac{k}{q_n}- \frac{1}{2nq_n},\frac{k}{q_n}+ \frac{1}{2nq_n}]$.

In higher dimensions, it was proved in \cite{schmidt}  that 

\begin{theorem} \label{DFVsBC} If $\sum_n r_n^d=\infty$ then
for almost every  vector $\a \in \T^d$, the sequence $(B({r_n}))$ 
is sBC for the translation $T_\a.$
\end{theorem}

\subsection{On the distribution of hits.} 
Theorems \ref{ChCo} and \ref{DFVsBC} motivate the study of the error terms 
$$\Delta_N(c,\a,x)= V_N(\a,x ) - \sum_{n=1}^N \Vol(B_{r_n})  
\text{ and }
\brDelta_N(c,\a)= U_N(\a)- \sum_{n=1}^N \Vol(B_{r_n}).$$ 

One can for example try to give lower and upper asymptotic bounds on the growth of $\Delta_N$ as a function of the arithmetic properties of $\a$ in the spirit of Kintchine-Beck Theorem \ref{th.beck} and Questions \ref{q7}--\ref{q9}.
Here we will be interested in the distribution of $\Delta_N(c,\a,x)$ after adequate normalization when $\a$ or $x$ or $(\a,x)$ are random.

%One can also study the distribution of $V_n.$

\begin{theorem}
(\cite{beck4, M-ST})
\label{CLTShrinkBad} Let $r_n=c n^{-1/d}.$
Suppose that  $x$ is uniformly distributed on $\T^d.$ For any $c >0$, if $\alpha\in \cD^*(0)$,  there is a constant $K$  such that all limit points of
$\dfrac{\Delta_N(c,\a,x)}{\sqrt{\ln N}}$ are $\sN(\sigma^2)$ with 
$\sigma^2\leq K.$
\end{theorem}

In the case of random $(\a,x)$ we have 
\begin{theorem} \label{dfv1} Let $r_n=c n^{-1/d}.$
(\cite{DFV}) There is $\Sigma(c,d)>0$ such that
if $(\alpha,x)$ is uniformly distributed on $\Tor^d \times \Tor^d$ then
$\dfrac{\Delta_N(c,\a,x)}{\sqrt{\ln N}}$ converges to $\sN(\Sigma(c,d)).$
\end{theorem} 

There is an analogous statement for the return times.

\begin{theorem} \label{dfv2}
(\cite{philipp.book, samur,DFV}) Let $r_n=c n^{-1/d}.$
There is $\bar{\Sigma}(c,d)>0$ such that
if $\alpha$ is uniformly distributed on $\Tor^d$ then 
$\dfrac{\brDelta_N(c,\a)}{\sqrt{b_N}}$ converges in distribution to $\sN(\bar{\Sigma}(c,d))$ where
$$b_N=\begin{cases} 
\ln N \ln \ln N & \text{ if }d=1 \\
\ln N & \text{ if }d\geq 2 .
\end{cases} $$
\end{theorem}

The case $d=1$ was obtained in \cite[Theorem 3.1.1 on page 44]{philipp.book} 
(see also \cite{samur}), based on the metric theory of the continued fractions. In fact, one can handle more general sequences. Namely, let $\phi(k)$ satisfy the following conditions
\begin{itemize}
\item[(i)] $\phi(k)\searrow 0,$ but $\sum_k \phi(k)=+\infty, $
\item[(ii)] There exists $0<\delta<1/2$ such that  
$\sum_{k=1}^n \frac{\phi(k)}{k^\delta} \leq 
C \sqrt{\sum_{k=1}^n \phi(k)} $
\item[(iii)] $\sum_{k=1}^n \phi^2(k)\leq 
C \sqrt{\sum_{k=1}^n \phi(k)}. $
\end{itemize}

\begin{theorem}(\cite{fuchs})
If $r_n=\frac{\phi(\ln n)}{n}$ and $\a$ is uniformly distributed on $\T$ then
$\dfrac{\brDelta_N(c,\a)}{\sqrt{F(n) \ln F(n)}}$ converges in distribution to $\sN(\bSigma(c))$ where
$F(n)=\sum_{k=1}^n \frac{\phi(\ln k)}{k}.$
\end{theorem}

The higher dimensional case is obtained {\it via}  ergodic theory of homogeneous flows and martingale methods in \cite{DFV}.

\begin{question} \label{q.intermediate2}
Study the limiting distribution of $U_N$ and $V_N$ in case
$r_n=\frac{c}{n^\gamma}$ with $\gamma<\frac1d.$
\end{question}

\begin{question}
Do Theorems \ref{DFVsBC},
\ref{dfv1} and \ref{dfv2} hold when the random vector $\a$ is taken from a proper submanifold of $\T^d,$
for example $\a=(s, s^2,\dots,s^d).$
\end{question}

One motivation for this question comes from {\it Diophantine approximation on manifolds} (see \cite{BBKM}
and references wherein), another is multidimensional extension of Kesten Theorem (cf. Question \ref{QSubMani}).

\subsection{Proofs outlines.} \label{tsp.proofs}

First we sketch a proof of Theorem \ref{DFVsBC} in case $r_n=c n^{-1/d}.$ %and \ref{dfv1} 
Consider the number $N_m(\a,x )$ of solutions to
$$ x+n\alpha\in B(0, cn^{-1/d}), \quad e^m<n<e^{m+1} . $$
The argument used to prove Theorem \ref{ThDO-Pois} shows that
$$ N_m(\a,x )=f(g_m (\Lambda_\a \Z^{d+1}+x)) $$
where $f$ is the function on the space of affine lattices given by
$$ f(L)=\sum_{v\in L} \chi_{B(0,c)\times [1, e]}(v). $$
Thus 
\begin{equation} \label{eq.danistp} 
\sum_{m=1}^{\ln N} N_m(\a,x ) \sim\sum_{m=1}^{\ln N} f(g_m (\Lambda_\a \Z^{d+1}+x)) \end{equation}
and Theorem \ref{DFVsBC} for $r_n=c n^{-1/d}$ 
reduces to the study of ergodic sums \eqref{eq.danistp} under the assumption that the initial condition has a density
on $n_+(\a,x)$. In fact, a standard argument allows to reduce the problem to the case when the initial condition has density on the space
of lattices. Namely, it is not difficult to check that the ergodic sums of $f$ do not change much  if we move in the stable or neutral direction
in the space of lattices. After this reduction, the sBC property follows from the Ergodic Theorem. 
%This can be viewed  {\it via} the Birkhoff ergodic theorem applied to the mixing action of $g_m$, and the  observation that the asymptotics of the RHS normalized by $\ln N$ remain almost unchanged if we move slightly the point given by $\a$ and $x$ in the other directions inside the space of lattices since these additional directions are contained in the  center stable direction of $g_m$. 

The relation \eqref{eq.danistp} also allows to reduce  Theorem \ref{dfv1} to a Central Limit Theorem for ergodic sums of $g_m$
which can be proven, for example, by a martingale argument (see \cite{LB}. We refer the reader to \cite{dSL} for a nice introduction
to the martingale approach to limit theorems for dynamical systems.)

The proof of Theorem \ref{dfv2} is similar but one needs to work with lattices centered at 0 rather than affine 
lattices.

In particular, the non-standard normalization in case $d=1$ is explained 
by the fact that $f$ in this case is not in $L^2$ and the main contribution comes from the region where 
$f$ is large (in fact, the analysis is similar to 
\cite[Section~4]{G-Bern}).

%\subsection{Intermediate regime} 
%Not much is known for lager yet shrinking targets. 

%\begin{question} \label{q.intermediate1}
%Show that $r_n=\frac{c}{n^\gamma}$ with $\gamma<\frac1d$ is sBC for almost every $T_\a$?
%\end{question}

%\begin{theorem}
%{\bf Kim 2007} For each $y\in \Tor^1, \alpha\in \reals-\rationals$ if $r_n=\frac{1}{n}$
%$S_N\to\infty$ almost surely.
%\end{theorem}

\section{Skew products. Random walks. } \label{sec.skew}
\subsection{Basic properties.}
The properties of ergodic sums along toral translations are crucial to the study of some classes of dynamical systems, such as skew products or special flows. In this section we consider the skew products. Special flows are the subject  of Section \ref{ScFlows}.

{\bf Skew products} above $T_\a$ will be 
denoted $S_{\a,A} : \T^d \times \T^r  \to \T^d \times \T^r $ They are given by $S_{\a,A}(x,y) = (x+\a,y+A(x) \text{ mod }1)$.
{\bf Cylindrical cascades} above $T_\a$   will be denoted  $W_{\a,A} : \T^d \times \R^r  \to \T^d \times \R^r .$ 
They are given by $W_{\a,A}(x,y) = (x+\a,y+A(x))$. 
Note that  $$W_{\a,A}^N(x,y)=(x+N\a, y+A_N(x))$$
(the same formula holds for $S_{\a, A}$ but the second coordinate has to be taken mod 1).
If $A$ takes integer values then $W_{\a, A}$ preserves $\T^d\times \Z^r$ 
and it is natural to restrict the dynamics to this subset.
Thus cylindrical cascades define random walks on $\R^r$ or $\Z^r$ driven by the translation $T_\a$.

If $\a $ is Diophantine and $A$ is smooth then the so called linear cohomological equation similar to \eqref{CoB} 
\begin{equation}
\label{CoB2}
A(x)-\int_{\T^d} A(u) du =-B(x+\alpha)+B(x)
\end{equation}
has a smooth solution $B$, thus  $S_{\a,A}$ and $W_{\a,A}$ are respectively smoothly conjugated to the translations  $S_{\a,\int_{\T^d} A}$ and  $W_{\a,\int_{\T^d} A}$  via the conjugacy $(x,y) \mapsto (x,y-B(x))$.

Hence the ergodic properties of the skew products and the cascades with smooth $A$
are interesting to study only in the Liouville case. The following is a convenient ergodicity criterion 
for skew products.

\begin{proposition} \label{erg}
\cite{KN}
$S_{\a,A}$ is ergodic 
iff for any $\lambda \in \Z^r-\{0\}$, $\langle \lambda , A \rangle$ is not a measurable 
multiplicative coboundary above $T_\a$, that is, iff there does not exist $\lambda \in \Z^r-\{0\}$ and a measurable solution $\psi : \T^d \to \mathbb C$ to 
%$$ \psi(x+\a)-\psi(x)=A(x)-\int_{\T^d} A \hspace{3cm} (L)$$
 \begin{equation}
 \label{MultCoBZ}
 e^{i2 \pi \langle \lambda , A(x) \rangle }= \psi(x+\a)/\psi(x).
 \end{equation}
\end{proposition}

This ergodicity criterion can be simply derived from the observation that the spaces $V_\lambda$ of functions of the form 
\begin{equation} \phi(x) e^{i2\pi \langle \lambda ,y\rangle}  \label{FiberFourier} \end{equation}
 are invariant under $S_{\a,A}$. It then follows that the existence or nonexistence  
of an invariant function $\varphi$ by $S_{\a,A}$ is determined by the existence or nonexistence  
 of a solution to \eqref{MultCoBZ}. We refer the reader to Section \ref{ScFlows} for further discussion concerning~\eqref{MultCoBZ}.

{ When $A$ is not a linear coboundary, i.e. \eqref{CoB2} does not have a solution, 
it is very likely and often easy to prove that \eqref{MultCoBZ} 
does not have a solution either. 
 For example, it suffices to show that the sums $A_{N_n}$ do not concentrate on a 
subgroup of lower dimension for a sequence $N_n$ such that $T_\a^{N_n} \to {\rm Id}.$ %because in this case  there cannot exist a measurable solution $\psi$ to \eqref{MultCoBZ}. 
Indeed, if a solution  to \eqref{MultCoBZ} exists then $|\psi|$ is constant by ergodicity of the base translation. Therefore by Lebesgue Dominated Convergence Theorem
 $$\lim_{n\to\infty} \int _{\T^d}  e^{i2 \pi \langle \lambda , A_{N_n}(x) \rangle } dx=\lim_{n\to\infty} \int _{\T^d} \psi(x+N_n \a)/\psi(x) dx=1$$
which means that $A_{N_n}(x)$ is concentrated near the set 
$$\{u\in\R^r: \langle  \lambda, u \rangle\in \Z\}.$$
}

In particular it was shown, in \cite{etds1}, that for every Liouville translation vector $\a \in \R^d,$  
the generic smooth function $A$ does not admit a solution to \eqref{MultCoBZ} for any $\lambda \in \R^d-\{0\}$. Hence 
the generic smooth skew product above a Liouville translation is ergodic (cf. Section \ref{SSErgCasc} and Theorem \ref{3} 
in Section \ref{ScFlows}). 

It is known that ergodic skew products $S_{\a,A}$ are actually uniquely ergodic (see \cite{parry}). On the other hand, skew products above translations are never weak mixing since they have the translation as a factor.
However, the same ideas as the ones used to prove ergodicity of the skew products often prove that all eigenfunctions come from that factor 
(see \cite{GLL, Fr, Iw, ILR, W}). 

If one considers skew products on $\T\times \T$ with smooth increasing functions on $(0,1)$ having a jump discontinuity at $0$ then the corresponding skew product will even be mixing in the fibers, that is,  the correlations between functions that depend only on the fiber coordinate tends to 0. A classical example is given by the skew shift $(x,y) \mapsto (x+\a,y+x)$. The mixing in the fibers can be easily derived from the invariance of 
$V_\lambda$ defined by \eqref{FiberFourier} and the fact that, by the Ergodic Theorem, 
$\frac{\partial A_n}{\partial x}=\left(\frac{\partial A}{\partial x}\right)_n\to +\infty.$ 
 A similar  phenomenon can occurs for analytic skew products that are homotopic to identity but over higher dimensional tori   $\T^d \times \T \ni (x,y) \mapsto (x+\a,y+\phi(x))$,
with $\a$ and $\phi$ as in Theorem \ref{nnnspecial} below (see \cite{fskew}). 
This mechanism can also be used to establish ergodicity of cylindrical cascades (see \cite{Pas}).
A fast decay of correlations in the fibers can be responsible for the existence of non trivial invariant distributions  for these skew products similarly to what occurs for the skew shift $(x,y) \mapsto (x+\a,x+y)$ (see \cite{katok}).

The deviations of ergodic sums for skew products, that is the behavior of the sums
$$ \sum_{n=0}^{N-1} B(S_{\a, A}^n (x, z))-N \int_{\T^d}Ê\int_{\T^r}  B(x,z) dx dz $$
is poorly understood. The only cases where some results are available have significant extra symmetry  
\cite{katok,M-LH, FF}.

\subsection{Recurrence.} 
\label{SSRec}
Our next topic are cylindrical cascades. As it was mentioned above they are sometimes called deterministic 
random walks.
So the first question one can ask is if the walk is {\bf recurrent}
(that is, $A_N$ returns to some bounded region infinitely many times) or {\bf transient}.
We will assume in this section that $A$ has zero mean since otherwise $W_{\a,A}$ is transient by the ergodic theorem.
If $r=1$ this condition is also sufficient. In fact, the next result is valid for skew products over arbitrary ergodic
transformations (in fact, there is a multidimensional version of this result, see Theorem \ref{SmallWorld}). 

\begin{theorem}
\label{ThRecDim1}
(\cite{At}) If $r=1,$ $A$ is integrable and has zero mean then $W$ is recurrent.
\end{theorem}

\subsubsection{Recurrence and the Denjoy Koksma Property.} \label{DKP}
Next we note that if the base dimension $d=1$ and $A$ has bounded variation then $W$   
is recurrent for all $r$ and for all $\alpha\in \reals-\rationals$ due to the Denjoy-Koksma inequality stating that 
%Of course, the Denjoy Koksma inequality for functions of bounded variations $\varphi:\T \to \R$ and irrational circle rotations $R_\a$ implies a strong DKP property, namely 
\begin{equation}
\label{EqDKstrong} \max_{x\in \T} |A_{q_n}-q_n \int_{\T} A(y) dy|\leq 2V 
\end{equation}
for every denominator of the convergence of $\a$, where $V$ is the total variation of $A$.

More generally we say that $A$ (not necessarily of zero mean) has {\bf the Denjoy-Koksma property} (DKP) if there exist
constants $C, \delta>0$ and
a sequence $n_k\to\infty$ such that 

\begin{equation}
\label{EqDK}
\Prob(|A_{n_k}-n_k \int_{\T^d} A(y) dy|\leq C)\geq \delta. 
\end{equation}
We say that $A$ has {\bf the strong Denjoy-Koksma property} (sDKP) if \eqref{EqDK} holds with $\delta=1.$

Note that if DKP holds and $A$ has zero mean then the set of points where 
$\lim\inf_{n\to\infty} |A_{n}|\leq C$ has positive measure and so by ergodicity of the base map
$W_{\a, A}$ is recurrent.

Later, we will also see  how the DKP can be very helpful  in proving ergodicity of the cylindrical cascades as well as weak mixing of special flows.

The situation with DKP for translations on higher dimensional tori is delicate. 
Of course it holds for almost all $\a$ and for every smooth function by the existence of smooth solutions to the linear cohomological equation \eqref{CoB}. But the DKP also holds above most translations even from a topological point of view. 

\begin{theorem}
\label{ThDKPRes} 
(\cite{etds2}) There is a residual set of vectors in $\a \in \R^d$ 
such that DKP holds above $T_\a$ for every function that is of class $C^4.$ 
\end{theorem}

 In fact, it is non-trivial to construct rotation 
vectors and smooth functions that do not have the DKP. The first construction is due to Yoccoz and it actually provides examples of non recurrent analytic cascades.

\begin{theorem}
\label{ThSmNonRec}
(\cite[Appendix]{Y}) For $d=2$ there exists an uncountable dense set $Y$ of translation vectors and a real analytic function 
$\cA:\T^2\to \C$ with zero mean
such that $W$ is not recurrent.
\end{theorem}

Denote the translation vector by $(\alpha', \alpha'').$ 
The main ingredient in the construction of \cite{Y} is that the denominators,
$q_n'$ and $ q''_n$
of the convergents  of $\a'$ and $\a''$ are alternated, and
 more precisely, they are such that the sequence $... q_n', q''_n, q_{n+1}',
 q''_{n+1}...$ increases exponentially. We will see later that the same construction can be used to 
create examples of mixing special flows with an analytic ceiling function.

Let $Y$ be the set of couples $(\a',\a'') \in {{\R}^2 - \Q}^2$, whose
sequences of best approximations $q_n'$ and $q''_n$ satisfy,
for any $n \geq n_0(\a',\a'')$
$$ q''_n \geq  e^{3q'_n}, \quad
q'_{n+1} \geq e^{3q''_n}. $$

Then  \cite{Y} constructs a real analytic function   
$\cA:\T^2 \to \C$ 
with zero integral such that for almost every $(x,y) \in \T^2$ $|A_n(x,y)| \to \infty$, hence $W_{\a,\cA}$ is not recurrent. 
Note that the set $Y$ as defined above is uncountable and dense in $\R^2$.

\subsubsection{Indicator functions.} 

Now we specify the study of $W_{\a,A}$ to the case where 
$$A=(\chi_{\Omega_j} -  {\rm Vol} (\Omega_j))_{j=1,\ldots,r} \text{ where } \Omega_j \subset X=\T^d \text{ are regular sets.} $$ 
If $d>1$ then the DKP does not seem to be well adapted for proving recurrence in this case (see Questions \ref{q7}--\ref{q11}).

\begin{question} Show that DKP does not hold when $d>1$ and $A=(\chi_{\Omega_j} -  {\rm Vol} (\Omega_j))_{j=1,\ldots,r}$  and the 
$\Omega_j \subset X$ are balls or boxes. 
\end{question}

 There is however another criterion for recurrence which is valid for arbitrary skew products. 

\begin{theorem}
\label{SmallWorld}
Given a sequence $\delta_n=o(n^{1/r})$ the following holds.  

(a) (\cite{CC})
Consider the map $T:X\to X$ preserving a measure $\mu.$
Let $ W(x,y)=(Tx, y+A(x)). $
If there exists 
a sequence $k_n$ such that
$\displaystyle\lim_{n\to\infty} \mu(x : A_{k_n}(x) \leq \delta_{n})=1$
then $W$ is recurrent.

(b) Consider a parametric family of maps $T_\alpha: X\to X,$ $\alpha\in\mathfrak{A}.$ Assume that 
$T_\a$ preserves a measure $\mu_\alpha.$ Let $(\a, x)$ be distributed according to a measure 
$\blambda$ on $\mathfrak{A}\times X$ such that
$d\blambda=d\nu(\a) d\mu_\a(x)$ for some measure $\nu$ on $\mathfrak{A}.$
If $\dfrac{\sum_{n=0}^{N-1} A(T_\a^n x)}{\delta_N}$ has a limiting distribution 
as $N\to\infty$ then
$W_{\a, A}$ is recurrent for $\nu$-almost all $\a.$
\end{theorem}

Note that $T$ is not required to be ergodic. On the other hand if $T$ is ergodic, $r=1$ and $A$ has zero mean, 
then by the Ergodic Theorem $\mu(|A_n/n|>\eps)\to 0$ for any $\eps$ so one can take $k_n=n$ and $\delta_n=\eps_n n$
where $\eps_n\to 0$ sufficiently slowly. Therefore Theorem \ref{SmallWorld} implies Theorem \ref{ThRecDim1}. 

\begin{proof} 
(a) Suppose $B$ is a wondering set (that is, $W^k B$ are disjoint) of positive measure which
is contained in $\{|z|<C\}.$ Let 
$$B_n=\{(x,z)\in B: A_{k_n}(x) \leq \delta_n\}.$$ 
Then $\mu(B_n)\to\mu(B)$ as $n\to\infty$ 
so for large $n$ 
$$\mu(\cup_{1\leq i\leq n} W^{k_i} (B_{k_i})) \geq n \frac{\mu(B)}{2}.$$ 
On the other hand, by assumption $W^{k_i}(B) \subset E_i :=\{ y \leq 2 C + \delta_i \} \subset E_n$ 
if $i \in [1,n].$  
Hence $\mu(\cup_{1\leq i\leq n} W^{k_i} (B_{k_i})) \leq \delta_n^r =o(n)$, 
a contradiction. 

(b) follows from (a) applied to the map $\cT:(\mathfrak{A}\times X)\times \R^r$ given by
$\cT(\a, x, y)=(\a, W_{\a, A}(x,y)).$
\end{proof}

Combining Theorems \ref{ThConvex} and \ref{SmallWorld}(b) we obtain

\begin{corollary}
\label{ThWConv}
If  $\{\Omega_j\}_{j=1,\ldots,r}$ are real analytic and strictly convex and $\frac{(d-1)}{2d}<\frac{1}{r}$
then $W$ is recurrent for almost all $\alpha.$
\end{corollary}

Note that the proof of Theorem \ref{SmallWorld} is not constructive.

\begin{question}
(a) Construct $\a$ and $\{\Omega_j\}_{j=1,\ldots,r}$ for which the corresponding $W$ is non recurrent.

(b) Find explicit arithmetic conditions which imply recurrence.
\end{question}

\begin{theorem}
(\cite{CC})
\label{ThWPoly}
(a) If $\{\Omega_j\}_{j=1,\ldots,r}$ are polyhedra then
$W$ is recurrent for almost all $\alpha.$

(b) There are polyhedra  $\{\Omega_j\}_{j=1,\ldots,r}$ and $\a$ in $\Tor^2$ such that $W$ is transient.
\end{theorem}

Here part (a) follows from Theorem \ref{SmallWorld} and a control on the growth of the ergodic sums. Namely it is proven in \cite{CC} that given any polyhedron $\Omega \subset \T^d$ then for any $\gamma>0$,  for almost every $\a \in \R^d$, we have that $\|A_n\|_2 = O(n^\gamma)$ where $A=\chi_{\Omega} -  {\rm Vol} (\Omega),$ the sums are considered above the translation $T_\a$ and the $L_2$ norm is considered with respect to the Haar measure on $\T^d$.  In the case of boxes, the latter naturally follows from the power log control given 
by Beck's Theorem (see Section \ref{sec.beck}).

The proof of part (b) proceeds by extending the method of
\cite{Y} discussed in Section \ref{DKP}    to the case of indicator functions.

% The next result follows from Theorem~\ref{ThConvex}.
%\margem{I changed from corollary to question because I think it needs some work} 

%\begin{proof}
%Theorem \ref{ThConvex} allows us to apply Theorem \ref{SmallWorld} with the base map
%$\cT: \T^{2d}\to\T^{2d}$ given by $\cT(\alpha, x)=(\alpha, x+\alpha).$ 
%\end{proof}

\begin{question}  
\label{QConvTrans}
Is it true that for a generic choice of $\Omega_j$ as in Question \ref{ThWConv},  
$W$ is transient for almost all $\alpha$ when $\frac{(d-1)}{2d}>\frac{1}{r}$?
\end{question}

An affirmative answer to Question \ref{QLLT} (Local Limit Theorem) 
would give evidence that Question \ref{QLLT} may be true due to Borel-Cantelli Lemma. 
(More precisely, to answer Question \ref{QConvTrans} we need a joint Local Limit Theorem
for ergodic sums of indicators of several sets.)

\begin{question}
Let $\alpha$ be as in Theorem \ref{ThWPoly} (a) or Question  \ref{ThWConv}. 
Does there exist $x$ such that $\lim_{N \to \infty} ||A_N(x)|| =\infty$?
\end{question}

Note that this is only possible if $d>1$ due to the Denjoy-Koksma inequality. On the other hand in any dimension one can 
have orbits which stay in a half space. Such orbits have  been studied extensively (see \cite{PR} and the references wherein).

Another case where recurrence is not easy to establish is that of skew products 
over circle rotations with functions having a singularity such as the examples discussed in Section \ref{ScErgSum}. 
We will come back to this question in the next section.

\subsection{Ergodicity.}
\label{SSErgCasc}
Next we discuss the ergodicity of cylindrical cascades. Here one has to overcome both problems of recurrence
discussed in Section \ref{SSRec} and issues of non-arithmeticity appearing in the study of ergodicity of $S_{\a, A}.$

The ergodicity of $W_{\a, A}$ is usually established using the fact that the sums $A_{N_n}$ are increasingly 
well distributed on $\R^r$ when considered above  any small scale balls in the base and for some {\it rigidity} sequence $N_n$, i.e.  
such that $\| N_n \a \| \to 0$. More precisely,  usual
methods of proving their ergodicity  take into
consideration a sequence of distributions 
\begin{equation}\label{e3}
\left(A_{n_k}\right)_\ast(\mu),\;k\geq1 \end{equation} along some rigidity
sequence $\{n_k\}$ as probability measures on ${\bar \R}^r$ where $\bar \R$ is the one-point
compactification of $\R$. As shown in \cite{Le-Pa-Vo} each point
in the topological support of a  limit measure of (\ref{e3})
is a so called {\it essential value} for $W_{\a,A}.$ Following \cite{Sch} $a\in\R^r$ is called an {\bf essential
value} of $A$ if for each $B\in\T^d$ of positive measure, for
each $\epsilon>0$ there exists $N\in\Z$ such that
$$
\mu(B\cap T^{-N}B\cap[|A_N(\cdot)-a|<\epsilon])>0.$$ 
Denote by
$E(A)$ the set of essential values of $A$. Then the essential value criterion  states as follows

\begin{theorem}\label{p1}
(\cite{Sch},\cite{Aa})

(a) $E(A)$ is a closed subgroup of $\R^r$.

(b) $E(A)=\R^r$ iff $W_{\a,A}$ is ergodic.

(c) If $A$ is integer valued and $E(A)=\Z^r$ then $W_{\a, A}$ is ergodic on $\Tor^d\times \Z^r.$
\end{theorem}

Hence if the supports of the 
probability measures in (\ref{e3}) are increasingly dense on $\R^r$ then $W_{\a,A}$ is ergodic.

The case where $d=r=1$ is the most studied  although there are still some open questions  in this context.
For $d=r=1$ ergodicity is often proved using the Denjoy Koksma Property. 
Indeed, if $A$ is not cohomologous to a constant then $A_N-N \int A$ are not bounded. 
%Then one uses slow divergence of the sums from $A_{q_n}$ that are bounded with 
Let $q_n$ be a best denominator for the base rotation. Pick  $K_n$ which is large but not too large. 
Then $Kq_n$ is still a rigidity time for the translation but 
$A_{Kq_n}$ have sufficiently large albeit controlled oscillations to yield that a given value $a$ in the fibers is indeed an essential value. 

This method is  well
adapted to $A$ whose Fourier transform satisfies
$\hat{A}(n)=\mbox{O}(1/|n|)$, since they display a DKP (see \cite{Lem1}).  Example of such functions are functions of
bounded variation and functions smooth everywhere except for a log symmetric singularity. 

Ergodicity also holds in general for characteristic functions of intervals.

\begin{theorem}
(a) \cite{etds1}  
If $\a$ is Liouville, 
there is a residual set of smooth functions $A$ with zero integral such that the skew product $W_{\a,A}$ is ergodic.

(b) (\cite{Fr-Le})
If $A$ has a symmetric logarithmic singularity then $W_{\a, A}$ is ergodic for all 
irrational $\a.$

(c) (\cite{CK}) If $A=\chi_{[0, 1/2]}-\chi_{[1/2, 1]}$ and $\a$ is irrational then $W_{\a, A}$
is ergodic on $\T^1\times \Z.$ 

(d) (\cite{Oren}) 
If $A=\chi_{[0,\beta]}-\beta$ then $W_{\a, A}$ is ergodic iff $1, \a$ and $\beta$ are rationally
independent.

(e) (\cite{Pas}) If $A$ is piecewise absolutely continuous,  $\int_{\T^1} A(x) dx=0,$ 
$A'$ is Riemann integrable and $\int_{\T^1} A'(x) dx \neq 0$ then $W_{\a, A}$
is ergodic for all $\a\in \R-\Q.$

(f) (\cite{conze})  If $A: \T \to \R^r = (A_1,\ldots,A_r) $ with $A_j= \sum c_{j,i} \chi_{I_{j,i}} - \beta_j$ with $I_{j,i}$ a finite family of intervals, $c_{j,i} \in \Z$ and $\beta_j$ is such that $\int_\T A_j(x)dx=0$ and if the sequence $(\{q_n \beta_1 \},\ldots,\{q_n \beta_r \})$ is equidisitributed on $\T^r$ { as $n\to\infty$}, where $q_n$ is the sequence of denominators of $\a$,  then $W_{\a,A}$ is ergodic. In the case $r=1$, it is sufficient to ask that $(\{q_n \beta\})$ has infinitely many accumulation points, then $W_{\a,A}$ is ergodic.

\end{theorem}

For further results on the ergodicity of cascades defined over circle rotations with step functions as in (f), we refer to the recent work \cite{ConzeP}.

The proofs of (a) and (b) are based on DKP 
and progressive divergence of the sums as explained above. (c)--(e) are treated differently since the ergodic sums take discrete values. For example, the proof of (e)  in the case $r=1$ is based on the fact that $A_{q_n}$ is bounded by DKP and then the hypothesis on $\{q_n \beta\}$ implies that the set of essential values is not discrete, hence it is all of $\R$, 
and the ergodicity follows.

The cases of slower decay of the Fourier coefficients of ${A}$ are more difficult to handle. We have nevertheless a positive result in the particular situation of log singularities. 

\begin{theorem}
\cite{FL} \label{thFL}
If $A$ has (asymmetric) logarithmic singularity then $W_{\a, A}$ is ergodic for almost every $\a.$
\end{theorem}

The delicate point  in Theorem \ref{thFL} is that the DKP does not hold. Indeed, it was shown in \cite{SK} that the special flow above $T_\a$ and under a function that has asymmetric log singularity is mixing for a.e. $\a$. But, as we will see in the next section, mixing of the special flow is not compatible with the DKP. {\it A contrario}  special flows under functions with  symmetric logarithmic singularities are not mixing \cite{Kc2,Lem1} because of the DKP.

In the proof of Theorem \ref{thFL}, one first shows that the DKP 
\eqref{EqDK} holds if the constant $\delta$ is replaced by a sequence $\delta_n$ which decays sufficiently slowly
and then uses this to push through the standard techniques under appropriate arithmetic conditions.

The case of general angles for the base rotation or the case of stronger singularities are harder and all questions are still open.

%The second case is indeed sensitively different from the first one
%for the following reason that we will further comment in Section \ref{sec.skew} : while the distribution of the ergodic sums related to a zero average function with symmetric log singularity   along a special subsequence of integers are bounded on a positive measure  proportion of points of the circle, the sums for power like singularities or asymmetric log are concentrated  at infinity. This is why for example  the special flow over $R_\a$ and under a smooth
%function with at least one  power like singularity or with asymmetric log singularities is
%mixing \cite{Kc1,SK,bsmf,Kc3} while the one under a smooth function
%with symmetric logarithmic singularities is not \cite{Kc2,Lem1}. For the same reason, it is relatively easy to show that cylindrical cascades above irrational rotations with symmetric logarithm singularities are ergodic  \cite{Fr-Le}, while ergodicity of the cylindrical cascades for asymmetric logarithms is harder \cite{FL} and still unknown for power like singularities. 
%{\color{red} Add reference and result for compact extensions}

\begin{question} Are there examples of ergodic cylindrical cascades with smooth functions having power like singularities?
\end{question}

Conversely, we may ask the following 
\begin{question} Are there examples of non ergodic cylindrical cascades with smooth functions having  non symmetric logarithmic  or (integrable) power singularities?
\end{question}

The study of ergodicity when $d>1$ and $r>1$ is more tricky essentially because of the absence of DKP. 

%Again, one has to separate the case of smooth observables and non smooth ones. Indeed, for $\a$ Diophantine and $A$ smooth, $W_{\a,A}$ is smoothly conjugated to $W_{\a,0}$ since all the coordinates of $A$ are smooth coboundaries above $\a$. Then obviously ergodicity fails. So for smooth observables only the Liuoville case is interesting. 

For smooth observable, only the Liouville frequencies are interesting. The ergodic sums above such frequencies tend to stretch at least along a subsequence of integers. And this stretch usually occurs gradually and independently in all the coordinates of $A$ hence a positive answer to the following question is expected. 

\begin{question}
Show that for any Liouville vector $\a$, 
there is a residual set of smooth functions $A$ with zero integral such that the skew product $W_{\a,A}$ is ergodic.
\end{question}

As we discussed in the proof of Theorem \ref{thFL}, the cylindrical cascade on $\T \times \R$ with a function $A$ having an asymmetric logarithmic singularity is ergodic for almost every $\alpha$ although the ergodic sums $A_N$ 
above $T_\a$ concentrate at infinity as $N \to \infty$. The slow divergence of these sums that compare to $\ln N$ (see Question \ref{q.log}) plays a role in the proof of ergodicity. 
The logarithmic control of the discrepancy relative to a polyhedron
%\margem{isn't it a box rather than polyhedron otherwise why would one need Conze in the recurrence theorem? Of course one expects Beck to hold for polyhedron and we can keep polyhedron in the question}  
(see Theorems \ref{th.beck}, \ref{ThBox} and \ref{ThWPoly}) motivates the following question.

 \begin{question} \label{q.ergodicity}
  Is it true that for (almost) every polyhedra  $\Omega_j \subset \T^d$, $j=1,\ldots,r$, and  $A=(\chi_{\Omega_j} -  {\rm Vol} (\Omega_j))_{j=1,\ldots,r}$, the cascades $W_{\a,A}$ are ergodic? 
  \end{question} 

We note that the answer is unknown even for boxes with $d=2$ and $r=1$.

%Special flows with power like singularities of the form $|x|^{-a}$, $a \in (0,1)$ are mixing \cite{Kc1}. The ergodic sums are too large (see Question \ref{q.a}) to obtain ergodicity of the cylindrical cascades {\it via} the essential value criterion and the slow convergence to $0$ of the sets where the sums are bounded. 

\subsection{Rate of recurrence.}
Section \ref{SSErgCasc} described several situations where $W_{\a, A}$ is ergodic. However for infinite measure preserving transformations
the (ratio) ergodic theorem does not specify the growth of ergodic sums. Rather it shows that for any $L^1$ functions
$B_1(x,y), B_2(x,y)$ with $B_2>0$ we have
\begin{equation}
\label{RatETh}
\frac{\sum_{n=0}^{N-1} B_1(W_{\a, A}^n (x,y))}{\sum_{n=0}^{N-1} B_2(W_{\a, A}^n (x,y))}
\to
\frac{\iint B_1(x,y) dx dy}{\iint B_2(x,y) dx dy}.  
\end{equation}
In fact (\cite{Aa}) there is no sequence $a_N$ such that
\begin{equation}
\label{RescaledSumCasc}
\frac{\sum_{n=0}^{N-1} B_1(W_{\a, A}^n (x,y))}{a_N}
\end{equation}
converges to 1 almost surely. On the other hand, one can try to find $a_N$ such that 
\eqref{RescaledSumCasc} converges in distribution. By \eqref{RatETh} it suffices to do it 
for one fixed function $B.$ For example one can take $B=\chi_{B(0,1)}.$ This motivates the following question.

\begin{question}
Let $\alpha$ be as in Theorem \ref{ThWPoly} (a) or Question \ref{ThWConv}. 
How often is $||W^N||\leq R$?
\end{question}

So far this question has been answered only in a special case. Namely,
let $d=r=1,$ $Z_N(x)=\sum_{n=0}^{N-1}\left[\chi_{[0, 1/2]}(x+n\alpha)-1/2\right].$
Denote $\LL_N=\text{Card}(n\leq N: Z_n=0).$

\begin{theorem}
\label{ThChi2}
\cite{ADDS}
If $\alpha$ is a quadratic surd then there exists a constant $c=c(\alpha)$ such that
$\frac{\sqrt{\ln N} }{c N} \LL_N $ converges to $e^{-\sN^2/2} .$
\end{theorem} 

Similar results have been previously obtained by Ledrappier-Sarig for abelian covers of compact hyperbolic
surfaces (\cite{LS}). The fact that  the correct normalization is $N/\sqrt{\ln N}$ was established in \cite{AaK}.

\begin{question}
Extend Theorem \ref{ThChi2} to the case when $1/2$ is replaced by

(a) any rational number; 

(b) any irrational number,  (in which case one needs to replace $\{A_N=0\}$ by $\{|A_N| \leq 1\}$). 
\end{question}

\begin{question}
\label{RWTypAlpha}
What happens for typical $\alpha?$ 
\end{question}
Note that in contrast to Theorem \ref{ThChi2}, 
$\LL_N\sim \dfrac{N}{\ln N}$ 
(rather than $\LL_N\sim \dfrac{N}{\sqrt{\ln N}}$) is expected in view of
Kesten's Theorem \ref{ThKesten}. Ideas of the proof of Theorem \ref{ThChi2}
 will be described in Section \ref{SSAppl}.

\section{Special flows.}
\label{ScFlows}
\subsection{Ergodic integrals.}
In this section we consider {\bf special flows} above $T_\a$   which will be denoted  $T^t_{\a,A}.$ 
Here $A(\cdot) >0$ is called the {\bf ceiling function} and   the flow is given by  
\begin{eqnarray*}
\T^d \times \R / \sim  &  \rightarrow &  \T^d \times \R / \sim  \\
 (x,s) & \rightarrow & (x,s+t), \end{eqnarray*}
where $\sim$ is the
identification 
\begin{equation}
\label{FlowSpace}
(x, s + A(x)) \sim (T_\a(x),s). 
\end{equation}
Equivalently the flow is defined for  $t+s \geq 0$ by 
$$T^t(x,s) = (x+n\a, t+s-A_n(x))$$ 
   where $n$ is the unique integer such that 
\begin{equation}
\label{D-C}
   A_{n}(x) \leq t+s < A_{n+1}(x).
\end{equation}    
Since $T_\a$ preserves a unique probability measure $\mu$ then the special flow will preserve a unique probability measure that is the normalized product measure of $\mu$ on the base and the Lebesgue measure on the fibers. 

Special flows above ergodic maps are always ergodic for the product measure constructed as above. The interesting feature of special flows is that they can be more "chaotic" then the base map, displaying  properties such as weak mixing or mixing even if the base map does not have them. Actually any map of a  very wide class of zero entropy measure theoretic transformations, so called Loosely Bernoulli maps, 
are isomorphic to sections of special flows above any irrational rotation of the circle with a continuous ceiling function
(see \cite{ORW}).

If $A=\beta$ is constant then $T_{\a, A}$ is the linear flow on $\T^{d+1}$ with frequency vector $(\a, \beta).$
Thus special flows $T^t_{\a,A}$ can be viewed as time changes of translation flows on $\T^{d+1}.$ In particular, 
if we consider the linear flow
on $\T^{d+1}$ and multiply the velocity vector by  a smooth {\it non-zero} function $\phi$ we get a special flow with 
a smooth ceiling function~$A.$

\subsection{Smooth time change.} We recall that a translation flow frequency $v \in \R^d$ is said to be Diophantine if there exists $\sigma,\tau>0$ such that $||(k,v)||\geq C |k|^{-\sigma}$ for every $k\in \Z^d$. Hence a translation vector $(1,v) \in \R^{d+1}$ is Diophantine (homogeneous Diophantine or Diophantine in the sense of flows) if and only if $v$ is a Diophantine vector in the sense 
of \eqref{DefDio}.

\begin{theorem}  \label{th.kol} \cite{Kol1} Smooth non  vanishing time changes of translation flows with a Diophantine frequency vector are smoothly conjugated to translation flows. 
\end{theorem}

\begin{proof} Let $v$ be a constant vector field on $\T^{d+1}$. We suppose WLOG that $v=(1,\a)$. Let $u(x)$ be a smooth function on the torus 
and $\dot{x}=u(x)v$. Then, making a change of
variables $y=T^{\phi(x)}_v(x)$ we obtain the equation
$\dot{y}=(\phi+\partial_v \phi)(y) u(y) v.$ The equation for $y$ is linear if $\phi+\partial_v \phi=\dfrac{c}{u}.$
Passing to Fourier series, this equation can be solved if $c=\int \phi(x) dx \left(\int \dfrac{dx}{u(x)}\right)^{-1}$ and $v$ is such that $|1+(k,v)| |\geq C |k|^{-\sigma}$ for every $k \in \Z^{d+1}$ which is equivalent to $\a$ Diophantine as in  \eqref{DefDio}. 

One can also see this fact at the level of the special flow $T^t_{\a,A}$ associated to $\dot{x}=u(x)v$. Namely, making a change of variables 
$(y,s)=T_{\a, A}^{B(x,t)} (x,t)$ transforms $T_{\a, A}$ to $T_{\a, D}$ with
$$D(x)=A(x)+B(x+\a,0)-B(x,0)$$ so one can make the LHS constant provided $\a$ is Diophantine.
Finally, the similarity between linear and nonlinear flows in the Diophantine case is also reflected in
\eqref{D-C} since for Diophantine vectors $\a$ $A_n=n\int A(x) dx+O(1).$  \end{proof}

An interesting question is that of deviations of ergodic sums above time changed linear flows.  In fact, the case of linear flows is already non trivial and can be studied by the methods described in Section \ref{SSProofsD}. More precisely, as for translations the interesting case occurs when the function under consideration has singularities,
for example, for indicator functions.

Namely, given a set $\Omega$ let

\begin{equation}
\bD_\Omega(r,v,x,T)  = \int_{0}^{T} \chi_{\Omega_r}(T_v^t x) dt - T \text{Vol}({\Omega_r})
\end{equation}
where $T_v^t$ denotes the linear flow with velocity $v.$ 

We assume that $(x,v, r)$ are distributed according to a smooth density.

\begin{theorem}
\label{ThFlows}
(\cite{DF1, DF2}).
Suppose that $\Omega$ is analytic and strictly convex.

(a) If $d=2$ then 
$\bD_\Omega(r,v,x,T)$ 
converges in distribution.

(b) If $d=3$ then $\frac{\bD_\Omega(r,v,x,T)}{\ln T}$ converges to a Cauchy distribution.

(c) If $d\geq 4$ then 
$\frac{\bD_\Omega(r,v,x,T)}
{ r^{\frac{d-1}{2}} T^{\frac{d-3}{2(d-1)}} }$
has limiting distribution as $T\to \infty.$

(d) For any $d \in \N$, if $\Omega$ is a box then $\bD_\Omega(r,v,x,T)$ 
converges in distribution.
\end{theorem}

The proof of  Theorem \ref{ThFlows} is similar to the proofs of Theorems \ref{ThConvex} and \ref{ThBox} and Corollary \ref{CorErgSmooth}.

\begin{corollary}
Theorem \ref{ThFlows} remains valid for time changes $T_{u(x)v}$ where $u(x)$ is a fixed smooth positive function
and $v$ is random as in Theorem \ref{ThFlows}.
\end{corollary}

\begin{proof}
To fix our ideas let us consider the case where $\Omega$ is analytic and strictly convex. 
Note that $T_{uv}^t x=T_v^{\tau(x,t)}$ where the by the above discussion the function 
$\tau$ satisfies 
$$\tau(t,x)=at+\eps(t,x,v) \text{ where }a=\left(\int\frac{dx}{u(x)}\right)^{-1} $$ 
and  $\eps(t,x,v)$ is bounded for
almost all $v$ uniformly in $x$ and $t.$ Accordingly it suffices to see how much time is spend
inside $\Omega_r$ for the linear segment of length $at.$ 

 Next if the linear flow stays inside 
$\Omega_r$ during the time $[t_1, t_2]$ then the time spend in $\Omega_r$ by the orbit of
$T_{uv}$ equals to
$ \int_{t_1}^{t_2} \frac{d t}{u(x(t))}.$ Thus we need to control the following integral for linear flow
$$  \tilde{\bD}_\Omega(r,v,x,T)  = \int_{0}^{T} \frac{\chi_{\Omega_r}(T_v^t x)}{u(T_v^t x)}  dt - T 
\int_{\Omega_r} \dfrac{dx}{u(x)}. $$
However the Fourier transform of $\frac{\chi_{\Omega_r}(x)}{u(x)}$ has a similar asymptotics at
infinity as the Fourier transform of $\chi_{\Omega_r}(x)$ (see \cite{SW}) so the proof of 
the Corollary proceeds in the same way as the proof of Theorem \ref{ThFlows} in \cite{DF1}.
\end{proof}

Up to now, we were interested in smooth time change of linear flows with typical frequencies. We will further discuss smooth time changes for special frequencies in Section \ref{sec.mixing} devoted to mixing properties.

{
\subsection{Time change with singularities.} 

If the time changing function of an irrational flow has zeroes then the ceiling function of the corresponding special flow has 
poles. In this case the smooth invariant measure is infinite. In the case of a unique singularity, we have that the time changed flow is uniquely ergodic with the Dirac mass at the singularity the unique invariant probability measure:

\begin{proposition}
\label{PrDirac}
Consider a flow $T^t$ given by a smooth time change of an irrational linear flow obtained by multiplying the constant vector field by a function 
which is smooth and non zero
everywhere except for one point $x_0$, then for any continuous function $b$ and any $x$
$$\lim_{t\to\infty} \frac{1}{t} \int_0^t b(T^ux) du=b(x_0). $$
\end{proposition}

\begin{proof}
To simplify the notation we assume that the time change preserves the orientation of the flow. We use the representation as a special flow $T^t_{\a,A}$ with $A$ having a pole. 

It suffices to prove this statement in case $b$ equals to $0$ in a small neighborhood of $x_0.$
In that case we have 
\begin{equation}
\label{Int-Sum}
\int_0^t b(T^u_{\a, A} (x, s)) du=B_{n(t)}(x)+O(1)
\end{equation}
where $B(x)=\int_0^{A(x)} b(x, s) ds$ and $n(t)$ is defined by \eqref{D-C}.
If $b$ vanishes in a small neighborhood of $x_0$ then $B$ is bounded and so 
$|B_{n(t)}|\leq C n(t).$ Therefore it suffices to show that $\frac{n(t)}{t}\to 0$
which is equivalent to $\frac{A_n}{n}\to \infty.$ Let $\tilde{A}$ be a continuous function which is 
less or equal to $A$ everywhere. Then
$$\lim\inf \frac{A_n}{n}\geq \lim_{n\to\infty} \frac{\tilde{A}_n}{n}=\int \tilde{A}(x) dx. $$
Since $\int A(x) dx=\infty$ we can make $\int\tilde{A}(x) dx$ as large as possible proving our claim.  
\end{proof}

\begin{question}
In the setting of Proposition \ref{PrDirac} describe the deviations of ergodic integrals from $b(x_0).$
\end{question}

\begin{question}
Consider the case where the time change has finite number of zeroes $x_1, x_2,\dots, x_m.$
In that case all limit measures are of the form $\sum_{j=1}^m p_j \delta_{x_j}.$ Which 
$p_j$ describe the behavior of Lebesgue-typical points?
\end{question}

In view of the relation \eqref{Int-Sum} these questions are intimately related to 
Theorems \ref{ThPole2} and  \ref{ThMer} and Questions \ref{Q1xasym}, \ref{QM2-NI}  and \ref{q6} from Section \ref{ScErgSum}.
}

\begin{figure}[htb]
    \centering
    \resizebox{!}{4cm}{\includegraphics{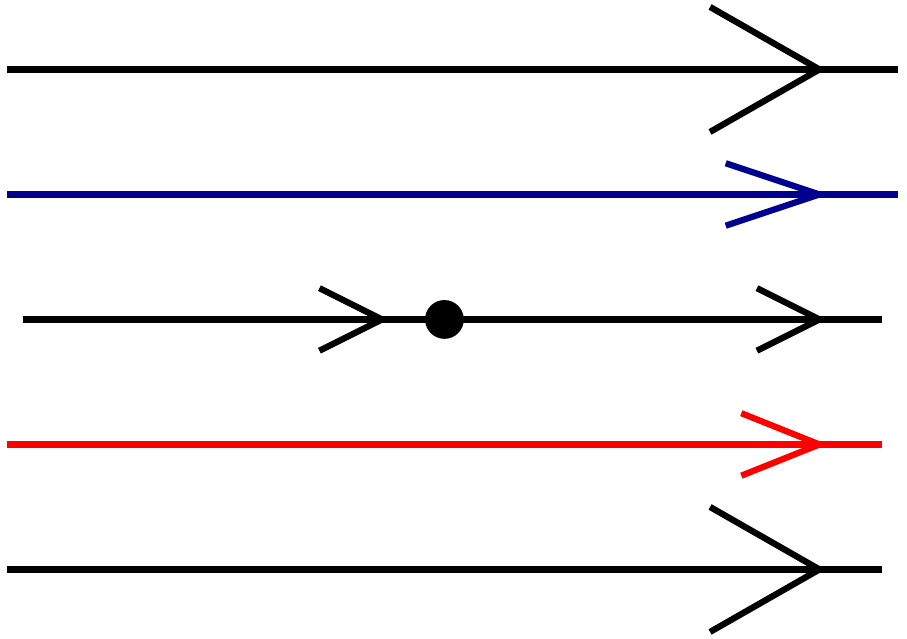}}
    \caption{Kocergin Flow is topologically equivalent to the area preserving flow shown on Figure \ref{FigSaddle}
with separatrix loop removed. The rest point is responsible for the shear along the orbits.}
    \label{FigKoc}
\end{figure}

If one is interested in flows with singularities preserving a finite non-atomic measure then 
the simplest example can be obtained by plugging (by smooth surgery) in the phase space of the minimal 
two dimensional linear flow
an isolated singularity coming from a Hamiltonian flow in
$\R^2$ (see Figure \ref{FigKoc}).
The so called Kochergin flows thus obtained preserve
besides the Dirac measure  at the singularity a measure that is
equivalent to Lebesgue measure \cite{Kc1}. 
As it was explained in Section \ref{ScErgSum} Kochergin flows model smooth area preserving 
flows on $\T^2.$
These flows still have $\T$ as a
global section with a minimal rotation for the return map, but the
slowing down near the fixed point produces a singularity for the
return time function above the last point where the section
intersects the incoming separatrix of the fixed point. The strength of the singularity depends on
how abruptly the linear flow is slowed down in the neighborhood of
the fixed point. A mild slowing down, or mild {\it shear}, is typically represented by
the logarithm while stronger singularities such as $x^{-a}, a\in (0,1)$ are also possible. 
Powerlike singularities appear naturally in the study of area preserving flows with degenerate fixed points.
We shall see below that dynamical properties of the special flows are quite different for 
logarithmic and power like singularities.

\begin{question}
What can be said about the deviations of the ergodic sums above  Kocergin flows?
\end{question}

\subsection{Mixing properties.} \label{sec.mixing} 
 %is also obtained {\it via} a diffusive distribution of $A_{N_n}$ (mod $(1)$ in this case) for a sequence  $N_n$ such that $\| N_n \a \| \to 0$. 
 
 We give first a classical criterion for weak mixing of special flows. 
%that can be found for example in \cite{}. 
% {\color{red} I COULD NOT FIND THIS RESUL IN \cite{cfs}. COULD YOU GIVE A MORE PRECISE REFERENCE}
 Its proof  is similar to the proof of the ergodicity criterion for skew products given by Proposition \ref{erg}.
 \begin{proposition} 
(\cite{VN}) %[Proposition 6.2]{PP})
$T_{\a, A}$ is weak mixing iff
 for any $\lambda \in \R^*$, there are no measurable solutions to the multiplicative cohomological equation  
 \begin{equation}
 \label{MultCoB}
 e^{i2 \pi \lambda A(x) }= \psi(x+\a)/\psi(x). \end{equation}
 \end{proposition}
Indeed if $h(x,t)$ is the eigenfunction when for almost all $x$
the function
$h(x,t) e^{-\lambda t}$ takes the same value $\psi(x)$ for almost almost all $t.$ Then
\eqref{MultCoB} follows from the identification \eqref{FlowSpace}.

\begin{theorem} (\cite{etds1}) 
\label{3} 
If the vector $\a \in {\R}^d $ is not $\beta$-Diophantine then there exists a dense
$G_{\delta}$ for the $C^{\beta+d}$ topology, of functions $\varphi \in
C^{\beta+d} ( {\T}^d,{\R}^*_+), $ such that the special flow constructed over $T_{\a}$ with
the ceiling function $\varphi $ is weak mixing.
\end{theorem}

This result is optimal 
since smooth time changes of linear flows with Diophantine vectors $\a,$
are smoothly conjugated to the linear flow 
%\cite{Kol1}, 
and, hence, are not weak mixing.

Mixing of  special flows is  more delicate to establish since one needs to have uniform distribution on increasingly large scales in $\R^+$ of the sums $A_N$ for {\it all} integers $N \to \infty$, and this above arbitrarily small sets of the base space. Indeed mixing of special flows above non mixing base dynamics is in general proved as follows: 
%{\color{red} PICTURE???}
if the ergodic sums $A_N$ become as $N \to \infty$ uniformly stretched (well distributed inside large intervals of $\R_+$) above small sets, the image by the special flow at a large time $T$ of these small sets decomposes into long strips  that are well distributed in the fibers due to uniform stretch and well distributed in projection on the base because of ergodicity of the base dynamics (see Figure \ref{FigFlowMix}).

\begin{figure}[htb]
    \centering
    \resizebox{!}{6cm}{\includegraphics{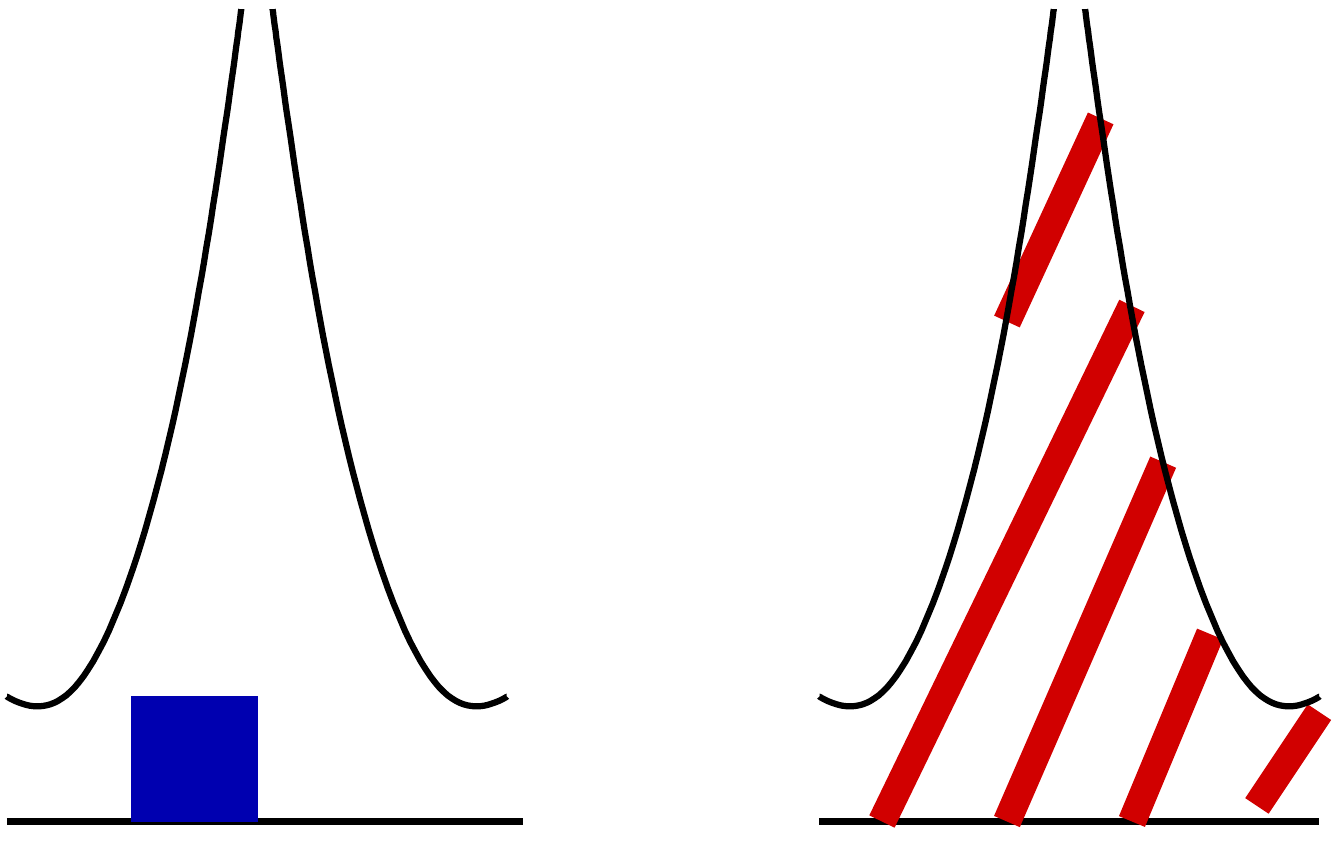}}
\caption{Mixing mechanism for special flows: the image of a rectangle is
a union of long narrow strips which fill densely the phase space.} 
%    \resizebox{!}{8cm}{\includegraphics{FlowMix.pdf}}
%    \caption{}{\it Mixing due to uniform stretch of the ergodic sums.}  the image of a small interval $J$ in the $x$ direction at a time $t$ that is large and far from the denominators of $\a'$, consists of
%a lot of almost linear curves whose projection on the basis lies
%along a piece of a trajectory under the translation $R_{\a',\a''}$. For intervals in the $y$ direction, the same phenomenon happens as time $t$ is far from the denominators of $\a''$.}
    \label{FigFlowMix}
\end{figure}

The delicate point however is to have uniform stretch for all integers $N \to \infty$. 
In particular the following result has been essentially  proven in \cite{Kc0}. 
\begin{theorem}
\label{ThMixVsDKP}
If $A$ has DKP then $T_{\a, A}^t$ is not mixing. 
\end{theorem}
\begin{proof}
If $A$ has the DKP then there is a set $\Omega$ of positive measure on which  \eqref{EqDK} holds for positive density of $n_k.$
By passing to a subsequence we can find a set $I$ of positive measure, a sequence $\{t_k\}$  and a vector $\beta$ such that
on $\Omega$ $|A_{n_k}-t_k|<C$ and $\alpha n_k\to \beta.$ Pick a small $\eta.$ 
$$\Omega_i=\cup_{0\leq t\leq \eta} T_{\a, A}^t [I\times \{0\}], \quad
\Omega_f=\cup_{0\leq |t|\leq C+\eta} T_{\a, A}^t [(I+\beta) \times \{0\}]. $$
By decreasing $I$ if necessary we obtain that those sets have measures strictly between $0$ and $1.$
On the other hand it is not difficult to see from the definition of the special flow that 
$\mu(T_{\a, A}^{t_k}\Omega_i\cap \Omega_f)\to \mu(\Omega_i)$ contradicting the mixing property.
\end{proof}

In particular the flows with ceiling functions $A$ of bounded variation or functions with symmetric log 
singularities are not  mixing.

In fact, since the sDKP holds for any minimal circle diffeomorphism, it follows from \eqref{D-C} 
and \eqref{Int-Sum}
that any smooth flow on $\T^2$ without cycles or fixed points is not topologically mixing. 
We leave this as an exercise for the reader.

The first positive result about mixing of special flows is obtained in \cite{Kc1}.

\begin{theorem}
\label{ThDegSF}
If $\a\in\R-\Q$ and $A$ has (integrable) power singularities then $T_{\a, A}$ is mixing.
\end{theorem}

The reason why the case of power singularities is easier than the logarithmic case (corresponding to non-degenerate flows
on $\T^2$) is the following. The standard approach for obtaining the stretching of ergodic sums is to control 
$\frac{\partial A_n}{\partial x}=\left(\frac{\partial A}{\partial x}\right)_n$ 
For $A$ as in theorem \ref{ThDegSF}, 
$\frac{\partial A}{\partial x}$ has singularities of the type $x^{-a}$ with $a>1.$
In this case the main contribution to ergodic sums comes from the closest encounter with the singularity
(cf. Theorem \ref{ThPole2}) making the control of the stretch easier.
Moreover, the strength of the singularity allows to obtain speed of mixing estimates.

\begin{theorem}(\cite{bsmf}) If $\alpha$ is Diophantine and $A$ has a (integrable) power singularity then 
$T_{\a, A}^t$ is power mixing.
\end{theorem}

More precisely, there exists a constant $\beta=\beta(\alpha)$ such that if $R_1, R_2$ are rectangles in $\T\times \R$ then
\begin{equation}
\label{PowerMix}
\left|\mu(R_1 \cap T^t R_2)-\mu(R_1) \mu(R_2)\right|\leq C t^{-\beta}. 
\end{equation}
The exponent $\beta$ in \cite{bsmf} seems to be non optimal.

\begin{question}
For $\alpha$ Diophantine find the asymptotics of the LHS of \eqref{PowerMix}.
\end{question}

It is interesting to surpass 
the threshold $\beta=1/2$.  In particular, one would like to answer the following question. 

\begin{question} \cite{LemEnc}
Can a smooth area preserving flow on $\T^2$ have Lebesgue spectrum?
\end{question}

 On the other hand for logarithmic singularities there might be cancelations in ergodic sums
of $\frac{\partial A}{\partial x},$ making the question of mixing more tricky.

\begin{theorem}
\label{ThLogMix}
Let $A$ be as in Question \ref{q.log}.

(a) (\cite{Kc2}) If $\sum_j a_j^+=\sum_j a_j^-$ then $T_{\a, A}^t$ is not mixing for any $\a\in\R-\Q.$

(b) (\cite{SK, Kc3}) If $\sum_j a_j^+ \neq \sum_j a_j^-$ then $T_{\a, A}^t$ is mixing for almost every $\a\in\R-\Q.$

(c) (\cite{Kc3}) If $a_j^+-a_j^-$ has the same sign for all $j$ then  $T_{\a, A}^t$ is mixing for each $\a\in\R-\Q.$
\end{theorem} 

\begin{question} 
(\cite{Kc4}) Does the condition that $\sum_j a_j^+ \neq \sum_j a_j^-$ imply $T_{\a, A}^t$ is mixing {\bf for every $\a\in\R-\Q$}?
\end{question}

\begin{question}
(\cite{Kc4}) Under the conditions of Theorems \ref{ThDegSF} and \ref{ThLogMix} is
$T_{\a, A}^t$ mixing of all orders?
\end{question}

In higher dimensions much less is known. Note that  for smooth ceiling functions Theorems \ref{ThDKPRes}, \ref{th.kol} and
\ref{ThMixVsDKP}
precludes mixing for a set of rotation vectors of full measure that also contains a residual set.

%For example, the Denjoy-Koksma inequality stating that $A_{q_n} - q_n \int_\T A$ is bounded if $A$ is a real valued bounded variation function on $\T$ and $q_n$ is the sequence of denominators of the best rational approximations of $\a$ precludes mixing of $T^t_{\a,A}$. Indeed, suppose to simplify that  $A_{q_n} - q_n \int_\T A$ is very small on a set of points $E_n$ of the base that has a lower bounded measure, then the image of the set $\Delta =\T \times [0,\delta]$ for a small $\delta > 0$  satisfies $T^{q_n} \Delta \cap \Delta \gg \delta^2$, which contradicts mixing. The case  $A_{q_n} - q_n \int_\T A$ bounded on $E_n$ is similar.  In the case of $A \in C^1(\T, \R_+^*)$, we even get that $T^{q_n} \to {\rm Id}$ in the uniform topology (we say that the flow is rigid). As a consequence, we get that any $C^1$ reparametrization of a translation flow of $\T^2$ is not mixing (it is actually rigid) \cite{Kc0}. 

The following was shown in \cite{etds2}. Recall the definition of the set $Y$ used in Theorem \ref{ThSmNonRec}.
Define the following real analytic  complex valued  function on ${\T}^2$:             
$$ \cA(x,y)=  \left( \sum_{k=2}^{\infty} { e^{i2 \pi kx}
\over e^{k}} +
\sum_{k=2}^{\infty} {e^{i2 \pi ky} \over e^{k}} \right). $$

\begin{theorem} \label{nnnspecial} For any $(\a',\a'') \in Y$, the special flow constructed over the translation 
$T_{\a',\a''}$ on
${\T}^2$, with the ceiling function $1+\mathrm{Re}\cA$ is mixing.
\end{theorem}

 Because of the
disposition of the best approximations of $\a'$ and $\a''$ the  ergodic
sums $\varphi_m$ of the function $\varphi$, for any  $m$ sufficiently
large, will be always stretching (i.e. have big derivatives), in one
or in the other of the two directions, $x$ or $y$, depending on
whether $m$ is far from $\{q'_n\}$ or far from $\{q''_n\}$. And this stretch
will increase when $m$ goes to infinity. So when time goes from $0$ to
$t$, $t$ large, the image of a small typical interval $J$ from the basis
${\T}^2$ (depending on $t$ the intervals should be taken along the $x$
or the $y$ axis) will be more and more distorted and stretched  in the
fibers' direction, until the image of $J$ at time $t$ will consist of
a lot of almost vertical curves whose projection on the basis lies
along a piece of a trajectory under the translation $T_{\a',\a''}$. By
unique ergodicity these projections become more and more uniformly
distributed, and so will $T^t(J)$. For each $t$, and except for
increasingly small subsets of it (as function of $t$), we will be able
to cover the basis with such ``typical'' intervals. Besides, what is
true for $J$ on the basis is true for $T^s(J)$ at any height $s$ on the
fibers. So applying Fubini Theorem in two directions, first along the other direction on the basis (for a time $t$ all typical intervals are in the same direction),
and second along the fibers, we will obtain the asymptotic uniform
distribution of any measurable subset, which is, by definition, the
mixing property. %(see Figure \ref{FigFlowMix}).

\begin{question} Are the flows obtained in Theorem \ref{nnnspecial} mixing of all orders?
\end{question}

\begin{question} For which vectors $\a \in \R^d$, there exist special flows above $T_\a$ with smooth functions $A$ such that $T^t_{\a,A}$ is mixing?
\end{question}

The foregoing discussion demonstrates that both ergodicity of cylindrical cascades and mixing of special flows require a 
detailed analysis of ergodic sums \eqref{ErgSum}. However, the estimates needed in those two cases are quite different 
and somewhat conflicting. Namely, for ergodicity we need to bound from below the probability
that ergodic sums hit certain intervals,
while for mixing one needs to rule out too much concentration. For this reason it is difficult to construct functions $A$ such that
$W_{\a, A}$ is ergodic while $T_{\a, c+A}^t$ is mixing. In fact, so far this has only been achieved for smooth functions with asymmetric logarithmic singularities.
However, it seems that in higher dimensions there is more flexibility so such examples should be more common.

\begin{question} \label{q.mixing}
  Is it true that for (almost) every polyhedron $\Omega \in \T^d$, $d\geq 2$,  and almost every $a>0$, and almost every $\a \in \T^d$, the special flow  above $\a$ and under the function $a+ \chi_\Omega$ is mixing?
\end{question} 

Note that a positive answer to both this question and Question \ref{q.ergodicity} will give a large class of 
interesting examples  where ergodicity of  
$W_{\a, A}$ and mixing for $T_{\a, c+A}$ (for any $c$ such that $c+A>0$) hold simultaneously.

\begin{question}
Answer Questions  \ref{q.ergodicity} and \ref{q.mixing}   in the case $\Omega$ is a strictly convex analytic set. 
\end{question}

%Power like singularities  of the form $|x|^{-a}$ with 
%$a \in (0,1)$ are interesting because they give the simplest examples of mixing flows on $\T^2$ that are time changes of linear flows with a stopping point \cite{Kc1}. In \cite{bsmf} the control of the speed of growth of the ergodic sums for a singularity of the type $|x|^{-a}$, $a \in (0,1)$ allowed to compute the speed of mixing of the special flows above irrational rotations 
%of adequate Diophantine type. The speed was shown to be bounded by $t^{-b}$ for some $b$ related to $a$ and to the Diophantine property of the base rotation. The method used there is bound to yield a speed slower than $1/\sqrt{t}$, which is a threshold that has to be bypassed if one wants to obtain an absolutely continuous spectral measure for the special flow using speed of mixing. 

%\begin{question} Can the speed $1/\sqrt{t}$ be superseded for special flows above translations?
%\end{question}

%As we see in the following question, as $a$ tends to $1$ the size of the fluctuations (or stretching) of the ergodic sums become large which can accelerate mixing, but in the same time the measure of the neighborhood of the stopping point becomes more important which can have a downside on the speed of mixing. 

\subsection{An application.}
\label{SSAppl}
Here we show how the geometry of special flows above cylindrical cascades can be used to study the ergodic sums. 

\begin{proof}[Proof of Theorem \ref{ThChi2}.]

\begin{figure}[htb]
    \centering
    \resizebox{!}{8cm}{\includegraphics{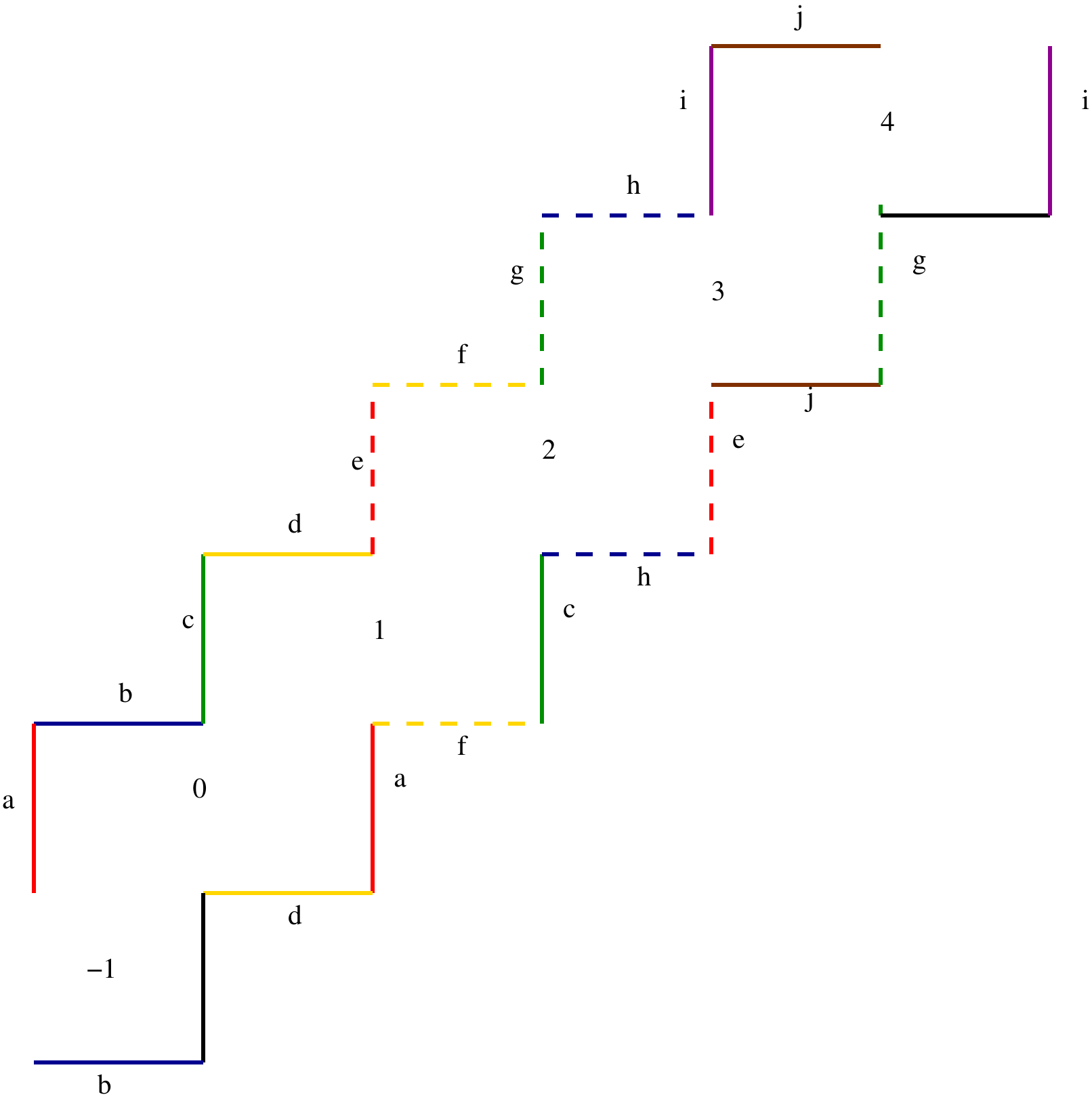}}
    \caption{Staircase surfaces. The sides marked by the same symbol are identified.}
    \label{FigStair}
\end{figure}

The proof uses the properties of the staircase surface $St$ shown on Figure \ref{FigStair}. The staircase is
an infinite pile of $2\times 1$ rectangles so that the left bottom corner of the next rectangle is attached to the center of the
top of the previous one. The sides which are differ by two units in either horizontal or vertical direction are identified.
We  number all the rectangles from $-\infty$ to $+\infty$ as shown on Figure \ref{FigStair}. There is a translational symmetry
given by $G(x,y)=(x+1, y+1)$ and $St/G$ is a torus. We shall use coordinates $\brp=(p,z)$ on the staircase
where $p$ are coordinates on the torus which is the identified with rectangle zero and $z\in\Z$ is the index of rectangle.
Thus we have $(p,z)=G^z (p, 0).$ 

The key step in the proof is an observation of \cite{HHW} that $St$ is a Veech surface.
Namely, given $A\in SL_2(\integers)$ such that $A\equiv I \mod 2$ there exists
unique automorphism $\phi_A$ of $St$ which 
commutes with $G,$ fixes the singularities of $St,$ has derivative $A$ at the non-singular points
and has drift $0.$ That is, in our coordinates 
\begin{equation}
\label{ZAuto}
\phi(p,z)=(Ap, z+\tau(p))
\end{equation}
 and the drift condition means that 
$\int_{\T^2} \tau(p) dp=0.$

\begin{figure}[htb]
    \centering
   \resizebox{!}{5cm}{\includegraphics{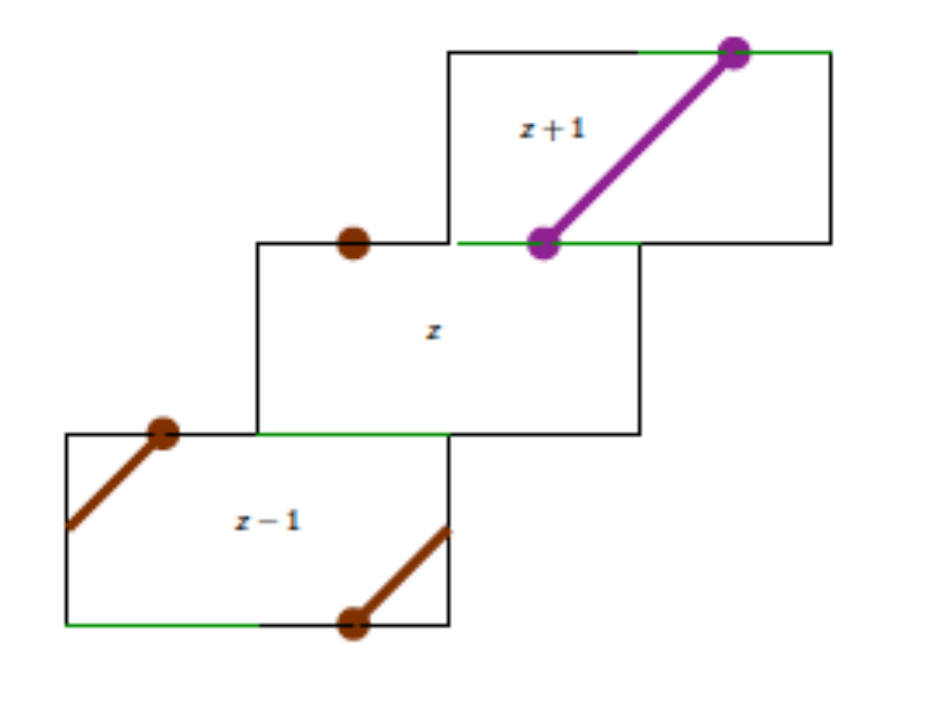}}
    \caption{Poincare map for a linear flow on the staircase. Orbits starting from $[1/2,1 ]$ go up while 
    orbits starting from $[0, 1/2)$ have to go down due to the gluing conditions.}
    \label{StairPoincare}
\end{figure}

Consider the linear flow on $St$ with slope $\theta$ which is locally given by $T^t(x,y)=(x+t\cos\theta, y+t\sin \theta).$
Let $\Pi$ be the union of the top sides of the rectangles in $St.$ We identify $\Pi$ with $\T\times \Z$ using the map
$\eta:\T\times\Z\to \Pi$ such that $\eta(x,z)$ is the point on the top side of rectangle $z$ at the distance  
$2x$ from the left corner. It is easy to check (see Figure~\ref{StairPoincare})
that under this identification the Poincare map for $T^t$  takes form
$$ (x, z)=(x+\alpha, z+\chi_{[1/2, 1]}(x)-\chi_{[0, 1/2)}(z)) \text{ where }
\alpha=\frac{\tan \theta+1}{2}. $$

Now suppose that $\alpha$ and hence $\tan \theta$ is a quadratic surd.
By Lagrange theorem there is $A\in SL_2(\integers)$ such that
$ A\left(\begin{array}{c}\cos \theta \\ \sin\theta \end{array}\right)=
\lambda \left(\begin{array}{c}\cos \theta \\ \sin\theta \end{array}\right).$
By replacing $A$ by $A^k$ for a suitable 
(positive or negative) $k$
we may assume that $A\equiv I \mod 2$ and that $\lambda<1.$ 
Let $\Gamma_{N}(x)$ be the ray starting from $\eta(x, 0)$  having slope $\theta$ and length $\frac{N}{\sin\theta}.$ 
$$ \LL_N= \frac{\mes(\brp\in \Gamma_N(x) : z(\brp)=0)}
{\length(\Gamma_N(x)}=
\frac{\mes(\brq\in \tGamma(x): z(\phi_A^{-m})=0)}{\length(\tGamma)}$$
$$=\Prob_x(z(\phi_A^{-m} \brq)=0)$$
where $\tilde\Gamma=\phi_A^m \Gamma_N(x)$ and
$\Prob_x$ is computed under the assumption that $\brq$ is uniformly distributed on $ \tGamma.$
Choose $m$ to be the smallest number such that 
$\length(\phi^m_A \Gamma_N(x))=\lambda^m \frac{N}{\sin\theta} \leq 1.$
Note that $m\approx \frac{\ln N}{\ln \lambda}.$
By our choice of $m,$ $\tGamma$ is either contained in a single rectangle or intersects two of them.
Let us consider the first case, the second one is similar. So we assume that $\tGamma$ is in the
rectangle with index $a$ so that $\brq=(q,a).$
Due to \eqref{ZAuto}
$ z(\phi^{-m} \brq)=a-\sum_{j=1}^{m} \tau(\phi_A^{-j} q).$
Thus
$$\Prob_x\left(z(\phi_A^{-m} \brq)=0\right)=\Prob_x\left(\sum_{j=1}^m \tau(\phi_A^{-j}q=a)\right). $$
Now we apply  the Local Limit Theorem for linear toral automorphisms 
(see \cite[Section 4]{PP} or \cite{G-Bern}) which says that there is a constant $\sigma^2$
$$\Prob_x\left(\sum_{j=1}^n \tau(\phi_A^{-j}q)\right)\approx \frac{1}{\sqrt{2\pi m} \sigma} e^{-a^2/2\sigma^2m} .$$

It remains to note that
$$a(x)=\sum_{j=0}^{m-1} \tau(\phi_A^j \eta(x, 0)) $$
so applying the Central Limit Theorem for linear toral automorphisms we see that
if $x$ is uniformly distributed on $\T^1$ then
$\frac{a(x)}{\sqrt{m}}$ is approximately normal with zero mean and variance $\sigma^2.$
\end{proof}

Next we discuss the proof of Theorem \ref{ThNRandom}(b) in case $l=\frac{1}{2}.$ The proof proceeds the same way as
the proof of Theorem \ref{ThChi2} with the following changes.

(I) Instead of estimating the probability that $z(\phi^{-m}_A\brq)=0$ we need to estimate  the probability that
$z(\phi^{-m}_A \brq)$ belongs to an interval of length $\sqrt{m}$ so we use the Central Limit Theorem instead 
of the Local Limit Theorem. 

(II) Instead of taking $x$ random we take $x$ fixed at the origin. Note that the origin is fixed by $A$
so $\tau(\phi^m_A (0,0))=C m.$ (More precisely $\tau$ is multivalued at the origin since it belong to several 
rectangles so by $\tau(\phi^m_A(0,0))$ we
mean the limit of $\tau(\phi^m_A(\brp))$ as $\brp$ approaches the origin inside $\Gamma_N(0).$)

\section{Higher dimensional actions}
\label{ScHD}
\begin{question}
Generalize the results presented in Sections \ref{ScErgSum}-\ref{ScFlows}  to higher dimensional actions. 
\end{question} 

The orbits of commuting shifts $T^n x=x+\sum_{j=1}^q n_j \alpha_j$ are much less studied than their one-dimensional counterparts.
We expect that some of the results of Sections \ref{ScErgSum}-\ref{ScFlows} admit straightforward extensions while in other cases significant new
ideas will be necessary. Below we discuss two areas of research where multidimensional actions appear naturally.

\subsection{Linear forms.} \label{sec.linearforms}
Statements about orbits of a single translation can be interpreted as results about joint distribution of
fractional part of inhomogenuous linear forms of one variable evaluated over $\Z$. From the point of view of Number Theory it is
natural to study linear forms of several variables  evaluated over $\Z^d$. Let 
$$l_i(n)=x_i+\sum_{j=1}^q \alpha_{ij} n_j,\quad i=1\dots d.$$
Thus it is of interest to study the discrepancy
\begin{multline*} \mathbb{D}_N(\Omega, \a,x)=\\
\Card(0\leq n_j<N, j=1,\ldots,q:
(\{l_1(n)\},\dots \{l_d(n)\})\in \Omega)-N^q \text{Vol}(\Omega). \end{multline*} 
The latter problem is a classical subject in Number Theory, and there are several important results related to it.
In particular, the Poisson regime is well understood (\cite{M-ETDS}). The following result generalizes Theorem
\ref{ThDO-Pois} 
and can be proven by a similar argument.

\begin{theorem}
\label{ThZqPois}
Let $(\a,x)$ be uniformly distributed on $\T^{d(q+1)}.$ Then,  for any cube $\Sigma \subset \R^q$, the distribution of
$$ \Card(n: \frac{n}{N}\in \Sigma \text{ and }
(\{l_1(n)\},\dots \{l_d(n)\})\in N^{-q/d} \Omega) $$
converges as $N\to\infty$ to
$$\cN(\Omega, \Sigma):=\Card(e\in L, e=(x,y) : x(e)\in \Omega, y(e)\in \Sigma)$$
where $L$ is a random affine lattice in $\R^{d+q}.$
\end{theorem}

Thus the Poisson regime for the rotations exhibits more regular behavior comparing to standard Poisson processes.
However then the number of rotations becomes large the limiting distribution approaches the Poisson. Namely,
the following is the special case of the result proven in \cite{V}.

\begin{theorem}
\label{ThLargeQ}
If $\Sigma_q$ are unit cubes in $\R^q$ then 
$\Omega\to \cN(\Omega, \Sigma_q)$ converges as $q\to\infty$ to the Poisson measure
$ \bmu(\Omega)=\Card(\mathfrak{P}\cap \Omega)$ where $\mathfrak{P}$ is a Poisson process on
$\R^d$ with constant intensity.
\end{theorem}

Next we present extensions of Theorems \ref{CLTShrinkBad}, \ref{DFVsBC}, \ref{dfv1} and \ref{dfv2} to the context of homogeneous and inhomogeneous linear forms. Let again $l_i(n)=x_i+\sum_{j=1}^q \alpha_{ij} n_j,$ $i=1\dots d.$
Consider
$$V_N(\a,x , c)=\Card(0\leq n_i<N: 
(\{l_1(n)\}, \dots \{l_d(n)\})\in B(c |n|^{-q/d})). $$

More generally given a function $\psi:\reals^+\to\reals^+$ define
$$V^\psi_N(\a,x ) =\Card(0\leq n_i<N: 
(\{l_1(n)\}, \dots \{l_d(n)\})\in B(\psi(|q|)). $$

We also let $U_N(\a, c)=V_N(0, \a, c)$ and $U_N^\psi(\a)=V_N^\psi(0, \a)$
be the quantities measuring the rate of recurrence.

In particular we call the matrix $\a$ 
{\bf badly approximable} if there exists  $c>0$ such that for, $V_N(0,\a, c)$ is bounded. 
On the other hand,
if $U_N^\psi(\a) \to\infty$ where $\psi(r)=r^{-(d/q+\eps)}$ 
then $\a$ is called {\bf very well approximable  (VWA)}.

The following result is known as Khinchine--Groshev Theorem. Almost sure is considered relative to Lebesgue measure on the space of matrices $\a \in \T^{dq}$. 
\begin{theorem} \cite{Kh, Gr, DS, schmidt0, BV, BBV}
(a) If $\sum_ {\Z^q} | \psi^d(|n|)<\infty$ then $U_N^\psi$ is bounded almost surely.

(b) If $\sum_{\Z^q}  \psi^d(|n|)=+\infty$ and either $\psi$ is decreasing or $dq>1$
then $\lim_{N\to\infty} U_N^\psi(\a)=+\infty$ almost surely.

(c) For $d=q=1$ there exists $\psi$ such that $\sum_{n\in Z} \psi(|n|)=+\infty$ but
$U_N^\psi$ is bounded almost surely.

(d) If $\psi$ is decreasing and $\sum_{n\in Z^d} \psi^d(|n|)=+\infty$ then
$V_N^\psi(\a, x)\to\infty$ almost surely.
\end{theorem}

In particular, both badly approximable and very well approximable $\a$s have zero measure.

When the number of hits is infinite, it is natural to consider the  question of the sBC property.

\begin{theorem} \label{th.schmidt0} \cite{schmidt0} (a) For almost all $\a$
$$ U_N^\psi(\a)=\EXP(U_N^\psi)+O\left(\sqrt{\Gamma(N) \ln^3 \Gamma(N)}\right) $$
where
$$\Gamma(N)=\sum_{|n|\leq N} \psi(|n|)^d D(gcd(n_1\dots n_q))$$
and $D$ denotes the number of divisors.

(b) $\Gamma(N)\leq C \EXP(U_N^\psi)$ if either $q>3$ or $q=2$ and $n\psi^2(n)$ is decreasing.

(c) If $q=1$ and $\psi(n)$ is decreasing then for each $\delta$
$$ U_N^\psi(\a)=\EXP(U_N^\psi)+O\left(\sqrt{\tGamma(N) \EXP(U_N^\psi)} \ln^{2+\delta} (\EXP(U_N^\psi))\right) $$
where
$$ \tGamma(N)=\sum_{n=1}^N \frac{\psi(n)}{n}. $$  
\end{theorem}

%In particular, for $\psi(r)=r^{-(d/q)}$ the sBC property is satisfied, namely,  $\frac{U_N^\psi(\a)}{\EXP(U_N^\psi(\a))}\to 1.$ Theorem \ref{DFVsBC} follows as a particular case when $d=1$. 

\begin{question} %\margem{have to check the paper of schmidt} 
Does a similar formula as that of Theorem \ref{th.schmidt0} hold for $V^\psi$? 
\end{question}
Some partial results are obtained in \cite{schmidt}.

It follows from the same arguments as the proof of Theorem~\ref{DFVsBC} 
sketched  in Section \ref{tsp.proofs} that the sBC property holds for $\psi(r)=r^{-(d/q)}$ for almost every  $(\a,x)$, that is  %\margem{check if we cite schmidt or DFV here} 
$$\lim_{N \to \infty}  \frac{V_N(\a,x)}{\EXP(V_N(\a,x))}=1.$$

%Similarly to Theorem \ref{DFVsBC} we also have typically the sBC property for the sequence $\psi(r)=r^{-d/q}$ as follows. 

%\margem{this also should be  known...}
%\begin{theorem} (\cite{DFV})Ê\label{sBCgeneral} For almost every $\a \in \T^{dq}$, we have that 
 %$$\lim_{N \to \infty}  \frac{U_N}{\EXP(U_N)}=1$$ 
%and for almost every  $(\a,x)$ 
%$$\lim_{N \to \infty}  \frac{V_N(\a,x)}{\EXP(V_N(\a,x))}=1$$ 
%\end{theorem}

For badly approximable $\a$ we have the following. 
\begin{theorem}
\cite{M-ST}  Let  $x$ be uniformly distributed on $\T^d.$ %\margem{ in (b) is it $\liminf$?}
If $\a$ is badly approximable, there exists a constant $K$ such that all limit points of
$\frac{V_N-\EXP(V_N)}{\sqrt{\ln N}}$  are normal random variables with zero mean and variance
$\sigma^2$ where $0\leq \sigma^2\leq K.$
\end{theorem}

\begin{question} 
\label{QSTAE} 

(a) Show that there exist a constant $\bar{\sigma}^2>0$ such that for almost all $\a$   
$\dfrac{V_N-\EXP(V_N)}{\sqrt{\ln N}}$ converges to  $\sN(\bar{\sigma}^2).$

(b) Does there exist $\a$ such that $\lim\inf_{N\to\infty} \dfrac{V_N-\EXP(V_N)}{\sqrt{\ln N}}=0$ (that is, $\sqrt{\ln N}$ is not a correct normalization
for such $\a$)?
\end{question}

For random $\a$ we have the following. 
\begin{theorem} (\cite{DFV})Ê \label{DFVgeneral} There exists $\sigma$ such that 
If  $\a_1,\ldots, \a_r$ and $x_1,\ldots,x_d$ are randomly distributed on $\T^{dr+d}$ then 
$\frac{V_N-\EXP(V_N)}{\sqrt{\ln N}}$  converges in distribution to a normal random variables with zero mean and variance
$\sigma^2$. A similar convergence holds if $d+r>2$, $(x_1,\ldots,x_r)=(0,\ldots,0)$ and only the $\a_i$'s are random. 
\end{theorem}

Still there are many open questions. We provide several examples.
\begin{question}
\label{QKestenLF}
Extend Theorems  \ref{ThConvex} and \ref{ThBox} to the case $q>1.$
\end{question}
We note that in the case of Theorem \ref{ThBox}, even the case $d=1$ seems quite difficult.
One can attack this question using the method of \cite{DF2} but it runs into the problem of lack of parameters
described after Question \ref{QSubMani}.

\begin{question}
Let $l, \hl: \R^d\to \R, $ be linear forms with random coefficients, $Q:\R^d\to \R$ be a positive definite quadratic form.
Investigate limit theorems, after adequate renormalization, for the number of solutions to

(a) $\{l(n)\}Q(n) \leq c, |n|\leq N;$ 

(b) $\{l(n)\} |\hl(n)|\leq c, |n|\leq N;$

(c) $|l(n)Q(n)|\leq c, |n|\leq N;$ 

(d) $|l(n)\hl(n)|<c, |n|\leq N.$
\end{question}
While (a) and (b) have obvious interpretation as shrinking target problems for toral translations, such interpretation for
(c) and (d) is less straightforward. Consider for example (c). Let $l(n)=\sum_{j=1}^q \alpha_j n_j.$ Dividing the distribution
of $\alpha$ into thin slices we may assume that $\alpha_d$ is almost constant. If $\alpha_d\approx a$ then we can compare 
our problem with $|\left(\sum_{j=1}^{q-1} \talpha_j n_j\right)+n_q| Q(n)<\tc$ where $\talpha_j=\alpha_j/a, \tc=c/a.$
Since $|l(n)|$ should be small we must have 
$|\left(\sum_{j=1}^{q-1} \talpha_j n_j\right)+n_q|=\{\sum_{j=1}^q \talpha_j\}$ in which case
$ Q(n_1, \dots, n_{q-1}, n_q)$ is well approximated by
$$ Q(n_1, \dots, n_{q-1}, -\sum_{j=1}^{q-1} \talpha_j n_j) $$
so we have a shrinking target problem in lower dimensions. In fact as we saw in Section \ref{ScProofs} 
typically the proof proceeds in the
opposite direction by getting rid of fractional part at the expense of increasing dimension since problems (c) and (d) have more
symmetry and so should be easier to analyze.

We note that part (d) deals with degenerate quadratic form. The case of non-degenerate forms is discussed in
\cite[Sections 5 and 6]{E}.

\subsection{Cut-and-project sets.} 
\label{SSCutProject}
Cut-and-project sets are used in physics literature to model quasicrystals. To define them we need the following data:
a lattice in $\R^d,$ a decomposition $\R^d=E_1\oplus E_2$ and a compact set (a window) $\cW\subset E_2.$ Let
$P_1$ and $P_2$ be the projections to $E_1$ and $E_2$ respectively. The cut-and-project set is defined by
$$\cP=\{P_1(e), e\in L \text{ and } P_2(e)\in \cW\}. $$
We suppose in the following discussion that 
$$ \overline{E_1+L}=\R^d \text{ and } L\cap E_2=\emptyset. $$
Then $\cP$ is a discrete subset of $E_1$ sharing many properties of lattices but having a more complicated structure.
Note that the limiting distributions in Theorems \ref{ThDO-Pois}  and \ref{ThZqPois} are described in terms of cut-and-project sets.
We refer the reader to \cite[Sections 16 and 17]{M-Bern} for more discussion of cut-and-project set.
Here we only mention the fact that such sets have asymptotic density.
Let $\cP_R=\{t\in \cP: |t|\leq R\}.$

\begin{theorem}
\label{ThDensity}
Suppose that $\cW$ is an open subset of $E_2$ with a  piecewise smooth boundary. Then
$$ \lim_{R\to\infty} \frac{\Card(\cP_R)}{\Vol(B(0,R))}=
\frac{\Vol(\cW)}{\mathrm{covol}(L)} \frac{\Vol_{\R^d}}{\Vol_{E_1} \Vol_{E_2}}. $$
\end{theorem}

\begin{proof}
(Following \cite{Ho}). Note that $t\in \cP$ iff there exists $e\in L$ such that $-t+e\in \cW,$ that is
$-t\in \cW \text{ mod } L.$ Consider the action of $E_1$ on $\R^d/L$ given by $T^t(x)=x+t.$ Then 
$\cP_R$ counts the number of intersections of the orbit of the origin of size $R$ with $\cW.$
Pick a small $\delta$ and let 
$ \cW_\delta=\{\cW+t, |t|\leq\delta\}.$ Then $\cW_\delta$ is a subset of $\R^d/L$ and  for small $\delta$ 
\begin{equation}
\label{VolFat}
\Vol(\cW_\delta)=\Vol(\cW) \Vol(B(0,\delta))
\dfrac{\Vol_{\R^d}}{\Vol_{E_1} \Vol_{E_2}} . 
\end{equation} 
Next,
\begin{equation}
\label{CountCross}
\int_{|t|<R} \chi_{\cW_\delta}(T^t 0) dt=\Vol(B(0, \delta)) \Card(\cP_R)+O(R^{q-1})
\end{equation}
where $q=\dim(E_1)$ and the second term represents boundary contribution. On the other hand by 
unique ergodicity of 
$T^t$
\begin{equation}
\label{RdErg}
\int_{|t|<R} \chi_{\cW_\delta}(T^t 0) dt=\Vol(B(0, R)) \frac{\Vol(\cW_\delta)}{\mathrm{covol}(L)}+o(R^{q}).
\end{equation}
Combining \eqref{VolFat}, \eqref{CountCross} and \eqref{RdErg} we get the result.
\end{proof}

\begin{question}
Describe the error term in the asymptotics of Theorem~\ref{ThDensity}.
\end{question}

If $q=1$ then the error term in \eqref{CountCross} is negligible and so \eqref{CountCross} can be used to describe the
deviations (see \cite{DF1} for the case where $\cW$ is convex). If $q>1$ more work is needed to control both the LHS
and the RHS of \eqref{CountCross}.

While the methods of \cite{DF1, DF2} deal with the case where $\dim(E_1)$ is as small as possible, the most classical 
case is the opposite one when $\dim(E_2)$ is as large as possible, that is, studying lattice points in large regions.
Here we can not attempt to survey this enormous topic, so we refer the reader to the specialized literature on the
subject (\cite{Hux, IKKN, Kr}. We just mention that  limit theorems similar in spirit to the results discussed in this paper
are obtained in \cite{HB, BB, BCDL, Pet2}. More generally, instead of considering large balls one can count the number of lattice points in
$R D$ where $D$ is a fixed regular set. As in Section \ref{ScErgSum} the order of the error term is sensitive to the geometry of $D$
(see e.g. \cite{Bl, KN, Lev, MN,  Pet3, R1, R2, Sk, SS} and references wherein). 
In fact, one can also consider the varying shapes $R D_R$ which includes both the Poisson regime where
$\Vol(R D_R)$ does not grow (see \cite{BL} and references wherein) and the intermediate regime where 
$\Vol(R D_R)$ grows but
at the rate slower than $R^d$ (see \cite{HR, Wig}). 

This motivates the following question 

\begin{question}
Extend the results of the above mentioned papers to cut-and-project sets.
\end{question}

%Similarly  to the proof of Theorem \ref{dimd}, Question \ref{linearform} would involve the study of the point process  
%$$\left( (\ln N)^d  \Pi_i k  \| k \a_i \|, Nk\a_1 {\rm mod} (2), \ldots,  Nk\a_d {\rm mod} (2)  \right).$$

%\vskip5mm
%One can also investigate analogues of the Shrinking Targets Theorems \ref{dfv1}--\ref{dfv2} : for $(\a,x) \in \T^d \times \T^d$ study the distribution  of the number of hits 
%$$S_R(\a,x)=\sum_{|k|\leq R} \chi_{B_{r_k}(x)}(\|(k,\a)\|)$$ for the sequence $r_k=1/|k|^d$.

%\begin{question} Let $B_R = 2 \sum_{|k|\leq R} \frac{1}{|k|^d}$. There exists $\sigma>0$ such that 
%$$(S_R-B_R)/A_R \longrightarrow N(0,\sigma).$$
%\end{question}

%We note that some of the results discussed in this note have already been extended to the higher dimensional setting. For example, in \cite{M-ETDS} it is shown that the $n$-point correlations between the values of the sequence $\{m_1 \a_1+\ldots+m_d \a_d (1)\}_{m_j \in \Z}$  converges in distribution if $(\a_1,\ldots,\a_d)$ is a random vector in $\T^d$.  {\color{Orange} more details. I could not find the paper \cite{M-ST}}

%\begin{figure}
%\resizebox{!}{15cm}{\includegraphics{penat.pdf}}
%\end{figure}

\end{document}